\documentclass[reqno]{amsart}            %,draft
\usepackage{amsthm,amssymb,amsfonts,graphicx,cite,color}
\usepackage[bookmarks,bookmarksnumbered,bookmarksopen,colorlinks,backref,linkcolor=blue,citecolor=red]{hyperref}%
\usepackage{xcolor}
\usepackage{graphicx}
\usepackage{mathrsfs}

\newtheorem{theorem}{Theorem}[section]

\newtheorem{lemma}{Lemma}[section]
\newtheorem{corollary}{Corollary}[section]
\newtheorem{remark}{Remark}[section]
\newtheorem{proposition}{Proposition}[section]
\numberwithin{equation}{section}

\begin{document}

\title[Asymptotic approximations for semilinear problems in thin networks]
{Asymptotic approximations for semilinear parabolic convection-dominated transport problems\\ in thin graph-like networks}
\author[Taras Mel'nyk \& Christian Rohde]{ Taras Mel'nyk$^\flat$ \ \& \ Christian Rohde$^\natural$}
\address{\hskip-12pt  $^\flat$ Department of Mathematical Physics, Faculty of Mathematics and Mechanics\\
Taras Shevchenko National University of Kyiv\\
Volodymyrska str. 64,\ 01601 Kyiv,  \ Ukraine
}
\email{melnyk@imath.kiev.ua}

\address{\hskip-12pt  $^\natural$ Institute of Applied Analysis and Numerical Simulation,
Faculty of Mathematics and Physics, Suttgart University\\
Pfaffenwaldring 57,\ 70569 Suttgart,  \ Germany
}
\email{christian.rohde@mathematik.uni-stuttgart.de }

\begin{abstract} \vskip-10pt
We consider time-dependent  convection-diffusion problems with high P\'eclet number  of  order $\mathcal{O}(\varepsilon^{-1})$
in thin three-dimensional graph-like networks consisting of   cylinders that are  interconnected by small domains (nodes) with diameters of order  $\mathcal{O}(\varepsilon).$
On  the lateral surfaces of the thin cylinders and the boundaries  of the nodes we account for solution-dependent inhomogeneous  Robin  boundary conditions which can render the  associated initial-boundary problem to be nonlinear.  The strength of the inhomogeneity  is controlled by an intensity factor of order ${\mathcal{O}} (\varepsilon^\alpha)$, $\alpha >0$.

The asymptotic behaviour of  the solution is studied as $\varepsilon \to 0,$ i.e., when  the diffusion coefficients are eliminated and the thin three-diemnsional network is shrunk into a  graph. There are three qualitatively different cases in the asymptotic behaviour of the solution depending on the value of the intensity parameter $\alpha:$ $\alpha =1,$ $\alpha > 1,$ and $\alpha \in (0, 1).$
We construct the asymptotic approximation of the solution, which provides us with the  hyperbolic limit model for $\varepsilon \to 0$ for the first two cases, and prove the corresponding uniform pointwise estimates and energy estimates. As the main result, we derive  uniform pointwise estimates for the  difference between the solutions of the convection-diffusion problem and
the zero-order approximation that includes the solution of the corresponding hyperbolic limit problem.
\end{abstract}

\keywords{Asymptotic approximation, convection-diffusion problem, asymptotic estimate, thin graph-like network, hyperbolic limit model
\\
\hspace*{9pt} {\it MOS subject classification:} \   35K20,  35R02,  35B40, 35B25, 	35B45, 35K57, 35Q49
}

\maketitle
\tableofcontents
% ----------------------------------------------------------------

\section{Introduction}\label{Sect1}
Parabolic (non-steady-state) convection-diffusion equations  arise as basic mathematical   models for  various  transport processes including  mass transfer  and heat transfer   through possibly inhomogeneous or porous media.  We are interested in regimes  with a high P\'eclet number (the ratio of  convective to  diffusive transport rates)  which is controlled by  a small parameter $\varepsilon>0$ scaling the diffusion operator in the rescaled differential equation.  The limit $\varepsilon \to 0 $  leads to a hyperbolic limit model, i.e.,  a change of type of the mathematical model appears. This singular behaviour complicates the study, especially the derivation of robust  a priori estimates  that allow estimating the difference between the solutions of the parabolic approximation and the hyperbolic-limit solution.
We  have overcome these difficulties for the transport dynamics on networks in our paper \cite{Mel-Roh_preprint-2022}, where a linear parabolic convection-dominated  problem was studied in a thin graph-like domain.

Here we generalize the approach to a selected nonlinear setting.  Heat and mass transfer processes in  biochemistry and medicine indicate the need for a broad study of various types  of initial-boundary value problems involving nonlinear boundary conditions, e.g., \cite{Conca,Mav_2023,Mel_IJB-2019,Murray-2,Pao}, especially the paper \cite{Mav_2023},
where many applied problems with nonlinear boundary conditions are described in detail.
Thus, denoting  the unknown concentration of some transported species by $u_\varepsilon$, we study  in the present work the asymptotic behavior (as $\varepsilon \to 0)$ of  $u_\varepsilon$  that satisfies a linear convection-diffusion equation  subject to the influence of semilinear boundary conditions
$$
\varepsilon  \partial_{\boldsymbol{\nu}_\varepsilon} u_\varepsilon  -  u_\varepsilon \, \overrightarrow{V_\varepsilon}(x,t)\boldsymbol{\cdot}\boldsymbol{\nu}_\varepsilon   =  \varepsilon^{\alpha} \varphi^{(i)}_\varepsilon\big(u_\varepsilon,x,t\big).
$$
The boundary conditions are imposed on  both the lateral surfaces of the thin cylinders and the boundary of nodes, which together
form  a thin three-dimensional graph-like network (see Fig.~\ref{f1}).
$\overrightarrow{V_\varepsilon}$ is a given convective vector field.  To quantify  the intensity of processes at the boundary governed by the functions $\{\varphi^{(i)}\}_{i=1}^M,$
 an intensity factor $\varepsilon^\alpha,$ where $\alpha$ is a positive parameter, is introduced.
\begin{figure}[htbp]
\centering
\includegraphics[width=8cm]{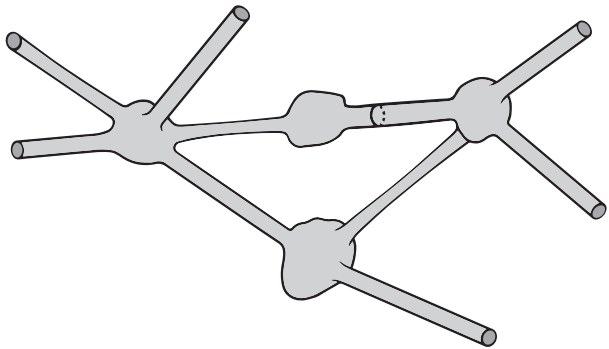}
\caption{Network of thin cylinders connected by  nodes of arbitrary geometry.}\label{f1}
\end{figure}

Network domains appear in varied engineering structures, as well as natural objects, e.g., cracked subsurface systems, plant roots or neural networks. Therefore, investigations of different processes in thin networks are  important  for numerous fields of natural sciences.   At present, special interest of researchers is focused on various effects observed in vicinities of local irregularities of the geometry. For example, aneurysms and stenosis in thin pipe junctions  (see \cite{Bor,Mardal}); another related scenario occurs for thin  debris-filled wellbores or when  modelling effective surface roughness in any kind of tubes and channels with varying aperture  \cite{BurbullaHoerlRohde22,BurbullaRohde22}.

As noted in \cite{Mardal} and \cite{Stynes-2018,Joh_Kno_Nov_2018}, numerical simulations of both boundary-value problems in thin pipe junctions and convection-dominated problems do not provide an acceptable level of accuracy. Therefore, asymptotic methods are
important to derive models especially in thin graph-like networks that ensure accuracy while saving computational resources. Reviews of different asymptotic methods  for various problems are given in \cite{Mel_Klev_M2AS-2018}.

\smallskip

The main goal of this paper is to develop an asymptotic approach that provides efficient and reliable modeling
for parabolic semilinear convection-dominated transport problems in thin $3D$ networks.
In particular, we are interested in finding the corresponding limit problem on a graph to which a thin network is shrinking as $\varepsilon $ tends to zero and the diffusion coefficients disappear.
As a result, one expects  that the limit concentration satisfies  first-order hyperbolic differential equations on the edges and some gluing conditions at the vertexes of this graph. It should be also expected that the limit problem will somehow depend on the parameter~$\alpha,$ and we need to find this dependence, as well as prove the correctness of this limit problem.

Since $1D$ transport modeling on graphs provides an incomplete understanding of multidimensional real processes \cite{Gar_Piccoli_2009,Mel-Roh_preprint-2022}, especially at the graph's vertices,  our approach
relies on an asymptotic mixed-dimensional approximation of  the solutions.
Three qualitatively different cases for the  asymptotic behaviour of the solutions are found,  depending on the value of the parameter $\alpha:$ $\alpha =1,$ $\alpha > 1,$ and $\alpha \in (0,1)$.
After a detailed description of the assumptions and the problem statement in Section~\ref{Sec:Statement}, we  present
the construction of the mixed-dimensional approximation for the parameter
$\alpha=1$  in Section \ref{Sec:expansions}, see in particular  \S~\ref{Sec:justification}.
This approximation depends on the ratio of cylinders which act either as inlet or outlet for the flow.
The leading coefficients in the approximations provide the limit of the sequence of solutions $\{u_\varepsilon\}_{\varepsilon >0}$ of the three-dimensional problem if  the  parameter $\varepsilon$ tends to $0$. The approximation results  are proven in \S~\ref{A priori estimates} (see Theorem \ref{Th_1}, Corollary \ref{corol_1} on the maximum norm of the error  and Theorem \ref{apriory_estimate} on first derivatives of the error).  The  asymptotic behaviour of the solutions  for $\alpha >1$ is then studied
 in Section~\ref{Sec:expansions+}.
In the conclusions,  the obtained results are analyzed and a
perspective for the third case $\alpha \in (0, 1)$  is given.

%%%%%%%%%%%%%%%%%%%
\section{Problem statement for the three-dimensional case}\label{Sec:Statement}
In this section we first describe the geometry of the network.  To understand the core of the problem it is sufficient to consider the case of a thin network with $M$ finite cylinders that are connected by a single node.  Such a simple network will be called a thin graph-like junction. In \S~\ref{explanation}, we show how to apply the results obtained to  a general network.
After a description of the convective flux which determines inlet and outlet cylinders, we present
the complete initial-boundary value problem on the thin graph-like junction  in \S~\ref{sec_ibvp}.

\subsection{Domain description}
In the three-dimensional Euclidean space $\mathbb R^3$ we consider the unit vectors
$\mathbf{e}^{(i)} = \big(e_1^{(i)},\,e_2^{(i)},\,e_3^{(i)}\big),$ $i\in \{1,\ldots,M\},$ where $M\in\mathbb{N}, \, M\ge2.$
For each index $i\in \{1,\ldots,M\},$ we choose a unit  vector
$\boldsymbol{\varsigma}^{(i)} = \big(\varsigma_1^{(i)},\,\varsigma_2^{(i)},\,\varsigma_3^{(i)}\big)$ that is
orthogonal to $\bf{e}^{(i)}$ and then consider their vector product
$$
\boldsymbol{\zeta}^{(i)} = \mathbf{e}^{(i)} \times  \boldsymbol{\varsigma}^{(i)} = \big(e_2^{(i)}\varsigma_3^{(i)}-e_3^{(i)}\varsigma_2^{(i)}, \
    e_3^{(i)}\varsigma_1^{(i)}-e_1^{(i)}\varsigma_3^{(i)}, \
    e_1^{(i)}\varsigma_2^{(i)}-e_2^{(i)}\varsigma_1^{(i)} \big).
$$
Thus,  every point $x=(x_1, x_2, x_3) \in \mathbb{R}^3$ has coordinates
$y^{(i)} = \big(y_1^{(i)}, y_2^{(i)}, y_3^{(i)}\big)$ in a new
right-handed coordinate system $\big(\mathbf{e}^{(i)},  \boldsymbol{\varsigma}^{(i)}, \boldsymbol{\zeta}^{(i)}\big),$
which are connected with the initial one by relation
$y^{(i)} = \mathbb{A}_i  x$ (here and below, we usually mean a column vector in the  product of a matrix and a vector and its result),  where the matrix  $\mathbb{A}_i$ is given by
$$ \mathbb{A}_i =
\begin{pmatrix}
    e_1^{(i)} &
    e_2^{(i)} &
    e_3^{(i)} \\[3pt]
    \varsigma_1^{(i)} &
    \varsigma_2^{(i)} &
    \varsigma_3^{(i)} \\[3pt]
    e_2^{(i)}\varsigma_3^{(i)}-e_3^{(i)}\varsigma_2^{(i)} &
    e_3^{(i)}\varsigma_1^{(i)}-e_1^{(i)}\varsigma_3^{(i)} &
    e_1^{(i)}\varsigma_2^{(i)}-e_2^{(i)}\varsigma_1^{(i)}
\end{pmatrix}.
$$

It is easy to verify that
$\mathbb{A}_i^{\,-1} = \mathbb{A}_i^\mathrm{\,T}$ $(T$ is the sign of transposition) and $\det\mathbb{A}_i = 1$ hold.
%It is known that
The main differential operators (Laplacian, divergence, gradient) are invariant under all Euclidean transformations (rotations and translations). Hence,
\begin{equation}\label{lap_1}
      \Delta_x u(x) := \frac{\partial^2 u}{\partial x_1^2}
  + \frac{\partial^2u}{\partial x_2^2}
  + \frac{\partial^2u}{\partial x_3^2}
  = \frac{\partial^2u^{(i)}}{\partial (y_1^{(i)})^2}
  + \frac{\partial^2u^{(i)}}{\partial (y_2^{(i)})^2}
  + \frac{\partial^2u^{(i)}}{\partial (y_3^{(i)})^2}
  =: \Delta_{y^{(i)}}u^{(i)}(y^{(i)}),
\end{equation}
\begin{equation}\label{div_1}
\mathrm{div}_x \vec{\mathcal{V}}(x):= \frac{\partial v_1}{\partial x_1} + \frac{\partial v_2}{\partial x_2} + \frac{\partial v_3}{\partial x_3}
= \frac{\partial v^{(i)}_1}{\partial y_1^{(i)}} + \frac{\partial v^{(i)}_2}{\partial y_2^{(i)}} + \frac{\partial v^{(i)}_3}{\partial y_3^{(i)}} =:
\mathrm{div}_{y^{(i)}} \vec{\mathcal{V}}^{(i)}({y^{(i)}}),
\end{equation}
\begin{align} \label{grad_1}
  \big(\boldsymbol{\eta},\,\nabla_x u(x)\big) = & \ \eta_1\frac{\partial u}{\partial x_1}
  + \eta_2\frac{\partial u}{\partial x_2}
  + \eta_3\frac{\partial u}{\partial x_3} \notag
  \\
  = & \  \eta_1^{(i)}\frac{\partial u^{(i)}}{\partial y_1^{(i)}}
  + \eta_2^{(i)}\frac{\partial u^{(i)}}{\partial y_2^{(i)}}
  + \eta_3^{(i)}\frac{\partial u^{(i)}}{\partial y_3^{(i)}}
  = \big(\boldsymbol{\eta}^{(i)},\,\nabla_{y^{(i)}} u^{(i)}(y^{(i)})\big),
\end{align}
where $\boldsymbol{\eta}^{(i)}= \big(\eta_1^{(i)},\,\eta_2^{(i)},\,\eta_3^{(i)}\big) =
\mathbb{A}_i \, \boldsymbol{\eta},$  $\boldsymbol{\eta}=(\eta_1,\,\eta_2,\,\eta_3),$
$u^{(i)} \big(y^{(i)}\big) = u\big(\mathbb{A}_i^{-1}y^{(i)}\big),$ and
$\vec{\mathcal{V}}^{(i)}\big(y^{(i)}\big) = \big(v_1^{(i)}(y^{(i)}),\, v_2^{(i)}(y^{(i)}),\, v_3^{(i)}(y^{(i)})\big) = \mathbb{A}_i \vec{\mathcal{V}}\big(\mathbb{A}_i^{-1}y^{(i)}\big).$

\begin{remark}
In what follows, we will denote a domain $Q$ in the new coordinates by the same symbol $Q$, and we will use both new and old coordinates to describe it.
\end{remark}
In each coordinate system $\big(\mathbf{e}^{(i)},  \boldsymbol{\varsigma}^{(i)}, \boldsymbol{\zeta}^{(i)}\big),$ we define the thin cylinder
$$
\Omega_\varepsilon^{(i)} =
    \big\{   y^{(i)} \in\mathbb{R}^3 \colon \
    \varepsilon \ell_0<y_1^{(i)}<\ell_i, \
    \big(y_2^{(i)}\big)^2 + \big(y_3^{(i)}\big)^2 < \varepsilon^2 h_i^2\big\},
$$
 where $\varepsilon$ is a small positive parameter,  $h_i$ is a positive constant, and  $\ell_0\in(0, \frac13), \ \ell_i\geq1$ are  given numbers.
 We denote its  lateral surface  by
$$
{\Gamma_\varepsilon^{(i)}} := \partial\Omega_\varepsilon^{(i)} \cap \{ y^{(i)}\colon   \varepsilon \ell_0<y_1^{(i)} <\ell_i \}
$$
and by
$$
\Upsilon_\varepsilon^{(i)}(\mu) := \Omega_\varepsilon^{(i)} \cap
\big\{  y^{(i)}\colon  \ y_1^{(i)} = \mu \big\}
$$
its cross-section at  $\mu \in [ \varepsilon \ell_0, \ell_i].$
\begin{figure}[htbp]
\begin{center}
\includegraphics[width=5cm]{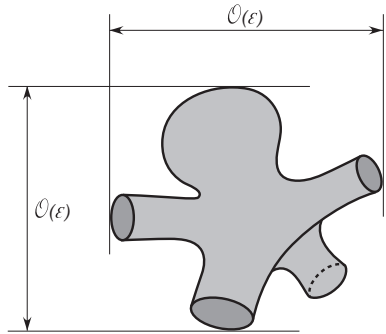}
\caption{The node $\Omega_\varepsilon^{(0)}$.}\label{f2}
\end{center}
\end{figure}

The thin cylinders are joined through a domain $\Omega_\varepsilon^{(0)}$ (referred to  as the "node",  see Fig.~\ref{f2}) that is formed by the homothetic transformation with coefficient $\varepsilon$ from a bounded domain $\Xi^{(0)}$ containing the origin,  i.e.,
$
\Omega_\varepsilon^{(0)} = \varepsilon\, \Xi^{(0)}.
$
In addition, we assume that the boundary $\partial  \Xi^{(0)}$ of $\Xi^{(0)}=\Omega_1^{(0)}$ contains the disks
$\Upsilon^{(i)}_1(\ell_0) := \overline{\Xi^{(0)}} \cap
\big\{ x\colon y_1^{(i)}= \ell_0\big\}, \   i\in\{1,\ldots,M\},$ that are the bases of some right cylinders, respectively, and the lateral surfaces of these cylinders belong to $\partial  \Xi^{(0)}.$ Thus, the boundary of the node $\Omega_\varepsilon^{(0)}$ consists of
the disks $\Upsilon_\varepsilon^{(i)} (\varepsilon\ell_0), \   i\in\{1,\ldots,M\},$
and the surface
$$
\Gamma_\varepsilon^{(0)} :=
\partial\Omega_\varepsilon^{(0)} \backslash
\left\{\overline{\Upsilon_\varepsilon^{(i)}(\varepsilon \ell_0)}\right\}_{i=1}^M .
$$

Hence,  the model thin graph-like junction  $\Omega_\varepsilon$  (see Fig.~\ref{f3})
is   the interior of the union
$\bigcup_{i=0}^{M}\overline{\Omega_\varepsilon^{(i)}}.$
We also assume that  the thin cylinders $\{\Omega_\varepsilon^{(i)}\}_{i=1}^M$  do not intersect, and that the surface $\partial\Omega_\varepsilon\setminus \bigcup_{i=1}^{M}\overline{\Upsilon_\varepsilon^{(i)} (\ell_i)}$ is smooth of class $C^3.$
\begin{figure}[htbp]
\begin{center}
\includegraphics[width=8cm]{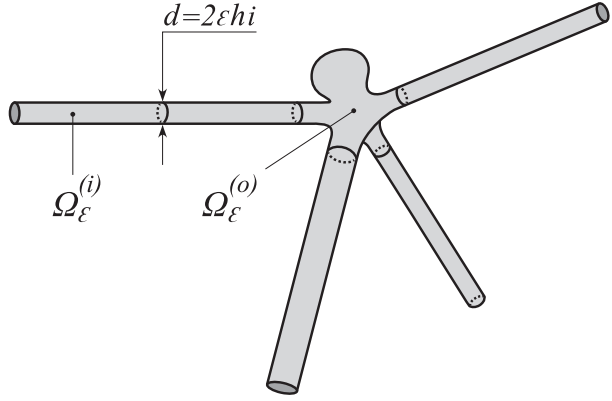}
\caption{Thin graph-like junction $\Omega_\varepsilon$ with cylinders $\Omega^{(1)}_\varepsilon, \ldots,  \Omega^{(M)}_\varepsilon$  and node $\Omega^{(0)}_\varepsilon$.}\label{f3}
\end{center}
\end{figure}
%%%%%%%%%%%%

\subsection{Description of the convective vector field in the cylinders and the node}\label{subsec_V}

All components of the vector-valued function $\overrightarrow{V_1}$ $(\varepsilon =1)$
belong   to the space $C^{3}\big(\overline{\mho}_1 \times [0,T]\big),$
where   $\mho_1$ is a domain that includes the junction $\Omega_\varepsilon$ for all $\varepsilon \in (0, 1].$
The structure of $\overrightarrow{V_\varepsilon}$ depends on  parts of the thin junction $\Omega_\varepsilon,$ namely,
$$
\overrightarrow{V_\varepsilon}(x,t)=
\left(v^{(0)}_1\big(\tfrac{x}{\varepsilon},t\big),  \  v^{(0)}_2\big(\tfrac{x}{\varepsilon},t), \  v^{(0)}_3(\tfrac{x}{\varepsilon},t\big)
\right) =: \overrightarrow{V_\varepsilon}^{(0)}(x,t), \qquad (x,t) \in \ \Omega_\varepsilon^{(0)}\times [0, T],
$$
and for each $ i\in \{1,\ldots,M\}$ and all $(y^{(i)},t) \in \Omega_\varepsilon^{(i)}\times [0,T]$
$$
\overrightarrow{{V}_\varepsilon}^{(i)}(y^{(i)},t) :=
\mathbb{A}_i \overrightarrow{V_\varepsilon}(x,t)\big|_{x= \mathbb{A}_i^{-1} y^{(i)}} =
\left({v}^{(i)}_1\big(y_1^{(i)},t\big), \ \varepsilon\, \overline{V}_\varepsilon^{(i)}(y^{(i)},t)\right),
$$
 where
 \begin{equation}\label{V_i}
  \overline{V}_\varepsilon^{(i)}(y^{(i)},t)  := \Big( {v}^{(i)}_2\big(y_1^{(i)}, \tfrac{\overline{y}_1^{(i)}}{\varepsilon},t\big), \
        {v}^{(i)}_3\big(y_1^{(i)}, \tfrac{\overline{y}_1^{(i)}}{\varepsilon},t\big) \Big)
 \end{equation}
and $\overline{y}_1^{(i)} = \big(y_2^{(i)},\, y_3^{(i)}\big)$. The main direction of the vector function $\overrightarrow{V_\varepsilon}^{(i)}$ is oriented  along the axis of the thin cylinder $\Omega_\varepsilon^{(i)}.$

In addition, for each $i\in\{1,\ldots,M\}$ the function ${v}^{(i)}_1\big(y_1^{(i)},t\big), \ (y_1^{(i)},t)\in [0, \ell_i]\times [0,T],$  is equal to a function $\mathrm{v}_i(t), \ t\in [0,T],$ in a narrow strip  $[0, \delta_i]\times [0,T]$ $(\delta_i < \ell_i)$
and the  components of $\overline{V_\varepsilon}^{(i)}$
have compact supports with respect to the  longitudinal  variable $y_1^{(i)},$ in particular, we will assume that they vanish in  $[0, \delta_i]\times [0,T].$
This means that $\overrightarrow{{V}_\varepsilon}^{(i)} = \mathrm{v}_i\, \mathbf{e}^{(i)}$ in the coordinates $x$ in a neighbourhood of the base $\Upsilon_\varepsilon^{(i)}(\varepsilon \ell_0)$ of the cylinder $\Omega_\varepsilon^{(i)}$ for $\varepsilon$ small enough. Thanks to the smoothness of $\overrightarrow{V_\varepsilon},$
\begin{equation}\label{V_1}
  \overrightarrow{V_\varepsilon}^{(0)}(x,t)\big|_{x \, \in \, \Upsilon_\varepsilon^{(i)}(\varepsilon \ell_0)} = \mathrm{v}_i(t)\, \mathbf{e}^{(i)},
  \quad t \in [0, T].
\end{equation}
We also suppose that for all $y_1^{(i)} \in [0, \ell_i]$ and $t \in [0,T]$
\begin{equation}\label{V_2}
{v}^{(i)}_1 > 0 \quad\text{for} \ \ i\in \{1,\ldots,m\}, \qquad {v}^{(i)}_1 < 0 \quad\text{for} \ \ i\in \{m+1,\ldots,M\},
\end{equation}
where $m\in \Bbb N$ and $m< M,$ i.e., we have $m$ inlet and $M-m$ outlet cylinders relative to the vector field~$\overrightarrow{V_\varepsilon}.$

As for the vector field at the node, we will make the following additional assumption. It  is conservative  in the node and
its potential $p$ is a solution to the boundary value problem
\begin{equation}\label{potential0}
\left\{\begin{array}{rcll}
    \Delta_\xi p(\xi,t)  & = & 0, & \quad
    \xi= (\xi_1, \xi_2, \xi_3) \in \Xi^{(0)},
\\[2pt]
\partial_{\boldsymbol{\nu}_\xi} p(\xi,t)   &=&  \mathrm{v}_i(t), & \quad
   \xi \in \Upsilon^{(i)}_1(\ell_0), \quad i\in \{1,\ldots, M\},
\\[2pt]
\partial_{\boldsymbol{\nu}_\xi} p(\xi,t)  &=&  0, & \quad
   \xi \in \Gamma^{(0)} := \partial{\Xi^{(0)}} \setminus \Big( \bigcup_{i=1}^M \Upsilon^{(i)}_1(\ell_0)\Big),
\end{array}\right.
\end{equation}
where $\Delta_\xi$ is  the Laplace operator,  $\partial_{\boldsymbol{\nu}_\xi}$ is the derivative along the  outward unit normal $\boldsymbol{\nu}_\xi$ to $\partial \Xi^{(0)},$  and the variable $t\in [0,T]$ is treated as a parameter.
The Neumann problem \eqref{potential0} has a  solution if and only if
the  condition
\begin{equation}\label{cond_1}
  \sum_{i=1}^{M} h_i^2 \, \mathrm{v}_i(t) = 0
\end{equation}
holds for each $t\in [0, T].$  To ensure  unique solvability of \eqref{potential0}  we  enforce the condition $\int_{\Xi^{(0)}} p\, d\xi =0.$
Thus,
\begin{equation}\label{field_0}
  \overrightarrow{V_\varepsilon}^{(0)}(x,t)  = \nabla_{\xi} p(\xi,t)\big|_{\xi = \frac{x}{\varepsilon}} = \varepsilon \nabla_x \big( p(\tfrac{x}{\varepsilon},t) \big), \quad x\in \Omega^{(0)}_\varepsilon,
\end{equation}
and, clearly,  $\mathrm{div}_x \overrightarrow{V_\varepsilon}^{(0)} = 0$ in $\Omega^{(0)}_\varepsilon \times [0,T],$ i.e.,
 $\overrightarrow{V_\varepsilon}^{(0)}$  is incompressible in $\Omega^{(0)}_\varepsilon $ for any $t\in [0, T],$
 there is no flow of $\overrightarrow{V_\varepsilon}^{(0)}$  through the surface $\Gamma^{(0)}_\varepsilon,$
 and the amount of the flow entering the node $\Omega^{(0)}_\varepsilon$ is equal to the amount of the flow outgoing it.

%%%%%%%%%%%%%%%%%%%%%
\subsection{The initial-boundary value problem\label{sec_ibvp}}
In $\Omega_\varepsilon,$ we consider the following semilinear parabolic convection-diffusion  problem:
\begin{equation}\label{probl}
\left\{\begin{array}{rcll}
 \partial_t u_\varepsilon -  \varepsilon\, \Delta_x u_\varepsilon +
  \mathrm{div}_x \big( \overrightarrow{V_\varepsilon} \, u_\varepsilon\big)
  & = & 0, &
    \text{in} \ \Omega_\varepsilon \times (0, T),
\\[2mm]
-  \varepsilon \,  \partial_{\boldsymbol{\nu}_\varepsilon} u_\varepsilon +  u_\varepsilon \, \overrightarrow{V_\varepsilon}\boldsymbol{\cdot}\boldsymbol{\nu}_\varepsilon &  = & \varepsilon^{\alpha} \varphi^{(0)}_\varepsilon\big(u_\varepsilon,x,t\big) &
   \text{on} \ \Gamma^{(0)}_\varepsilon \times (0, T),
\\[2mm]
-  \varepsilon \,  \partial_{\boldsymbol{\nu}_\varepsilon} u_\varepsilon +  u_\varepsilon \, \overrightarrow{V_\varepsilon}\boldsymbol{\cdot}\boldsymbol{\nu}_\varepsilon &  = & \varepsilon^{\alpha}  \varphi^{(i)}_\varepsilon\big(u_\varepsilon,x,t\big) &
   \text{on} \ \Gamma^{(i)}_\varepsilon \times (0, T),
\\[2mm]
 u_\varepsilon \big|_{x \in \Upsilon_{\varepsilon}^{(i)} (\ell_i)}
 & = & q_i(t), &  t\in (0, T)
\\[2mm]
u_\varepsilon\big|_{t=0}&=& 0, & \text{in} \ \Omega_{\varepsilon}.
 \end{array}\right.
\end{equation}
In \eqref{probl}, we denote  by ${\boldsymbol{\nu}}_\varepsilon$  the outward unit normal to $\partial \Omega_\varepsilon,$ and
$\partial_{\boldsymbol{\nu}_\varepsilon}$ denotes  the derivative along ${\boldsymbol{\nu}}_\varepsilon$.
The  intensity parameter $\alpha$  is from $ (0, +\infty)$ and $i \in \{1,\ldots,M\}.$

Taking into account \eqref{lap_1}, \eqref{div_1}, \eqref{grad_1}, and the assumptions made above, we can rewrite this problem as follows
\begin{equation}\label{probl_rewrite}
\left\{\begin{array}{rcll}
 \partial_t u_\varepsilon  +
    \overrightarrow{V_\varepsilon}^{(0)} \boldsymbol{\cdot} \nabla_x u_\varepsilon
  & = & \varepsilon\, \Delta_x u_\varepsilon, &
    \text{in} \ \Omega^{(0)}_\varepsilon \times (0, T),
\\[2mm]
-   \partial_{\boldsymbol{\nu}_\varepsilon} u_\varepsilon &  = & \varepsilon^{\alpha-1} \varphi^{(0)}_\varepsilon\big(u_\varepsilon,x,t\big) &
   \text{on} \ \Gamma^{(0)}_\varepsilon \times (0, T),
\\[2mm]
 \partial_t u^{(i)}_\varepsilon  +
  \mathrm{div}_{y^{(i)}} \big( \overrightarrow{V_\varepsilon}^{(i)} \, u^{(i)}_\varepsilon\big)
  & = &  \varepsilon\, \Delta_{y^{(i)}} u^{(i)}_\varepsilon, &
    \text{in} \ \Omega^{(i)}_\varepsilon \times (0, T),
\\[2mm]
-   \,  \partial_{\overline{\nu}_\varepsilon} u_\varepsilon +  u_\varepsilon \, \overline{V}_\varepsilon^{(i)} \boldsymbol{\cdot}\overline{\nu}_\varepsilon &  = & \varepsilon^{\alpha-1}\,  \varphi^{(i)}_\varepsilon\big(u^{(i)}_\varepsilon, y^{(i)},t\big) &
   \text{on} \ \Gamma^{(i)}_\varepsilon \times (0, T),
\\[2mm]
 u^{(i)}_\varepsilon \big|_{y^{(i)}_1= \ell_i}
 & = & q_i(t),  & t \in  (0, T),
\\[2mm]
u_\varepsilon\big|_{t=0}&=& 0, & \text{in} \ \Omega_{\varepsilon},
 \end{array}\right.
\end{equation}
where  $\varphi^{(0)}_\varepsilon(s,x,t) := \varphi^{(0)}\big(s, \tfrac{x}{\varepsilon},t\big)$
and $ \varphi^{(i)}_\varepsilon(s,y^{(i)},t) := \varphi^{(i)}\big(s, y_1^{(i)}, \tfrac{\overline{y}_1^{(i)}}{\varepsilon}, t\big),$
 \ $\overline{\nu}_\varepsilon = \big(\nu_2(\tfrac{\overline{y}_1^{(i)}}{\varepsilon}),  \nu_3(\tfrac{\overline{y}_1^{(i)}}{\varepsilon})\big)$
is the outward unit normal to the lateral surface $\partial \Omega^{(i)}_\varepsilon,$ and $i \in \{1,\ldots,M\}.$
\medskip

For the given functions on  the right-hand side of \eqref{probl_rewrite}, we suppose the following.
\begin{description}
  \item[{\bf A1}]
The  function $\varphi^{(0)}(s,\xi,t), \ (s,\xi,t) \in \big\{\Bbb R \times \overline{\Xi^{(0)}}\times [0,T]\big\} =: X_0$
vanishes uniformly with respect to $s\in \Bbb R$ and $t\in [0, T]$ in neighborhoods of
$\big\{\Upsilon^{(i)}_1(\ell_0)\big\}_{i=1}^M$.  It   belongs to
the space $C^2(X_0),$  and is  uniformly bounded in $X_0$ with all its derivatives.
  \item[{\bf A2}]
For each $i\in\{1,\ldots,M\},$ the function
$$
\varphi^{(i)}(s,y_1^{(i)}, \overline{y}_1^{(i)},t), \quad  \ \big\{(s,y_1^{(i)}) \in \Bbb R \times [0, \ell_i], \ \ |\overline{y}_1^{(i)}|\le h_i,  \  \  t\in [0,T]\big\} =: X_i,
$$
vanishes  uniformly with respect to $(s, \overline{y}_1^{(i)}, t)$  in neighborhoods of the ends  of the  closed interval  $[0, \ell_i],$
it   belongs to  $C^2(X_i),$  and it and all its derivatives are uniformly bounded in $X_i.$

  \item[{\bf A3}]
The  functions $\{q_i(t), \ t\in [0, T]\}_{i=1}^M $ are nonnegative and $q_i \in C^2([0, T]).$

\item[{\bf A4}] To satisfy the zero- and first-order matching conditions in \eqref{probl_rewrite}, it is necessary to fulfill the relations
\begin{equation}\label{match_conditions}
  q_i(0) =  \frac{d q_i}{dt}(0)=0, \quad i\in \{1,\ldots,M\}, \quad \varphi^{(i)}\big|_{t=0} =0, \quad i\in\{0, 1,\ldots,M\},
\end{equation}
or $\varphi^{(i)}\big|_{s=0} =0,$ if $\varphi^{(i)}$ is independent of the variable $t.$
  \end{description}

Based on the assumptions made above and according to  the classical theory of parabolic semilinear initial-boundary problems (see e.g. \cite[Chapt. V, \S\S 6, 7]{Lad_Sol_Ura_1968}) there exists  a unique classical solution to the problem \eqref{probl_rewrite} for each $\varepsilon > 0.$ Obviously, this is also a weak solution in the Sobolev space $L^2\big((0,T); H^1(\Omega_\varepsilon)\big).$

In the remainder of the paper we study the asymptotic behavior of the solution $u_\varepsilon$ as
$\varepsilon \to 0,$ i.e., when the thin junction $\Omega_\varepsilon$ is shrunk into the graph
$$
 \mathcal{I} := I_1 \cup I_2 \cup\ldots\cup I_M,
$$
 where $I_i := \{y^{(i)}\colon  y_1^{(i)} \in  [0, \ell_i], \ \overline{y}_1^{(i)} = (0, 0)\}$.  Namely, we  will
 \begin{itemize}
   \item
   derive the corresponding limit problem $(\varepsilon = 0)$ and prove its well-posedness,
      \item
   construct the asymptotic approximation for the solution to the problem \eqref{probl_rewrite} and prove the corresponding asymptotic estimates,
   \item
study the influence of the parameter $\alpha$ on the asymptotic behavior of $u_\varepsilon.$
 \end{itemize}
 %%%%%%%%%%%%%%%

\begin{remark}
Our approach would also allow to consider the problem \eqref{probl_rewrite} with a  more complex diffusion operator
$\varepsilon\, \mathrm{div}_x\big( \mathbb{D}_\varepsilon(x) \nabla_x u_\varepsilon\big)$ as  %where  the  matrix $\mathbb{D}_\varepsilon$  governs the  %diffusion  process, and  in a
as well as thin junctions with the curvilinear
cylinders (see e.g.\ \cite{Mel-Roh_preprint-2022}). However, we have omitted these generalizations to ease the presentation.
\
It is also possible  to consider different factors $\varepsilon^{\alpha_i}$ for each thin cylinder and the node,  respectively,
as was done in \cite{Mel_Klev_M2AS-2018},  but from the mathematical point of view, this will only complicate calculations. The case $\alpha =0$ for the node and $\alpha =1$ for the thin cylinders was considered by us for the linear convection-dominated problem in \cite{Mel-Roh_preprint-2022}.
\end{remark}

%%%%%%%%%%%%%%%%%%%%%

\section{Asymptotic approximation in the  case $\alpha =1$}\label{Sec:expansions}
Based on the approach from \cite{Mel-Roh_preprint-2022} for a linear convection-diffusion problem, we propose  an asymptotic approximation for  the solution $u_\varepsilon$ to  the problem~\eqref{probl_rewrite} in terms of three parts.
For each thin cylinder, we propose a regular ansatz that is augmented with an additional boundary layer ansatz in the case of an  outlet cylinder (see \S~\ref{subsec_V}) at its base, and all those regular ansatzes are connected with the node ansatz, namely
\begin{itemize}
  \item
  for each $ i \in \{1,\ldots,M\}$  the regular ansatz
 \begin{equation}\label{regul}
   \mathcal{U}^{(i)}_\varepsilon(y^{(i)},t) := w_0^{(i)}(y^{(i)}_1,t) + \varepsilon\Big(w_1^{(i)}(y^{(i)}_1,t)  + u_1^{(i)} \Big(y^{(i)}_1, \dfrac{\overline{y}^{(i)}_1}{\varepsilon}, t \Big)
     \Big) + \varepsilon^2 u_2^{(i)} \Big(y^{(i)}_1, \dfrac{\overline{y}^{(i)}_1}{\varepsilon}, t \Big),
 \end{equation}
 located inside of the thin cylinder $\Omega^{(i)}_\varepsilon;$
 \item
for each $ i \in \{m+1,\ldots,M\}$  the boundary-layer ansatz
\begin{equation}\label{prim+}
\mathcal{B}^{(i)}_\varepsilon(y^{(i)},t)
 := \sum\limits_{k=0}^{2}\varepsilon^{k} \, \Pi_k^{(i)}
    \left(\frac{\ell_i - y^{(i)}_1}{\varepsilon}, \frac{y^{(i)}_2}{\varepsilon}, \frac{y^{(i)}_3}{\varepsilon}, t\right),
 \end{equation}
 located in a neighborhood of the base $\Upsilon_{\varepsilon}^{(i)} (\ell_i)$ of the thin cylinder $\Omega^{(i)}_\varepsilon;$
 \item
and the node-layer one
  \begin{equation}\label{junc}
  \mathcal{N}_\varepsilon(x,t) := N_0\left(\frac{x}{\varepsilon},t\right)  +\varepsilon N_1\left(\frac{x}{\varepsilon},t\right),
\end{equation}
located in a neighborhood of the node $\Omega^{(0)}_\varepsilon$.
\end{itemize}

In the sequel we derive a closed set of problems for the  coefficients $ w^{(i)}_0,$ $ w^{(i)}_1,$ $u^{(i)}_1,$ $u^{(i)}_2$ for $i\in \{1,\ldots,M\}$,   $\Pi^{(i)}_0,\Pi^{(i)}_1,\Pi^{(i)}_2$   for $i\in \{m+1,\ldots,M\}$, and $ {{N}}_0, {N}_1$ of these ansatzes.  Based on these results we  construct  a   final approximation $\mathfrak{A}_\varepsilon$ for $u_\varepsilon $ in Section \ref{Sec:justification} which enables us to analyze the asymptotic regime $\varepsilon \to 0$ in \eqref{probl_rewrite}.

\subsection{Regular parts of the approximation}\label{regular_asymptotic}

Substituting the ansatz for $\mathcal{U}_\varepsilon^{(i)}$ from \eqref{regul} into the corresponding differential equation of the problem~\eqref{probl_rewrite},  collecting terms of the same powers of $\varepsilon^0$ and $\varepsilon^1,$ and equating these sum to zero,  we obtain the following differential equations:
\begin{align}\label{eq_1}
  \Delta_{\bar{\xi}_1}u^{(i)}_1\big(y^{(i)}_1, \bar{\xi}_1, t\big)
    = & \ \partial_t {w}_0^{(i)} (y^{(i)}_1,t) +  \big(  v_1^{(i)}(y^{(i)}_1, t) \, w^{(i)}_0(y^{(i)}_1, t) \big)^\prime \notag
    \\
    +  & \ w^{(i)}_0(y^{(i)}_1, t) \,   \mathrm{div}_{\bar{\xi}_1} \big( \overline{V}^{(i)}(y^{(i)}_1, \bar{\xi}_1, t) \big), \qquad \overline{\xi}_1\in\Upsilon_i (y^{(i)}_1),
\end{align}
and
\begin{align}\label{eq_2}
 \Delta_{\bar{\xi}_1}u^{(i)}_2\big(y^{(i)}_1, \bar{\xi}_1, t\big)
  = &  \ \Big(  v^{(i)}_1(y^{(i)}_1, t) \, \big[ w^{(i)}_1(y^{(i)}_1, t) + u^{(i)}_{1}(y^{(i)}_1,\bar{\xi}_1, t) \big] \Big)^\prime
  +  \partial_t {w}_1^{(i)}(y^{(i)}_1, t)  +   \partial_t {u}_1^{(i)}(y^{(i)}_1,\bar{\xi}_1, t) \notag
    \\
    + & \ \mathrm{div}_{\bar{\xi}_1} \Big( \overline{V}^{(i)}(y^{(i)}_1, \bar{\xi}_1,t) \,
            \big[ w^{(i)}_{1}(y^{(i)}_1, t) + u^{(i)}_{1}(y^{(i)}_1,\bar{\xi}_1, t) \big]\Big)  \notag
            \\
             - & \ \Big( w^{(i)}_{0}(y^{(i)}_1,t)  \Big)^{\prime\prime},
\qquad  \bar{\xi}_1 \in \Upsilon_i\big(y^{(i)}_1\big),
\end{align}
where  $\bar{\xi}_1= (\xi_2, \xi_3) = \big(\frac{{y}_2^{(i)}}{\varepsilon}, \frac{{y}_3^{(i)}}{\varepsilon}\big),$ $\Delta_{\bar{\xi}_1} := \frac{\partial^2}{\partial \xi_2^2} +  \frac{\partial^2}{\partial \xi_3^2},$  \ ``$\prime$'' denotes the derivative with respect to the longitudinal variable $y^{(i)}_1,$  $\Upsilon_i\big(y^{(i)}_1\big) := \big\{ \overline{\xi}_i\in\Bbb{R}^2 \colon |\overline{\xi}_i|< h_i \big\},$
\begin{equation}\label{overline_V}
  \overline{V}^{(i)}(y^{(i)}_1, \bar{\xi}_1, t) = \left( v^{(i)}_2(y^{(i)}_1, \bar{\xi}_1, t), \, v^{(i)}_3(y^{(i)}_1, \bar{\xi}_1, t) \right).
\end{equation}

Substituting \eqref{regul} in the boundary condition on the  lateral surface of the thin cylinder $\Omega^{(i)}_\varepsilon$ and using Taylor’s formula for $\varphi^{(i)},$ we deduce the following relations:
\begin{equation}\label{bc_1}
  - \partial_{\bar{\nu}_{\bar{\xi}_1}} u^{(i)}_1\big(y^{(i)}_1, \bar{\xi}_1, t\big) + w^{(i)}_{0}(y^{(i)}_1,t) \, \overline{V}^{(i)}(y^{(i)}_1, \bar{\xi}_1,t) \boldsymbol{\cdot} \bar{\nu}_{\xi_1}
 = \varphi^{(i)}\big(w^{(i)}_{0}(y^{(i)}_1,t), y^{(i)}_1, \bar{\xi}_1, t\big), \quad \bar{\xi}_1 \in \partial
  \Upsilon_i\big(y^{(i)}_1\big),
\end{equation}
and
\begin{multline}\label{bc_2}
  - \partial_{\bar{\nu}_{\bar{\xi}_1}} u^{(i)}_2\big(y^{(i)}_1, \bar{\xi}_1, t\big) + \big( w^{(i)}_{1}(y^{(i)}_1, t) + u^{(i)}_{1}(y^{(i)}_1,\bar{\xi}_1, t) \big) \, \overline{V}^{(i)}(y^{(i)}_1, \bar{\xi}_1,t) \boldsymbol{\boldsymbol{\cdot}} \bar{\nu}_{\xi_1}
   \\
   = \partial_s\varphi^{(i)}\big(w^{(i)}_{0}(y^{(i)}_1,t), y^{(i)}_1, \bar{\xi}_1, t\big) \, \big( w^{(i)}_{1}(y^{(i)}_1, t) + u^{(i)}_{1}(y^{(i)}_1,\bar{\xi}_1, t) \big), \quad \bar{\xi}_1 \in \partial \Upsilon_i\big(y^{(i)}_1\big),
\end{multline}
where $\partial_{\bar{\nu}_{\bar{\xi}_1}} $ is the derivative along the  outward unit normal $\bar{\nu}_{\bar{\xi}_1} =
\big(\nu_2(\xi_2,\xi_3),  \nu_3(\xi_2,\xi_3)\big)$ to the boundary of the disk $\Upsilon_i\big(y^{(i)}_1\big).$

Equations \eqref{eq_1} and \eqref{bc_1},  and equations \eqref{eq_2} and \eqref{bc_2}
are linear Neumann problems in $\Upsilon_i\big(y^{(i)}_1\big)$ with respect to the variables $\bar{\xi}_1,$ where  the variables $y^{(i)}_1$ and $t$ are regarded as parameters from the set $I_\varepsilon^{(i)} \times (0, T)$,
$ I_\varepsilon^{(i)} := \{y^{(i)}\colon y^{(i)}_1 \in (\varepsilon \ell_0, \ell_i),$ $\bar{y}^{(i)}_1 = (0, 0) \}.
$
To ensure  uniqueness, we supply each  problem with the  condition
\begin{equation}\label{uniq_1}
  \langle u_k^{(i)}(y^{(i)}_1, \, \cdot \, ,  t ) \rangle_{\Upsilon_i\big(y^{(i)}_1\big)} :=  \int_{\Upsilon_i\big(y^{(i)}_1\big)} u_k^{(i)}\big(y^{(i)}_1, \bar{\xi}_1, t\big) \, d\bar{\xi}_1 = 0 \quad (k\in \{1, 2\}).
\end{equation}

Writing down the solvability condition for each problem, we deduce the differential equations for the coefficients  $w^{(i)}_0$ and $w^{(i)}_1$:
\begin{equation}\label{lim_0}
  \partial_t{w}^{(i)}_0(y^{(i)}_1,t) + \big( v_i^{(i)}(y^{(i)}_1,t)\, w^{(i)}_0(y^{(i)}_1,t) \big)^\prime = - \widehat{\varphi}^{(i)}\big(w^{(i)}_0, y^{(i)}_1, t\big), \quad (y^{(i)}_1, t) \in I_\varepsilon^{(i)} \times (0, T),
\end{equation}
where
\begin{equation}\label{hat_phi}
  \widehat{\varphi}^{(i)}(w^{(i)}_0,y^{(i)}_1, t) := \frac{1}{\pi h^2_i} \int_{\partial \Upsilon_i(y^{(i)}_1)} \varphi^{(i)}\big(w^{(i)}_0(y^{(i)}_1,t), y^{(i)}_1, \bar{\xi}_1, t\big)\, d\sigma_{\bar{\xi}_1}.
\end{equation}
and
\begin{equation}\label{lim_1}
  \partial_t{w}^{(i)}_1(y^{(i)}_1,t) \, + \, \big( v_i^{(i)}(y^{(i)}_1,t)\, w^{(i)}_1(y^{(i)}_1,t) \big)^\prime
 \, + \,\partial_s \widehat{\varphi}^{(i)}\big(w^{(i)}_0, y^{(i)}_1, t\big) \, \, w^{(i)}_1(y^{(i)}_1,t) = f_i(y^{(i)}_1,t)
  \end{equation}
in $I_\varepsilon^{(i)} \times (0, T),$ where
\begin{equation}\label{fun_1}
  f_i(y^{(i)}_1,t) :=  \big(w^{(i)}_0(y^{(i)}_1,t) \big)^{\prime\prime} -
\frac{1}{\pi h^2_i} \int_{\partial \Upsilon_i(y^{(i)}_1)} \partial_s \varphi^{(i)}\big(w^{(i)}_0, y^{(i)}_1, \bar{\xi}_1, t\big)\,
u^{(i)}_1\big(y^{(i)}_1, \bar{\xi}_1, t\big) \, d\sigma_{\bar{\xi}_1}.
\end{equation}

We see that the equation \eqref{lim_0} is a semilinear partial differential equation of  first order. Let  $w^{(i)}_0$ be a solution of this equation (we will show its existence in Section \ref{regular_asymptotic}). Then there is a unique solution to the inhomogeneous Neumann problem consisting of the equations \eqref{eq_1}, \eqref{bc_1} and \eqref{uniq_1}  $(k=1).$

The  equation \eqref{lim_1} is a linear partial differential equation of the first order, in which the right-hand side is determined with the help of
$w^{(i)}_0,$ $u^{(i)}_1,$ and $f_i.$ Again, let there be a solution of this equation. Then there is a unique solution to the inhomogeneous Neumann problem consisting of the equations \eqref{eq_2}, \eqref{bc_2} and \eqref{uniq_1}  $(k=2).$

Thus, we can uniquely determine all ansatz functions
in the regular  part $ \mathcal{U}^{(i)}_\varepsilon$  in\eqref{regul}. From the calculations made above it follows
that $ \mathcal{U}^{(i)}_\varepsilon$  satisfies
\begin{equation}\label{ref_1}
\left\{\begin{array}{c}
 \partial_t \mathcal{U}^{(i)}_\varepsilon -  \varepsilon\, \Delta_{y^{(i)}} \mathcal{U}^{(i)}_\varepsilon +
  \mathrm{div}_{y^{(i)}} \big( \overrightarrow{V_\varepsilon}^{(i)} \, \mathcal{U}^{(i)}_\varepsilon\big)
   = \varepsilon^2\,  \mathfrak{F}^{(i)}_\varepsilon \quad
    \text{in} \ \Omega^{(i)}_\varepsilon \times (0, T),
\\[2pt]
-   \,  \partial_{\overline{\nu}_\varepsilon} \mathcal{U}^{(i)}_\varepsilon +  \mathcal{U}^{(i)}_\varepsilon \, \overline{V_\varepsilon}\boldsymbol{\cdot}\overline{\nu}_\varepsilon   =  \varphi^{(i)}_\varepsilon\big(\mathcal{U}^{(i)}_\varepsilon, y^{(i)},t\big) + \varepsilon^2 \,  \Phi^{(i)}_\varepsilon \quad  \text{on} \ \Gamma^{(i)}_\varepsilon \times (0, T),
    \end{array}\right.
\end{equation}
where
\begin{align}\label{res_2}
  \mathfrak{F}^{(i)}_\varepsilon (y^{(i)}, t) = &  \ \bigg[
  \Big(  v^{(i)}_1(y^{(i)}_1, t) \, u^{(i)}_{2}(y^{(i)}_1,\bar{\xi}_1, t)  \Big)^\prime
  +   \partial_t {u}_2^{(i)}(y^{(i)}_1,\bar{\xi}_1, t) \notag
  \\
 + & \
  \mathrm{div}_{\bar{\xi}_i} \Big( \overline{V}^{(i)}(y^{(i)}_1, \bar{\xi}_1,t) \,
            u^{(i)}_{2}(y^{(i)}_1,\bar{\xi}_1, t)\Big) \notag
            \\
             - & \
                \Big( w^{(i)}_{1}(y^{(i)}_1, t)  +  u^{(i)}_{1}(y^{(i)}_1, \bar{\xi}_1,t) \Big)^{\prime\prime} - \Big(u^{(i)}_{2}(y^{(i)}_1, \bar{\xi}_1,t) \Big)^{\prime\prime}\bigg] \Big|_{\bar{\xi}_1\,  = \, \bar{y}^{(i)}_1 /\varepsilon},
\end{align}
\begin{multline}\label{res_3}
  \Phi^{(i)}_\varepsilon(y^{(i)}, t) =  \bigg[ u^{(i)}_{2}(y^{(i)}_1,\bar{\xi}_1, t) \, \Big(\partial_s\varphi^{(i)}\big(w^{(i)}_{0}(y^{(i)}_1,t), y^{(i)}_1, \bar{\xi}_1, t\big)  - \overline{V}^{(i)}(y^{(i)}_1, \bar{\xi}_1,t) \boldsymbol{\boldsymbol{\cdot}} \bar{\nu}_{\xi_1}\Big)
            \\
             +  \frac{1}{2} \partial^2_{ss}\varphi^{(i)}\big(\theta_\varepsilon, y^{(i)}_1, \bar{\xi}_1, t\big)
                \Big( w^{(i)}_{1}(y^{(i)}_1, t)  +  u^{(i)}_{1}(y^{(i)}_1, \bar{\xi}_1,t) + \varepsilon u^{(i)}_{2}(y^{(i)}_1,\bar{\xi}_1, t)\Big)\bigg] \Big|_{\bar{\xi}_1\,  = \, \bar{y}^{(i)}_1 /\varepsilon}.
\end{multline}

\begin{remark}\label{r_2_1}
Since the functions $\{\varphi^{(i)}\}_{i=1}^M$ and $\{\overline{V}^{(i)}\}_{i=1}^M$  have compact supports with respect to the corresponding longitudinal variable, the coefficients $u_1^{(i)}, u_2^{(i)}$ and  the function $\Phi^{(i)}_\varepsilon$ vanish in the corresponding neighborhoods
of the ends of $[0, \ell_i].$
\end{remark}

\subsection{Node-layer part of the approximation\label{subsec_Inner_part}}
To find gluing conditions for the coefficients $\{w_0^{(i)}\}_{i=1}^M,$ and $\{ w_1^{(i)}\}_{i=1}^M$ at the origin,  we should launch
the inner part of the asymptotics for the solution.  For this purpose we pass to the scaled variables
$
\xi=\frac{x}{\varepsilon} .
$
Letting $\varepsilon \to 0,$ we see  that the domain $\Omega_\varepsilon$ is transformed into the unbounded domain $\Xi$ that  is the union of the domain~$\Xi^{(0)}$ and  the semi-infinite cylinders
$$
\Xi^{(i)}
 =  \big\{ \xi^{(i)} =\big(\xi^{(i)}_1,\xi^{(i)}_2,\xi^{(i)}_3\big)\in\Bbb R^3 \colon
    \  \ell_0<\xi^{(i)}_1 <+\infty,
    \ |\overline{\xi}^{(i)}_1|<h_i \big\},
\quad i\in \{1,\ldots,M\},
$$
i.e., $\Xi$ is the interior of  $\bigcup_{i=0}^M\overline{\Xi^{(i)}}.$
Here each cylinder is described  with corresponding variables $\xi^{(i)} =\dfrac{y^{(i)}}{\varepsilon}.$
The  boundary of $\Xi$ is partitioned according to
$$
\Gamma_0 = \partial\Xi \backslash \Big(\bigcup_{i=1}^M \Gamma_i \Big) \quad \text{and} \quad
\Gamma_i =
\big\{ \xi^{(i)}\in\Bbb R^3 \colon
    \  \ell_0<\xi^{(i)}_1 <+\infty,
    \ |\overline{\xi}^{(i)}_1|=h_i \big\}
$$
for $\  i\in \{1,\ldots,M\}.$

Substituting  \eqref{junc}
 into differential equations of the problem~\eqref{probl_rewrite} and the  boundary condition on  $\Gamma_\varepsilon^{(0)},$  collecting terms of the same powers of $\varepsilon^{-1}$ and $\varepsilon^0,$ and equating these sums to zero,  we get the following problems
 for $N_0$ and $ N_1$ from \eqref{junc}:
\begin{equation}\label{N_0_prob}
\left\{\begin{array}{rcll}
  -  \Delta_\xi {N}_0(\xi,t)  + \overrightarrow{V}(\xi,t) \boldsymbol{\cdot} \nabla_\xi {N}_0(\xi,t) & = &  0, &  \xi \in\Xi^{(0)},
\\[2mm]
\partial_{\boldsymbol{\nu}_\xi}  N_0(\xi,t) &=& 0, &   \xi \in \Gamma_0,
\\[2mm]
 -   \Delta_{\xi^{(i)}} {N}_0(\xi^{(i)},t) +
  \mathrm{v}_i(t) \, \partial_{\xi^{(i)}_1}{N}_0(\xi^{(i)},t)    & = &  0, &
     \xi^{(i)} \in\Xi^{(i)},
\\[2mm]
\partial_{\bar{\nu}^{(i)}} N_0(\xi^{(i)},t)  &=&  0, &
   \xi^{(i)} \in \Gamma_i,
\\[2mm]
N_0(\xi^{(i)},t)     \  \sim \   w^{(i)}_{0}(0,t)   &\text{as}  & \xi^{(i)}_1 \to +\infty, &
   \xi^{(i)}  \in \Xi^{(i)}, \ \ i\in \{1,\ldots,M\},
\end{array}\right.
\end{equation}
and
\begin{equation}\label{N_1_prob}
\left\{\begin{array}{rclll}
-   \Delta_\xi {N}_1(\xi,t) +
  \overrightarrow{V}(\xi,t) \boldsymbol{\cdot} \nabla_\xi {N}_1(\xi,t) & = &  - \partial_t {{N}}_{0}(\xi,t), &  \xi \in\Xi^{(0)},
\\[2mm]
- \partial_{\boldsymbol{\nu}_\xi}  N_1(\xi,t) &=& \varphi^{(0)}\big(N_0(\xi,t), \xi,t\big) , &   \xi \in \Gamma_0,
\\[2mm]
 -   \Delta_{\xi^{(i)}} {N}_1(\xi^{(i)},t)  +
  \mathrm{v}_i(t) \, \partial_{\xi^{(i)}_1}{N}_1(\xi^{(i)},t) & = &  - \partial_t{{N}}_{0}(\xi^{(i)},t), &
     \xi^{(i)} \in\Xi^{(i)},
\\[2mm]
\partial_{\bar{\nu}^{(i)}} N_1(\xi^{(i)},t)  &=&  0, &
   \xi^{(i)} \in \Gamma_i,
\\[2mm]
N_1(\xi^{(i)},t)     \  \sim \   w^{(i)}_{1}(0,t) + \Psi^{(i)}_{1}(\xi^{(i)},t)      &\text{as}  &\xi^{(i)}_1 \to +\infty, &
    \xi^{(i)}  \in \Xi^{(i)}, \ \ i\in \{1,\ldots,M\},
 \end{array}\right.
\end{equation}
where $ \partial_{\xi^{(i)}_1} = \frac{\partial}{\partial{\xi^{(i)}_1}},$ $\partial_{\bar{\nu}^{(i)}}$ is the derivative along the  outward unit normal $\bar{\nu}^{(i)} = \big(\nu_2(\xi^{(i)}_2,\xi^{(i)}_3),  \nu_3(\xi^{(i)}_2,\xi^{(i)}_3)\big)$ to $\Gamma_i,$
$\overrightarrow{V}(\xi) = \overrightarrow{V}_\varepsilon^{(0)}(x)\big|_{x= \varepsilon \xi}$ in $\Xi^{(0)},$ $\overrightarrow{V}(\xi^{(i)}) = \left( \mathrm{v}_i(t), \, 0,\, 0\right)$ in $\xi^{(i)} \in \Xi^{(i)}.$
Note that  the variable $t$ appears as a parameter only  in the  steady convection-diffusion problems \eqref{N_0_prob} and \eqref{N_1_prob}.

The asymptotic conditions at infinity in the last relations of  both of the problems \eqref{N_0_prob} and \eqref{N_1_prob}  appear by matching the regular and node-layer parts of the  asymptotics in a neighborhood of the node: the asymptotics of the terms $N_0$ and $N_1$ as $\xi^{(i)}_1  \to +\infty$ have to coincide with the corresponding asymptotics of  terms of the regular expansions (\ref{regul}) as $y^{(i)}_1 \to +0,$ $ i\in \{1,\ldots,M\},$ respectively. Using Taylor's formula for each term of the regular asymptotics  at  $y^{(i)}_1=0,$ then passing to the scaled variable $\xi_1^{(i)}$
and collecting the coefficients of the same powers of $\varepsilon,$  we get these relations with
\begin{equation}\label{Psi_1}
\Psi_{0}^{(i)} \equiv 0, \qquad   \Psi_{1}^{(i)}(\xi^{(i)}_1,t)
 =   \xi^{(i)}_1  \dfrac{\partial  w_{0}^{(i)}}{\partial y^{(i)}_1} (0,t),
\quad i\in \{1,\ldots,M\}.
\end{equation}

We look for solutions to these problems in the form
\begin{equation}\label{new-solution_k}
N_k(\xi^{(i)}_1,t)  = \sum\limits_{i=1}^M\big(w^{(i)}_{k}(0,t) + \Psi^{(i)}_{k}(\xi^{(i)}_1,t) \big) \,\chi_{\ell_0}(\xi^{(i)}_1) + \widetilde{N}_k \quad (k \in \{0, 1\}),
\end{equation}
where $ \chi_{\ell_0} \in C^{\infty}(\Bbb{R})$ is a smooth cut-off function such that
$\ 0\leq \chi_{\ell_0} \leq1,$  $\chi_{\ell_0}(s) =0$ if $s \leq  2\ell_0$  and
$\chi_{\ell_0}(s) =1$ if $s \geq  3\ell_0.$
Consequently  $\widetilde{N}_0$ and $\widetilde{N}_1$ have  to satisfy, respectively,
\begin{equation}\label{tilda_N_0_prob}
\left\{\begin{array}{rcll}
  -  \Delta_\xi \widetilde{N}_0(\xi,t)  + \overrightarrow{V}(\xi,t) \boldsymbol{\cdot} \nabla_\xi \widetilde{N}_0(\xi,t) & = &  0, &  \xi \in\Xi^{(0)},
\\[2mm]
\partial_{\boldsymbol{\nu}_\xi}  \widetilde{N}_0(\xi,t) &=& 0, &   \xi \in \Gamma_0,
\\[2mm]
 -   \Delta_{\xi^{(i)}} \widetilde{N}_0(\xi^{(i)},t) +
  \mathrm{v}_i(t) \, \partial_{\xi^{(i)}_1}\widetilde{N}_0(\xi^{(i)},t)    & = & g_0^{(i)}(\xi^{(i)}_1,t), &
     \xi^{(i)} \in\Xi^{(i)},
\\[2mm]
\partial_{\bar{\nu}^{(i)}} \widetilde{N}_0(\xi^{(i)},t)  &=&  0, &
   \xi^{(i)} \in \Gamma_i,
\\[2mm]
   \widetilde{N}_0(\xi^{(i)},t) \ \rightarrow \ 0 &\text{as} & \xi^{(i)}_1 \to +\infty, &  \xi^{(i)}  \in \Xi^{(i)}, \ \ i\in \{1,\ldots,M\},
\end{array}\right.
\end{equation}
and
\begin{equation}\label{tilda_N_1_prob}
\left\{\begin{array}{rcll}
-   \Delta_\xi \widetilde{N}_1 +
  \overrightarrow{V}(\xi,t) \boldsymbol{\cdot} \nabla_\xi \widetilde{N}_1(\xi,t) & = &  - \partial_t {{N}}_{0}(\xi,t), &  \xi \in\Xi^{(0)},
\\[2mm]
- \partial_{\boldsymbol{\nu}_\xi}  \widetilde{N}_1(\xi,t) &=& \varphi^{(0)}\big(N_0(\xi,t), \xi,t\big) , &   \xi \in \Gamma_0,
\\[2mm]
 -   \Delta_{\xi^{(i)}} \widetilde{N}_1  +
  \mathrm{v}_i(t) \, \partial_{\xi^{(i)}_1}\widetilde{N}_1(\xi^{(i)},t) & = &  - \partial_t \widetilde{N}_{0} + g_1^{(i)}(\xi^{(i)}_1,t), &
     \xi^{(i)} \in\Xi^{(i)},
\\[2mm]
\partial_{\bar{\nu}^{(i)}} \widetilde{N}_1(\xi^{(i)},t)  &=&  0, &
   \xi^{(i)} \in \Gamma_i,
\\[2mm]
\widetilde{N}_1(\xi^{(i)},t) \ \rightarrow \ 0 &\text{as} & \xi^{(i)}_1 \to +\infty, &  \xi^{(i)}  \in \Xi^{(i)}, \ \ i\in \{1,\ldots,M\},
 \end{array}\right.
\end{equation}
where
\begin{equation}\label{F_1-0}
g_0^{(i)}(\xi^{(i)}_1,t)  =   w^{(i)}_0(0,t) \, \chi''_{\ell_0}(\xi^{(i)}_1) - \mathrm{v}_i(t)  \, w^{(i)}_0(0,t) \, \chi'_{\ell_0}(\xi^{(i)}_1) ,
\end{equation}
\begin{align}\label{F_1}
g_1^{(i)}(\xi^{(i)}_1,t) = & \  w^{(i)}_1(0,t) \, \chi''_{\ell_0}(\xi^{(i)}_1) + \dfrac{\partial w_0^{(i)}}{\partial y^{(i)}_1}(0,t)\,   \Big( \big(\xi^{(i)}_1 \chi_{\ell_0}^{\prime}(\xi^{(i)}_1)\big)^\prime
 +  \chi_{\ell_0}^{\prime}(\xi^{(i)}_1) \Big) \notag
\\
- & \ \mathrm{v}_i(t)  \Big( w^{(i)}_1(0,t)  + \dfrac{\partial w_0^{(i)}}{\partial y^{(i)}_1}(0,t) \,  \xi^{(i)}_1 \Big) \chi^\prime_{\ell_0}(\xi^{(i)}_1).
\end{align}

Before we proceed with the limit problem for $w^{(i)}_0$
we employ a result from  the appendix in Section \ref{appendixellip}.
The equality \eqref{cong_cond_g} from   Proposition \ref{Prop-2-1}
that is necessary and sufficient for   unique weak solvability of \eqref{tilda_N_0_prob}  results in the Kirchhoff condition
\begin{equation}\label{cong_cond}
  \sum_{i=1}^{M}  h_i^2 \,\mathrm{v}_i(t)  \,  w_0^{(i)}(0, t) = 0.
\end{equation}
This leads us to a problem on the graph $\mathcal{I}.$
%%%%%%%%%%%%%%%%%%%%%%%%%%%%%%%%%%%%%%%%%%%%%%%%%%%%%%

%%%%%%%%%%%%%%%%%%%%%%%%%%%%%%%%%%%%%%%%%%%%%%%%%%%%%%
\subsection{The limit problem\label{limitproblemr}}
%%%%%%%%%%%%%%%%%%%%%%%%%%%%%%%%%%%%%%%%%%%%%%%%%%%%%%
On the graph $\mathcal{I},$ we have the relations
\begin{equation}\label{limit_prob_1}
 \left\{\begin{array}{rcll}
 \partial_t{w}^{(i)}_0 + \Big( v_i^{(i)}(y^{(i)}_1,t)\, w^{(i)}_0 \Big)^\prime &=& - \widehat{\varphi}^{(i)}\big(w^{(i)}_0, y^{(i)}_1, t\big),&  (y^{(i)}_1, t) \in  I_i \times (0, T),
 \\
 & &  & i \in \{1,\ldots,M\},
 \\
 \sum_{i=1}^{M}  h_i^2 \,\mathrm{v}_i(t)  \,  w_0^{(i)}(0, t) &=& 0, & t \in (0, T),
 \end{array}\right.
\end{equation}
which are supplemented by the following boundary and initial conditions:
\begin{equation}\label{limit_prob_2}
 \left\{\begin{array}{rclll}
    w_0^{(i)}(\ell_i,  t) & = & q_i(t), & t \in [0, T],  & i \in \{1,\ldots,m\},
\\[2mm]
    w_0^{(i)}(y^{(i)}_1,  0) & = & 0, & y^{(i)}_1  \in [0, \ell_i],  & i \in \{1,\ldots,M\}.
  \end{array}\right.
\end{equation}
The problem \eqref{limit_prob_1}-\eqref{limit_prob_2} is called the hyperbolic \textit{limit problem to} \eqref{probl_rewrite} in the case $\alpha =1.$
In the sequel we discuss its solvability depending on the number of  inlet and outlet cylinders.
The solvability criteria base on the characteristics approach which is summarized in
Section \ref{appendixhyp} of the appendix.
%
%We return to the limit problem \eqref{limit_prob_1}  and observe  f
From the general  results therein we deduce that for each $i \in \{1\ldots,m\}$, i.e., for each inlet cylinder,  we can find a solution to  the problem
\begin{equation}\label{limit_prob_1_2}
 \left\{\begin{array}{rcll}
 \partial_t{w}^{(i)}_0 + \Big( v_i^{(i)}(y^{(i)}_1,t)\, w^{(i)}_0 \Big)^\prime &=& - \widehat{\varphi}^{(i)}\big(w^{(i)}_0, y^{(i)}_1, t\big),&  (y^{(i)}_1, t) \in  I_i \times (0, T),
   \\[2mm]
    w_0^{(i)}(\ell_i,  t) & = & q_i(t), & t \in [0, T],
\\[2mm]
    w_0^{(i)}(y^{(i)}_1,  0) & = & 0, & y^{(i)}_1  \in [0, \ell_i].
  \end{array}\right.
\end{equation}
Indeed, this problem is reduced to \eqref{prob_1}  with the substitution $x = \ell_i -y^{(i)}_1.$ As a result, we get the equation
$$
\partial_t \tilde{w}^{(i)}_0(x,t) + \tilde{v}_i(x,t)\, \partial_{x} \tilde{w}^{(i)}_0(x,t) =  \partial_{x}\tilde{v}_i(x,t)\, \tilde{w}^{(i)}_0(x,t) - \widehat{\varphi}^{(i)}\big(\tilde{w}^{(i)}_0, \ell_i -x, t\big),
$$
where $\tilde{w}^{(i)}_0(x,t)= {w}^{(i)}_0(\ell_i -x,t)$ and $\tilde{v}(x,t) = - v_i^{(i)}(\ell_i -x,t) > 0.$
Taking into account the assumptions {\bf A2}--{\bf A4}, the conditions \eqref{bound_1} and \eqref{math_1} are satisfied for the problem \eqref{limit_prob_1_2}, and therefore, it has a unique classical solution.

The other components $\{w^{(i)}_0\}_{i=m+1}^M$ of the solution to the limit problem \eqref{limit_prob_1}-\eqref{limit_prob_2}
cannot be uniquely determined except of the cases when $m=1$ or $m=M-1.$

If $m=1,$ then  $\{w^{(i)}_0\}_{i=2}^M$ are respectively classical solutions to the problem
\begin{equation}\label{limit_prob_1_3}
 \left\{\begin{array}{rcll}
 \partial_t{w}^{(i)}_0 + \Big( v_i^{(i)}(y^{(i)}_1,t)\, w^{(i)}_0 \Big)^\prime &=& - \widehat{\varphi}^{(i)}\big(w^{(i)}_0, y^{(i)}_1, t\big),&  (y^{(i)}_1, t) \in  I_i \times (0, T),
   \\[2mm]
    w_0^{(i)}(0,  t) & = & w_0^{(1)}(0,  t), & t \in [0, T],
\\[2mm]
    w_0^{(i)}(y^{(i)}_1,  0) & = & 0, & y^{(i)}_1  \in [0, \ell_i].
  \end{array}\right.
\end{equation}
Such a solution exists, since $v_i^{(i)} >0,$ $w_0^{(1)}(0,0)=0,$ and from \eqref{Volt_1} we deduce that
$\partial_t w_0^{(1)}(0,0)=0.$
Due to \eqref{cond_1}, the Kirchhoff transmission condition in \eqref{limit_prob_1} is automatically satisfied. Thus, in this case,
the limit problem \eqref{limit_prob_1}-\eqref{limit_prob_2} has a classical solution, and two conditions  are satisfied at the graph vertex
(the continuity  condition and the Kirchhoff condition).

If $m=M-1,$ then the function $w^{(M)}_0$ must be a solution to the problem
\begin{equation}\label{limit_prob_1_4}
 \left\{\begin{array}{c}
 \partial_t{w}^{(M)}_0 + \Big( v_M^{(M)} \, w^{(M)}_0 \Big)^\prime = - \widehat{\varphi}^{(M)}\big(w^{(M)}_0, y^{(M)}_1, t\big),   \quad (y^{(M)}_1, t) \in  I_M \times (0, T)
   \\[2mm]
    w_0^{(M)}(0,  t)  =  - \dfrac{1}{h_M^2 \mathrm{v}_M(t) } \displaystyle{\sum_{i=1}^{M-1}  h_i^2 \,\mathrm{v}_i(t)  \,  w_0^{(i)}(0, t),} \quad t \in [0, T],
\\[4mm]
    w_0^{(M)}(y^{(M)}_1,  0)  =  0, \quad  y^{(M)}_1  \in [0, \ell_M],
  \end{array}\right.
\end{equation}
so that the Kirchhoff condition is fulfilled for the limit problem \eqref{limit_prob_1}-\eqref{limit_prob_2}. By the same arguments as before, we conclude that there exists a unique classical solution to the problem  \eqref{limit_prob_1_4},  and hence a unique classical solution to the limit problem.

If $1< m < M-1,$ then for each $i\in \{m+1,\ldots,M\}$  we propose to define $w^{(i)}_0$ as a classical solution to the problem
\begin{equation}\label{limit_prob_1_5}
 \left\{\begin{array}{c}
 \partial_t{w}^{(i)}_0 + \Big( v_i^{(i)}(y^{(i)}_1,t)\, w^{(i)}_0 \Big)^\prime  = - \widehat{\varphi}^{(i)}\big(w^{(i)}_0, y^{(i)}_1, t\big), \quad   (y^{(i)}_1, t) \in  I_i \times (0, T),
   \\[2mm]
    w_0^{(i)}(0,  t)  =  - \dfrac{1}{(M-m)\, h_i^2 \, \mathrm{v}_i(t) } \displaystyle{\sum_{i=1}^{m}  h_i^2 \,\mathrm{v}_i(t)  \,  w_0^{(i)}(0, t),}  \ \ t \in [0, T],
\\[2mm]
    w_0^{(i)}(y^{(i)}_1,  0)  =  0, \quad  y^{(i)}_1  \in [0, \ell_i],
  \end{array}\right.
\end{equation}
(our choice is argued and discussed in the final Section  \ref{conclusions}). Then the Kirchhoff condition is obviously satisfied in the limit problem \eqref{limit_prob_1}-\eqref{limit_prob_2}, and hence the limit problem has a classical solution.
\begin{remark}\label{add_assumptios}
The right-hand side of the differential equation \eqref{ref_1} contains the term $\big(u^{(i)}_{2}(y^{(i)}_1, \bar{\xi}_1,t) \big)^{\prime\prime}$ (see \eqref{res_2}), which means that  the solutions $\{w^{(i)}_0\}_{i=1}^M$ require  $C^4$-smoothness in $y^{(i)}_1.$ Using the approach in the proof of Theorem 2 \cite{Myshkis_1960}, one can  show that the boundedness of the derivatives of $\varphi^{(i)}$ up to  the fifth order in the variables $s$
 and $y^{(i)}_1$ from $X_i$ (the domain of definition of $\varphi^{(i)}$),
${v}^{(i)}_1 \in C^{4}([0, \ell_i]),$ $\partial_{\xi_2}\overline{V}^{(i)}\in C^{2}([0, \ell_i]),$
$\partial_{\xi_3}\overline{V}^{(i)}\in C^{2}([0, \ell_i])$ for $i\in \{1,\ldots,M\},$
 and the relations
\begin{equation}\label{match_conditions+}
 \frac{d^2 q_i}{dt^2}(0)= \frac{d^3 q_i}{dt^3}(0) = \frac{d^4 q_i}{dt^4}(0) = 0
\end{equation}
are sufficient conditions  for the $C^4$-smoothness of $w^{(i)}_0$ $(i \in \{1,\ldots,m\}).$ The conditions \eqref{match_conditions+}
 are necessary to satisfy the corresponding matching condition at the point $(\ell_i,0),$ to check which we also use the assumption {\bf A2}, namely,
the fact that $\varphi^{(i)}$ has the compact support  with respect to $y^{(i)}_1 \in (0, \ell_i).$

In addition, from the corresponding Volterra integral equation we derive that
$$
\frac{\partial^k w_0^{(i)}}{\partial t^k}(0,0)=0, \quad k\in \{1,\ldots,4\}, \quad  i\in \{1,\ldots,m\}.
$$
Assuming the same smoothness for the other functions $\{\varphi^{(i)}\}_{i=m+1}^M,$ we obtain $C^4$-smoothness  for the solutions $\{w^{(i)}_0\}_{i=m+1}^M$ as well.
\end{remark}

Thus, the limit problem \eqref{limit_prob_1}-\eqref{limit_prob_2} has a solution and the solvability condition \eqref{cong_cond} for the problem \eqref{tilda_N_0_prob} is satisfied. This means that there exists a unique solution $N_0$ to the problem \eqref{N_0_prob}. The solvability condition \eqref{lim_0} also holds for the problem \eqref{eq_1} and \eqref{bc_1} and therefore, there  exists a unique solution $u_1^{(i)}$ that satisfies the condition  \eqref{uniq_1}. It is easy to verify that
\begin{equation}\label{zero_t_0}
   N_0\big|_{t=0} = \partial^k_t N_0\big|_{t=0} = 0 \quad \text{and} \quad
   u_1^{(i)}\big|_{t=0} = \partial_t u_1^{(i)}\big|_{t=0} = 0
\end{equation}
for $k\in \{1,\ldots,4\}$  and $i\in \{1,\ldots,M\}.$
In addition, since the right-hand sides  in the problem \eqref{tilda_N_0_prob} are uniformly bounded with respect to $(\xi, t)\in \Xi\times [0, T]$ and have compact supports in $\xi^{(i)}_1$,  the solution to the problem~\eqref{N_0_prob} has  the following  asymptotics uniform with respect to $t\in [0, T]$:
\begin{equation}\label{rem_exp-decrease+0}
N_0(\xi,t) = w^{(i)}_{0}(0,t)  +  \mathcal{ O}(\exp(-\beta_0\xi_i))
\ \ \mbox{as} \ \ \xi_i\to+\infty,  \   \xi  \in \Xi^{(i)},  \  i=\{1,\ldots,M\} \ \ (\beta_0 >0).
\end{equation}
%Here, the symbol $\mathcal{ O}(\cdot)$ is big O notation.

%%%%%%%%%%%%%%%%%%%%%%%%%%%%

\subsection{Existence of solutions $\{w^{(i)}_1\}_{i=1}^M, \, \{u^{(i)}_2\}_{i=1}^M$ and $N_1$}\label{par_222}
Due to Proposition~\ref{Prop-2-1}  in  Section \ref{appendixellip}
the solvability condition for the problem \eqref{tilda_N_1_prob}
is as follows
$$
\sum_{i=1}^{M}  h_i^2 \,\mathrm{v}_i \,  w_1^{(i)}(0, t) = {\bf d}_1(t),
$$
where
\begin{align}\label{d_1}
 {\bf d}_1(t) :=&  - \frac{1}{\pi}\int_{\Gamma_0} \varphi^{(0)}\big(N_0, \xi,t\big)\, d\sigma_\xi - \frac{1}{\pi}\int_{\Xi^{(0)}} \partial_t {N}_{0}(\xi,t)\, d\xi - \frac{1}{\pi} \sum_{i=1}^{M}\int_{\Xi^{(i)}}  \partial_t \widetilde{N}_{0}(\xi,t) \, d\xi
\notag
\\
 & + \sum_{i=1}^{M}h^2_i \, \partial_{y_1^{(i)} }w_0^{(i)}(0,t)  \Big(1 - \mathrm{v}_i(t)  \,  \int_{2\ell_0}^{3\ell_0} \xi^{(i)}_1\,  \chi^\prime_{\ell_0}(\xi^{(i)}_1) \, d\xi^{(i)}_1 \Big).
\end{align}

Thus, for $\{w^{(i)}_1\}_{i=1}^M$ we get the problem
\begin{equation}\label{prob_w_1}
 \left\{\begin{array}{rcll}
 \partial_t{w}^{(i)}_1 \, + \,  v_i^{(i)}\, \partial_{y^{(i)}_1} w^{(i)}_1 & =& a_i\,  {w}^{(i)}_1 \, + \, f_i  & \text{in} \  I_i \times (0, T),  \ \ i \in \{1,\ldots,M\},
 \\[2mm]
\sum_{i=1}^{M} h_i^2 \,  \mathrm{v}_i \,   w_1^{(i)}(0, t) & = & {\bf d}_1(t), & t \in (0, T),
 \\[2mm]
    w_1^{(i)}(\ell_i,  t) & = & 0, & t \in [0, T],   \ \ i \in \{1,\ldots,m\},
\\[2mm]
    w_1^{(i)}(x_i,  0) & = & 0, & x_i  \in [0, \ell_i],  \ \ i \in \{1,\ldots,M\},
  \end{array}\right.
\end{equation}
where
\begin{equation}\label{a_i}
  a_i(y^{(i)}_1,t) =  - \partial_{y^{(i)}_1}v_i^{(i)}(y^{(i)}_1,t)\, - \,
\partial_s \widehat{\varphi}^{(i)}\big(w^{(i)}_0(y^{(i)}_1,t), y^{(i)}_1, t\big),
\end{equation}
with coefficients  $\widehat{\varphi}^{(i)}$ and $f_i$  determined in \eqref{hat_phi} and \eqref{fun_1}, respectively.

To find a solution to this linear problem, first we consider for each $i\in \{1,\ldots,m\}$ the problem
\begin{equation}\label{prob_w_1_m}
 \left\{\begin{array}{rcll}
 \partial_t{w}^{(i)}_1 \, + \,  v_i^{(i)}\, \partial_{y^{(i)}_1} w^{(i)}_1 & =& a_i\,  {w}^{(i)}_1 \, + \, f_i  & \text{in} \  I_i \times (0, T),
 \\[2mm]
    w_1^{(i)}(\ell_i,  t) & = & 0, & t \in [0, T],
 \\[2mm]
    w_1^{(i)}(x_i,  0) & = & 0, & x_i  \in [0, \ell_i].
  \end{array}\right.
\end{equation}
Since $v_i^{(i)} < 0,$ this problem has a unique weak solution. It will be a classical solution if $f_i(0,0)=0$ holds, which  is obviously fulfilled (see \cite[\S 3.2.1]{Mel-Roh_preprint-2022}).
However, \eqref{res_2} implies that the solution ${w}^{(i)}_1$ must be $C^2$-smooth. This follows from
$\partial_t f_i(0,0)=0$, which is also obviously true.
In addition, for the solution it is possible to obtain the explicit  representation using the method of characteristics (see \eqref{Volt_1} and \eqref{Volt_2} or for more detail \cite[\S 3.2.1]{Mel-Roh_preprint-2022}). From this representation it follows that
$\partial_t w_1^{(i)}(0,0)=\partial^2_{tt} w_1^{(i)}(0,0)=0.$

The other coefficients $\{w^{(i)}_1\}_{i=m+1}^M$ are determined as solutions to the corresponding problems
\begin{equation}\label{prob_w_1_M}
 \left\{\begin{array}{c}
 \partial_t{w}^{(i)}_1 \, + \,  v_i^{(i)}\, \partial_{y^{(i)}_1} w^{(i)}_1 = a_i\,  {w}^{(i)}_1 \, + \, f_i   \ \ \text{in} \  I_i \times (0, T),
 \\[2mm]
    w_1^{(i)}(0,  t)  =    \dfrac{1}{(M-m)\, h_i^2 \, \mathrm{v}_i(t) } \Big({\bf d}_1(t) - \displaystyle{\sum_{i=1}^{m}  h_i^2 \,\mathrm{v}_i(t)  \,  w_1^{(i)}(0, t)}\Big), \ \  t \in [0, T],
 \\[2mm]
    w_1^{(i)}(x_i,  0)  =  0,  \ \ x_i  \in [0, \ell_i],
  \end{array}\right.
\end{equation}
 Inasmuch as $w_1^{(i)}(0,0)= \partial_t w_1^{(i)}(0,0)=0,$ $i\in \{1,\ldots,m\},$ there exists a unique classical solution to the problem \eqref{prob_w_1_M} if ${\bf d}_1(0)=0$ and ${\bf d}'_1(0) =0.$  It is easy to see that ${\bf d}_1(0)=0$ holds thanks to the assumption
$\varphi^{(0)}\big|_{t=0} =0$ and to \eqref{zero_t_0}. The second relation is satisfied if  $\partial_t \varphi^{(0)}\big|_{t=0} =0.$ Since the $C^2$-smoothness is needed for ${w}^{(i)}_1$, we additionally assume that
\begin{equation}\label{add_phi_0}
  \partial_t \varphi^{(0)}\big|_{t=0} = \partial^2_{tt} \varphi^{(0)}\big|_{t=0} = 0.
\end{equation}
An explicit representation of the solution to the problem \eqref{prob_w_1_M} is also possible (see \cite[\S 3.2.1]{Mel-Roh_preprint-2022}).

Consequently, the problem \eqref{prob_w_1} has a classical solution. This means that the solvability condition both for the problem \eqref{tilda_N_1_prob} and the problem  \eqref{eq_2}, \eqref{bc_2} and \eqref{uniq_1}  is satisfied.
 It is easy to verify that
\begin{equation*}%\label{zero_t_0}
  \widetilde{N}_1\big|_{t=0} = N_1\big|_{t=0} =  0, \quad \partial_t N_1\big|_{t=0} = 0, \quad u_2^{(i)}\big|_{t=0} = 0, \quad i\in \{1,\ldots,M\}.
\end{equation*}
In addition, since $\{g_1^{(i)}\}_{i=1}^M$  in the problem \eqref{tilda_N_1_prob} are uniformly bounded with respect to $(\xi, t)\in \Xi\times [0, T]$ and have the compact supports, and
$$
\partial_t {\widetilde{N}}_{0} =  \mathcal{ O}(\exp(-\beta_0\xi_i))
\quad \mbox{as} \ \ \xi_i\to+\infty,  \ \  \xi  \in \Xi^{(i)}  \quad (i=\{1,\ldots,M\},  \  \beta_0 >0),
$$
 the solution to the problem~\eqref{N_1_prob}  has  the following  asymptotics uniform with respect to $t\in [0, T]$:
\begin{equation}\label{rem_exp-decrease+1}
N_1(\xi,t) = w^{(i)}_{1}(0,t)  +  \Psi^{(i)}_{1}(\xi_i,t) + \mathcal{ O}(\exp(-\beta_0\xi_i))
\quad \mbox{as} \ \ \xi_i\to+\infty,  \ \  \xi  \in \Xi^{(i)}.
\end{equation}

%%%%%%%%%%%%%%%
%%%%%%%%%%%%%%%
%
\subsection{Boundary-layer parts of the approximation}\label{subsec_Bound_layer}

The regular ansatzes  $\{\mathcal{U}_{\varepsilon}^{(i)}\}_{i=m+1}^M$ constructed in~\S~\ref{regular_asymptotic}
don't satisfy the boundary conditions at  the bases $\{\Upsilon_{\varepsilon}^{(i)} (\ell_i)\}_{i=m+1}^M$  of the outlet cylinders. Thus, we must determine  boundary layer parts of the approximation compensating the residuals of the regular one at each base $\Upsilon_{\varepsilon}^{(i)} (\ell_i)$ $(i\in\{m+1,\ldots,M\})$  of the corresponding thin cylinders.

We additionally assume that component $v_1^{(i)}$ of the vector-valued function $\overrightarrow{V_\varepsilon}^{(i)}$ is independent of the variable $y_1^{(i)}$ in a neighborhood of $\Upsilon_{\varepsilon}^{(i)}(\ell_i),$  i.e.,
$$
\overrightarrow{V_\varepsilon}^{(i)} = \big(v_1^{(i)}(\ell_3,t), 0, 0 \big)
$$
in a neighborhood of $\Upsilon_{\varepsilon}^{(i)} (\ell_i) \  (i\in\{m+1,\ldots,M\}).$ This is a technical assumption. In the general case, the function  $v_1^{(i)}$ need to be expanded in terms of  Taylor series in a neighborhood of the point $y_1^{(i)}=\ell_i.$

Substituting  \eqref{prim+}  into the differential equation and boundary conditions  of the problem \eqref{probl} in
 a neighborhood the base $\Upsilon_{\varepsilon}^{(i)} (\ell_i)$ of the thin cylinder $\Omega_{\varepsilon}^{(i)}$ and collecting  coefficients at the same powers of~$\varepsilon$, we get the following  problems:
 \begin{equation}\label{prim+probl+0}
 \left\{\begin{array}{rcll}
    \Delta_\eta \Pi_0^{(i)}(\eta,t) +  v_1^{(i)}(\ell_3,t) \partial_{\eta_1}\Pi_0^{(i)}(\eta,t)
  & =    & 0,
   & \quad \eta\in \mathfrak{C}_+^{(i)},
   \\[2mm]
  \partial_{\nu_{\overline{\eta}_1}} \Pi_0^{(i)}(\eta,t) & =
   & 0,
   & \quad \eta\in \partial\mathfrak{C}_+^{(i)} \setminus \Upsilon^{(i)},
   \\[2mm]
  \Pi_0^{(i)}(0, \overline{\eta}_1,t) & =
   & \Phi^{(i)}_0(t),
   & \quad \overline{\eta}_1\in\Upsilon^{(i)},
   \\[2mm]
  \Pi_0^{(i)}(\eta,t) & \to
   & 0,
   & \quad \eta_1\to+\infty,
 \end{array}\right.
\end{equation}
where  $\eta=(\eta_1, \eta_2, \eta_3),$ $\eta_1 = \frac{\ell_i -y_1^{(i)}}{\varepsilon},$ $\overline{\eta}_1 =(\eta_2, \eta_3) = \frac{\overline{y}_1^{(i)}}{\varepsilon},$ $\Phi^{(i)}_0(t) := q_{i}(t) - w_{0}^{(i)}(\ell_i,t),$
$$
\Upsilon^{(i)}:=\big\{(\eta_2,\eta_3)\colon \sqrt{\eta^2_2 + \eta^2_3} < h_i\big\}, \quad
\mathfrak{C}_+^{(i)}:=\big\{\eta \colon \  \overline{\eta}_1\in\Upsilon^{(i)}, \quad \eta_1\in(0,+\infty)\big\};
$$
and
\begin{equation}\label{prim+probl+k}
 \left\{\begin{array}{rcll}
    \Delta_\eta \Pi_k^{(i)}(\eta,t) +  v_1^{(i)}(\ell_3,t) \partial_{\eta_1}\Pi_k^{(i)}(\eta,t)
  & =    & \partial_{t}\Pi_{k-1}^{(i)}(\eta,t),
   &  \eta\in \mathfrak{C}_+^{(i)},
   \\[2mm]
  \partial_{\nu_{\overline{\eta}_1}} \Pi_k^{(i)}(\eta,t) & =
   & 0,
   &  \eta\in \partial\mathfrak{C}_+^{(i)} \setminus \Upsilon^{(i)},
   \\[2mm]
  \Pi_k^{(i)}(0,\overline{\eta}_1,t) & =
   & \Phi^{(i)}_k(t) ,
   &  \overline{\eta}_1\in\Upsilon^{(i)},
   \\[2mm]
  \Pi_k^{(i)}(\eta,t) & \to
   & 0,
   &  \eta_1\to+\infty.\end{array}\right.
\end{equation}
for $k\in \{1, 2\},$ where $\Phi^{(i)}_1(t) =  - w_{1}^{(i)}(\ell_i,t)$ and $\Phi^{(i)}_2(t) =  0.$

Using  the  Fourier method, it is easy to find solutions of \eqref{prim+probl+0} and \eqref{prim+probl+k}, e.g.,
\begin{gather*}%\label{Pi_0}
\Pi_0^{(i)}(\eta_1,t) =
\Phi_0(t) \, e^{- v_1^{(i)}(\ell_3,t) \, \eta_1} ,
\\
\Pi_1^{(i)} = \left(\Phi_1(t) +  \left(\frac{{\Phi}_0 \, \partial_t v_1^{(i)}(\ell_3,t)}{(v_1^{(i)}(\ell_3,t))^2} - \frac{\partial_t {\Phi}_0}{v_1^{(i)}(\ell_3,t)} \right) \eta_1 + \frac{{\Phi}_0\, \partial_t v_1^{(i)}(\ell_3,t)}{2 v_1^{(i)}(\ell_3,t)} \, \eta_1^2\right)   e^{-v_1^{(i)}(\ell_3,t) \, \eta_1}
\end{gather*}
Since $v_1^{(i)}(\ell_3,t) \ge \theta_0 > 0$ for all $t\in [0, T]$ and the factors near $e^{-v_1^{(i)}\, \eta_1}$ are bounded with respect to $t\in [0,T],$
\begin{equation}\label{exp_decay}
  \Pi_k^{(i)}(\eta_1,t) = \mathcal{O}\big(e^{- \frac{\theta_0}{2} \eta_1}\big) \quad \text{as} \quad \eta_1 \to +\infty
\end{equation}
uniformly with respect to $t\in [0, T]$ $(k\in\{0, 1, 2\}, \ i\in\{m+1,\ldots,M\}).$ Obviously,
\begin{equation}\label{in_cond}
  \Pi_0^{(i)}\big|_{t=0} =  \Pi_1^{(i)}\big|_{t=0} = \Pi_2^{(i)}\big|_{t=0} =  0, \quad i\in\{m+1,\ldots,M\}.
\end{equation}

\subsection{Construction of the final approximation in the thin junction $\Omega_\varepsilon$\label{Sec:justification}}
%%%%%%%%%%%%%%%%%%%

We have determined all coefficients of the ansatzes \eqref{regul} - \eqref{junc}.
With the help of the smooth cut-off functions $\chi_{\ell_0}$  (see \eqref{new-solution_k}) and
\begin{equation}\label{cut-off_functions}
\chi_\delta^{(i)}(s) =
    \left\{
    \begin{array}{ll}
        1, & \text{if} \ \ s \ge \ell_i -  \delta,
    \\
        0, & \text{if} \ \ s \le \ell_i - 2\delta,
    \end{array}
    \right.
\quad i \in \{m+1,\ldots,M\},
\end{equation}
we construct the approximation function
\begin{equation}\label{first_app}
\mathfrak{A}_\varepsilon =
\left\{
  \begin{array}{ll}
   \mathcal{U}^{(i)}_\varepsilon(y^{(i)}, t) & \text{in} \  \  \Omega^{(i)}_{\varepsilon,3\ell_0,\gamma}, \ \ i\in\{1,\ldots,m\},
\\[3pt]
   \mathcal{U}^{(i)}_\varepsilon(y^{(i)}, t) + \chi_\delta^{(i)}(y^{(i)}_1) \, \mathcal{B}^{(i)}_\varepsilon(y^{(i)}, t)& \text{in} \ \
    \Omega^{(i)}_{\varepsilon,3\ell_0,\gamma}, \ \ i\in\{m+1,\ldots,M\},
    \\[3pt]
    \mathcal{N}_\varepsilon(x,t) & \text{in} \ \  \Omega^{(0)}_{\varepsilon, \gamma},
\\[3pt]
\chi_{\ell_0}\big(\frac{y^{(i)}_1}{\varepsilon^\gamma}\big)\, \mathcal{U}^{(i)}_\varepsilon +
\Big(1- \chi_{\ell_0}\big(\frac{y^{(i)}_1}{\varepsilon^\gamma}\big)\Big) \mathcal{N}_\varepsilon & \text{in} \ \
\Omega^{(i)}_{\varepsilon,2\ell_0,3\ell_0,\gamma }, \ \ i\in\{1,\ldots,M\},
  \end{array}
\right.
\end{equation}
where $t\in [0,T],$ $\gamma$ is a fixed number from $(\frac23, 1),$
 $\delta$ is a sufficiently small fixed positive number such that $\chi_\delta^{(i)}$ vanishes in the support of $\varphi_\varepsilon^{(i)}$
$(i\in\{m+1,\ldots,M\}),$ and
\begin{gather}
   \Omega^{(0)}_{\varepsilon, \gamma}:=\Omega^{(0)}_\varepsilon \bigcup \Big(\bigcup\limits_{i=1}^M \Omega^{(i)}_\varepsilon \cap \{y^{(i)}\colon y^{(i)}_1\in [\varepsilon \ell_0,  2\ell_0 \varepsilon^\gamma]\}\Big), \notag
   \\
   \Omega^{(i)}_{\varepsilon,2\ell_0,3\ell_0,\gamma }  := \Omega^{(i)}_\varepsilon \, \cap \, \big\{y^{(i)}\colon y^{(i)}_1\in [2\ell_0 \varepsilon^\gamma, 3\ell_0 \varepsilon^\gamma]\big\}, \notag
\\[2pt]
\Omega^{(i)}_{\varepsilon,3\ell_0,\gamma}  := \Omega^{(i)}_\varepsilon \, \cap\,  \big\{y^{(i)}\colon y^{(i)}_1 \in [3\ell_0 \varepsilon^\gamma, \ell_i]\big\}.  \label{note_doms}
\end{gather}

Obviously,
 $$
 \mathfrak{A}_\varepsilon\big|_{t=0} = 0 \quad \text{and} \quad \mathfrak{A}_\varepsilon\big|_{y^{(i)}_1=\ell_i} = q_i(t), \quad i\in \{1,\ldots,M\}.
 $$
In $\Omega^{(i)}_{\varepsilon,3\ell_0,\gamma}$ $( i\in\{1,\ldots,m\})$ the approximation
$\mathfrak{A}_\varepsilon$ satisfies the relations \eqref{ref_1}. For $ i\in\{m+1,\ldots,M\} $ we should additionally calculate  residuals from the boundary-layer:
\begin{multline}\label{Res_3}
    \chi_\delta^{(i)}(y^{(i)}_1) \, \partial_t\mathcal{B}^{(i)}_\varepsilon(y^{(i)}, t) -  \varepsilon\, \Delta_x \big(\chi_\delta^{(i)}(y^{(i)}_1) \, \mathcal{B}^{(i)}_\varepsilon\big) +
   v_1^{(i)}(\ell_3,t) \, \partial_{y^{(i)}_1} \big(\chi_\delta^{(i)}(y^{(i)}_1) \, \mathcal{B}^{(i)}_\varepsilon\big)
          \\
           = \varepsilon^2\, \chi_\delta^{(i)} \, \partial_t \Pi_2^{(i)}
       - \varepsilon \,  \big(\chi_\delta^{(i)}\big)^{\prime\prime}  \sum\limits_{k=0}^{2} \varepsilon^{k} \Pi_k^{(i)}
       + 2  \big(\chi_\delta^{(i)}\big)'  \sum\limits_{k=0}^{2} \varepsilon^{k} \partial_{\eta_1} \Pi_k^{(3)}
  +  v_1^{(i)}  \,  \big(\chi_\delta^{(i)}\big)'  \sum\limits_{k=0}^{2} \varepsilon^{k} \Pi_k^{(i)}.
  \end{multline}
The supports of summands in the second line of \eqref{Res_3} coincide with $\mathrm{supp}\big(\big(\chi_\delta^{(i)}\big)'\big),$ where
the functions $\{\Pi_k^{(i)}\}_{k=0}^{2}$ exponentially decay as $\varepsilon$ tends to zero. Therefore, the right-hand side of the differential equation \eqref{Res_3} has  the order $\varepsilon^2$ for $\varepsilon$ small enough.

In the  neighborhood $\Omega^{(0)}_{\varepsilon, \gamma}$ of  $\Omega^{(0)}_{\varepsilon},$
\begin{equation}\label{Res_6}
    \partial_t\,\mathcal{N}_\varepsilon -  \varepsilon\, \Delta_x  \mathcal{N}_\varepsilon +
  \mathrm{div} \big( \overrightarrow{V_\varepsilon}(x) \, \mathcal{N}_\varepsilon\big)
       = \varepsilon\,  \partial_t N_1,
 \end{equation}
 and on  the boundary condition $\Gamma^{(0)}_\varepsilon$
\begin{equation}\label{Res_7}
-   \partial_{\boldsymbol{\nu}_\varepsilon} \mathcal{N}_\varepsilon  =   \varphi^{(0)}_\varepsilon\big(\mathcal{N}_\varepsilon,x,t) + \varepsilon\, \Phi^{(0)}_\varepsilon(x, t),
 \end{equation}
where $\Phi^{(0)}_\varepsilon(x, t) = \partial_s \varphi^{(0)}_\varepsilon(\theta,x, t) N_1,$ and due to the assumption ${\bf A1}$
the support of  $\Phi^{(0)}_\varepsilon$ lies in the support of   $\varphi^{(0)}_\varepsilon$ and
\begin{equation}\label{Res_5}
  \sup_{\Gamma^{(0)}_{\varepsilon} \times (0,T)} |\Phi^{(0)}_\varepsilon(x, t)| \le C_1.
\end{equation}
 On the other part of the boundary $\partial_{\boldsymbol{\nu}_\varepsilon}\,\mathcal{N}_\varepsilon =0$ since
$\varphi^{(0)}_\varepsilon$ vanishes there and  the functions $\{\varphi^{(i)}_\varepsilon\}_{i=1}^M$ vanish
on $\Gamma^{(i)} _\varepsilon \cap \{y^{(i)}\colon y^{(i)}_1 \in [ \varepsilon \ell_0, 3 \ell_0 \varepsilon^\gamma]\},$ respectively (see ${\bf A2}).$

Owing to   Remark~\ref{r_2_1},
\begin{equation}\label{Res_8}
-   \,  \partial_{\overline{\nu}_\varepsilon} \mathfrak{A}_\varepsilon  = 0
 \end{equation}
on the cylindrical  surface of $\Omega^{(i)}_{\varepsilon,2\ell_0,3\ell_0,\gamma }.$
Based on \eqref{ref_1}, \eqref{res_2} and \eqref{Res_6}, we have
\begin{multline}\label{Res_9}
   \partial_t\, \mathfrak{A}_\varepsilon -  \varepsilon\, \Delta \mathfrak{A}_\varepsilon +
 \mathrm{v}_i \,\partial_{y^{(i)}} \mathfrak{A}_\varepsilon
         =
       - \varepsilon^2\, \chi_{\ell_0}\Big(\frac{y^{(i)}_1}{\varepsilon^\gamma}\Big) \Big( w^{(i)}_{1}(y^{(i)}_1, t)  \Big)^{\prime\prime}
             +  \varepsilon \Big(1 -    \chi_{\ell_0}\Big(\frac{y^{(i)}_1}{\varepsilon^\gamma}\Big)\Big) \partial_t N_1
      \\
      -  \chi_{\ell_0}'' \sum\limits_{k=0}^{1} \varepsilon^{k+1 -2\gamma}    \Big(w_k^{(i)} (y^{(i)}_1, t) - N_k\Big)
      - 2  \chi_{\ell_0}' \sum\limits_{k=0}^{1} \varepsilon^{k+1 -\gamma}  \Big(\partial_{y^{(i)}_1} w_k^{(i)} (y^{(i)}_1, t) - \varepsilon^{-1} \partial_{\xi_1^{(i)}} N_k\Big)
      \\
      + \mathrm{v}_i \, \chi_{\ell_0}' \sum\limits_{k=0}^{1} \varepsilon^{k -\gamma}   \Big(w_k^{(i)} (y^{(i)}_1, t) - N_k\Big) \quad \text{in} \ \ \Omega^{(i)}_{\varepsilon,2\ell_0,3\ell_0,\gamma }.
      \end{multline}
Summands in the second and third lines of \eqref{Res_9} are localized in
the support of  $\big(\chi_{\ell_0}\big)'.$  Therefore, using  the Taylor formula  for the functions $w_{0}^{(i)}$ and $w_{1}^{(i)}$ at the point $y^{(i)}_1=0$ and the formula \eqref{new-solution_k}, and taking into account  \eqref{rem_exp-decrease+0} and \eqref{rem_exp-decrease+1}, these summands can be rewritten as follows
$$
  \chi_{\ell_0}'' \sum\limits_{k=0}^{1} \varepsilon^{k+1 -2\gamma}    \widetilde{N}_k\Big(\frac{x}{\varepsilon},t\Big)   + 2 \chi_{\ell_0}' \sum\limits_{k=0}^{1} \varepsilon^{k -\gamma} \,  \partial_{\xi_i}\widetilde{N}_k(\xi,t)\big|_{\xi=\frac{x}{\varepsilon}}
    - \mathrm{v}_i \, \chi_{\ell_0}' \sum\limits_{k=0}^{1} \varepsilon^{k -\gamma}   \widetilde{N}_k\Big(\frac{x}{\varepsilon},t\Big) + \mathcal{O}(\varepsilon^{{\gamma}}) \quad \text{as} \ \ \varepsilon \to 0.
$$
The  maximum of $|\widetilde{N}_k|$ and $|\partial_{\xi_i}\widetilde{N}_k|$ over
$\Omega^{(i)}_\varepsilon \cap \big\{y^{(i)}\colon y^{(i)}_1\in [2\ell_0 \varepsilon^\gamma, 3\ell_0 \varepsilon^\gamma]\big\}\times [0, T]$
(see \eqref{rem_exp-decrease+0} and \eqref{rem_exp-decrease+1})
are
 of the order $\exp\big(-\beta_0 2 \ell_0 \, \varepsilon^{\gamma -1}\big),$ i.e.,
these terms exponentially decrease as  $\varepsilon$ tends to zero.  Thus,  the right-hand side of \eqref{Res_9} has the order  $\varepsilon^{{\gamma}}.$

Summing-up  calculations in \S~\ref{regular_asymptotic} and in this one, we get the statement.

\begin{lemma}\label{Prop-3-1} There is a positive number $\varepsilon_0$ such that for all $\varepsilon\in (0, \varepsilon_0)$  the difference between the approximation \eqref{first_app} and the solution to the problem \eqref{probl_rewrite} satisfies the following relations for all $t\in (0,T):$
\begin{equation}\label{dif_1}
  \partial_t(\mathfrak{A}_\varepsilon - u_\varepsilon) -  \varepsilon\, \Delta_x \big(\mathfrak{A}_\varepsilon - u_\varepsilon\big) +
   \mathrm{div} \big( \overrightarrow{V_\varepsilon}(x) \, (\mathfrak{A}_\varepsilon - u_\varepsilon)\big)
    =  \varepsilon\,  \mathcal{R}^{(0)}_\varepsilon \quad  \text{in} \ \ \Omega_{\varepsilon}^{(0)},
\end{equation}
\begin{equation}\label{dif_2}
 \partial_t(\mathfrak{A}_\varepsilon - u_\varepsilon) -  \varepsilon\, \Delta_{y^{(i)}} \big(\mathfrak{A}_\varepsilon - u_\varepsilon\big) +
   \mathrm{v}_i \,\partial_{y^{(i)}}(\mathfrak{A}_\varepsilon - u_\varepsilon)  =  \varepsilon^{{\gamma}}\,  \mathcal{R}^{(i)}_{\varepsilon,\ell_0}  \quad
    \text{in} \ \ \Omega_\varepsilon^{(i)}\setminus \Omega^{(i)}_{\varepsilon,3\ell_0,\gamma},
\end{equation}
\begin{equation}\label{dif_3}
  \partial_t(\mathfrak{A}_\varepsilon - u_\varepsilon)   -  \varepsilon\, \Delta_{y^{(i)}} (\mathfrak{A}_\varepsilon - u_\varepsilon)\big) +
  \mathrm{div}_{y^{(i)}} \big( \overrightarrow{V_\varepsilon}^{(i)} \, (\mathfrak{A}_\varepsilon - u_\varepsilon)\big)
   =  \varepsilon^2 \, \mathcal{R}^{(i)}_\varepsilon \quad  \text{in} \ \ \Omega^{(i)}_{\varepsilon,3\ell_0,\gamma},
\end{equation}
\begin{equation}\label{bc_1+}
 -   \partial_{\boldsymbol{\nu}_\varepsilon}\big(\mathfrak{A}_\varepsilon - u_\varepsilon\big)  -   \varphi^{(0)}_\varepsilon\big(\mathfrak{A}_\varepsilon,x,t)  + \varphi^{(0)}_\varepsilon\big(u_\varepsilon,x,t) =  \varepsilon\, \Phi^{(0)}_\varepsilon  \quad
    \text{on} \ \ \Gamma_\varepsilon^{(0)},
\end{equation}
\begin{equation} % double label to be checked \label{bc_2}
  -   \partial_{\boldsymbol{\nu}_\varepsilon}\big(\mathfrak{A}_\varepsilon - u_\varepsilon\big)  =  0  \quad
    \text{on} \ \ \Gamma_\varepsilon^{(i)}, \ \ y^{(i)}_1\in [\ell_0 \varepsilon, 3\ell_0 \varepsilon^\gamma],
\end{equation}
\begin{equation}\label{bc_3}
 -    \partial_{\overline{\nu}_\varepsilon}(\mathfrak{A}_\varepsilon - u_\varepsilon) +  (\mathfrak{A}_\varepsilon - u_\varepsilon) \, \overline{V}^{(i)}_\varepsilon\boldsymbol{\cdot}\overline{\nu}_\varepsilon
 -   \varphi^{(i)}_\varepsilon\big(\mathfrak{A}_\varepsilon, y^{(i)},t)  + \varphi^{(i)}_\varepsilon\big(u_\varepsilon, y^{(i)},t)
  =  \varepsilon^{2} \Phi_\varepsilon^{(i)}
\end{equation}
on $ \Gamma_\varepsilon^{(i)}, \ \ y^{(i)}_1\in [3\ell_0 \varepsilon^\gamma, \ell_i],$
\begin{equation*}
 (\mathfrak{A}_\varepsilon - u_\varepsilon)\big|_{x_i= \ell_i}
 = 0 \quad  \text{on} \  \ \Upsilon_{\varepsilon}^{(i)} (\ell_i),
\  i\in\{1,\ldots,M\}, \qquad  (\mathfrak{A}_\varepsilon - u_\varepsilon)\big|_{t=0}
  = 0 \quad \text{on} \ \ \Omega_{\varepsilon},
\end{equation*}
where  the vector-function $\overline{V}^{(i)}_\varepsilon$ is defined in \eqref{V_i},
\begin{gather}\label{Res_10}
  \sup_{(\Omega^{(i)} _{\varepsilon} \setminus \Omega^{(i)}_{\varepsilon,3\ell_0,\gamma}) \times (0,T)} |\mathcal{R}^{(i)}_{\varepsilon,\ell_0}| +
  \sup_{\Omega^{(i)} _{\varepsilon}\times (0,T)} |\mathcal{R}^{(i)}_\varepsilon| \le C_i, \quad i\in\{1,\ldots,M\},
\\ \label{Res_11}
\sup_{\Omega^{(0)} _{\varepsilon}\times (0,T)} |\mathcal{R}^{(0)}_\varepsilon| \le C_0, \quad  \sup_{\Gamma^{(i)}_{\varepsilon} \times (0,T)} |\Phi_\varepsilon^{(i)}| \le \tilde{C}_i, \quad i\in\{0,\ldots,M\},
\end{gather}
and the support of $\{\Phi_\varepsilon^{(i)}\}_{i=1}^M$  with respect the variable $y^{(i)}_1$  lies in $(\varepsilon \ell_0, \ell_i)$ uniformly in $t\in [0, T]$.
\end{lemma}

\begin{remark}
 In \eqref{Res_10} and \eqref{Res_11} onwards, all constants in inequalities are independent of the parameter~$\varepsilon.$
\end{remark}

%%%%%%%%%%%%%%%

%%%%%%%%%%%%%%%
%
\subsection{The main results}\label{A priori estimates}

Here we prove asymptotic estimates using  maximum principles for solutions of the first and second initial-boundary value problems \cite[Ch. 1, \S 2]{Lad_Sol_Ura_1968}. The novelty lies in finding out how the constants in this estimates will depend on the small parameter $\varepsilon$  and choosing special comparison functions that help to do this. Similar estimates for linear problems were established in our paper \cite{Mel-Roh_preprint-2022}.

\begin{theorem}\label{Th_1} Let the assumptions made in Section~\ref{Sec:Statement}, in Remark~\ref{add_assumptios} and in \eqref{add_phi_0} be satisfied.
Let the  approximation function  $\mathfrak{A}_\varepsilon$ from \eqref{first_app} and the unique solution $u_\varepsilon$ of  the problem~\eqref{probl_rewrite} be given.

Then,  there is a positive number $\varepsilon_0$ such that for all $\varepsilon\in (0, \varepsilon_0)$ the difference $U_\varepsilon := \mathfrak{A}_\varepsilon - u_\varepsilon,$  satisfies the estimate
  \begin{equation}\label{max_1}
  \max_{\overline{\Omega_\varepsilon}\times [0, T]} |U_\varepsilon|  \le C_0 \, \varepsilon^{{\gamma}}.
\end{equation}
 In \eqref{max_1},  the constant $C_0$ depends  on  $T$  nad the constants in \eqref{Res_10} and  \eqref{Res_11}, but not on $\varepsilon.$ Moreover, $\gamma$ is a fixed number from the interval $(\frac23, 1).$
\end{theorem}

\begin{proof}
{\bf 1.} First, let us introduce a new function $\Psi_\varepsilon = U_\varepsilon  \, e^{-\lambda t},$ where the   constant
\begin{equation}\label{e_7}
  \lambda = 1+ \max_{i\in \{1,\ldots,M\}} \Big( \max_{[0, \ell_i]\times [0,T]}
 \Big|\frac{\partial v_1^{(i)}}{\partial y^{(i)}_1}\Big| + \max_{\mathfrak{X}_i} \Big| \mathrm{div}_{\bar{\xi}_1}  \overline{V}^{(i)}(y^{(i)}_1, \overline{\xi}_1, t)\Big|\,
    \Big) \ge 1,
\end{equation}
and the vector-valued function $\overline{V}^{(i)}$ with the domain $\mathfrak{X}_i= \{y^{(i)}_1 \in [0, \ell_i], \  |\overline{\xi}_1| \le h_i, \  t\in[0, T]\}$ is defined in \eqref{overline_V}.

From \eqref{dif_1} -- \eqref{dif_3} it follows that $\Psi_\varepsilon$ satisfies for all $t\in (0, T)$  the differential equations
\begin{equation}\label{eq_psi1}
  \partial_t \Psi_\varepsilon -  \varepsilon\, \Delta \Psi_\varepsilon +
   \overrightarrow{V_\varepsilon}^{(0)} \cdot \nabla \Psi_\varepsilon + \lambda \Psi_\varepsilon  =  \varepsilon\, e^{-\lambda t}\,  \mathcal{R}^{(0)}_\varepsilon  \quad  \text{in} \ \ \Omega_\varepsilon^{(0)},
\end{equation}
\begin{equation}\label{eq_psi2-}
\partial_t \Psi_\varepsilon -  \varepsilon\, \Delta\Psi_\varepsilon +
   \mathrm{v}_i \,\partial_{y^{(i)}_1}\Psi_\varepsilon  + \lambda \Psi_\varepsilon =  \varepsilon\,  e^{-\lambda t}\, \mathcal{R}^{(i)}_{\varepsilon,\ell_0}  \quad     \text{in} \ \  \Omega^{(i)}_\varepsilon \cap \{y^{(i)}\colon y^{(i)}_1\in [\varepsilon \ell_0,  2\ell_0 \varepsilon^\gamma]\},
\end{equation}
\begin{equation}\label{eq_psi2}
\partial_t \Psi_\varepsilon -  \varepsilon\, \Delta\Psi_\varepsilon +
   \mathrm{v}_i \,\partial_{y^{(i)}_1}\Psi_\varepsilon  + \lambda \Psi_\varepsilon =  \varepsilon^{{\gamma}}\,  e^{-\lambda t}\, \mathcal{R}^{(i)}_{\varepsilon,\ell_0}  \quad     \text{in} \ \  \Omega^{(i)}_{\varepsilon,2\ell_0,3\ell_0,\gamma },
\end{equation}
\begin{equation}\label{eq_psi3}
  \partial_t\Psi_\varepsilon   -  \varepsilon\, \Delta\Psi_\varepsilon +
  \overrightarrow{V_\varepsilon}^{(i)}\cdot \nabla \Psi_\varepsilon
  +\Big( \partial_{y^{(i)}_1} v_1^{(i)} + \mathrm{div}_{\bar{\xi}_1}  \overline{V}^{(i)} +\lambda\Big) \Psi_\varepsilon =  \varepsilon^2 \, e^{-\lambda t}\,  \mathcal{R}^{(i)}_\varepsilon \quad  \text{in} \ \ \Omega^{(i)}_{\varepsilon,3\ell_0,\gamma},
\end{equation}

Consider any $t_1 \in (0, T).$ Three cases are possible: $a) \  \max_{\overline{\Omega_\varepsilon}\times [0, t_1]} \Psi_\varepsilon  \le  0,$
$$
  b) \  \ 0 < \max_{\overline{\Omega_\varepsilon}\times [0, t_1]} \Psi_\varepsilon  \le \max_{\overline{\Gamma_\varepsilon}\times (0, t_1]} \Psi_\varepsilon, \quad
 c) \ \ 0 < \max_{\overline{\Omega_\varepsilon}\times [0, t_1]} \Psi_\varepsilon  \le \Psi_\varepsilon(x^0, t_0),
 $$
 where $\Gamma_\varepsilon = \cup_{i=0}^M \Gamma_\varepsilon^{(i)}$   and the point $P_0 = (x^0, t_0)\in \Omega_\varepsilon \times (0, t_1].$

 In the third case, when the positive maximum of  $\Psi_\varepsilon$ is  reached at the point  $P_0$, the relations
\begin{equation}\label{e_5}
  \partial_t \Psi_\varepsilon \ge 0, \qquad \nabla \Psi_\varepsilon = \vec{0}, \qquad - \Delta  \Psi_\varepsilon \ge 0
\end{equation}
are satisfied at $P_0.$ Therefore, it follows from \eqref{e_7} -- \eqref{eq_psi3}  that
\begin{equation}\label{e_6}
   \Psi_\varepsilon\big|_{P_0}  \le   \varepsilon^{{\gamma}} \, \Big(\sum_{i=0}^{M}  \max_{\overline{\Omega^{(i)}_\varepsilon}\times [0,T]} |\mathcal{R}^{(i)}_\varepsilon|
   + \sum_{i=1}^{M} \max_{\overline{\Omega^{(i)}_\varepsilon}\times [0,T]} |\mathcal{R}^{(i)}_{\varepsilon,\ell_0}|\Big) .
\end{equation}
Similarly, we consider  the point of the smallest non-positive value of the function $\Psi_\varepsilon,$ which is reached  at a point in
$\Omega_\varepsilon \times (0, t_1].$ As a result, taking \eqref{Res_10} into account, we get the estimate
 \begin{equation}\label{e_8}
   \max_{\overline{\Omega_\varepsilon}\times [0, T]} |U_\varepsilon|  \le  e^{\lambda T} \, C \, \varepsilon^{{\gamma}}.
 \end{equation}

{\bf 2.} Now consider the case $b)$, i.e., the function $\Psi_\varepsilon$ takes the largest positive value at a  point $P_1$ belonging to $\Gamma_\varepsilon \times (0, t_1].$ Using the mean value theorem, the boundary conditions \eqref{bc_1+} and \eqref{bc_3} for $\Psi_\varepsilon$ can be rewritten in the form
\begin{equation}\label{eq_4}
   \partial_{\boldsymbol{\nu}_\varepsilon}\Psi_\varepsilon  +  \partial_s \varphi^{(0)}_\varepsilon(\theta_0,x,t) \, \Psi_\varepsilon = - \varepsilon\, e^{-\lambda t}\, \Phi^{(0)}_\varepsilon  \quad
    \text{on} \ \ \Gamma_\varepsilon^{(0)},
\end{equation}
and
\begin{equation}\label{eq_5}
  \partial_{\boldsymbol{\nu}_\varepsilon}\Psi_\varepsilon + \big( \partial_s \varphi^{(i)}_\varepsilon(\theta_i,x,t)  - \overline{V}^{(i)} \boldsymbol{\cdot}\overline{\nu}_\varepsilon \big)\, \Psi_\varepsilon =
   - \varepsilon^{2} e^{-\lambda t}\,  \Phi_\varepsilon^{(i)} \quad  \text{on} \  \Gamma_\varepsilon^{(i)},
\end{equation}
where the values $\{\theta_i\}$ depend on $\Psi_\varepsilon.$

First, we suppose  $P_1 \in \Gamma_\varepsilon^{(i)}\times (0, T)$ for some fixed index $i\in \{1,\ldots,M\}$ and
introduce a new function $\mathcal{K}_\varepsilon = \dfrac{\Psi_\varepsilon}{z_\varepsilon},$ where
\begin{equation*}
  z_\varepsilon(\overline{y}^{(i)}_1) = 1 + \varepsilon\, \frac{\varpi_i \, h_i}{2 } \Bigg( 1 - \Big(\frac{y^{(i)}_2}{\varepsilon h_i}\Big)^2 - \Big(\frac{y^{(i)}_3}{\varepsilon h_i}\Big)^2\Bigg), \quad {y}^{(i)} \in \overline{\Omega^{(i)}_\varepsilon},
\end{equation*}
with the constant
\begin{equation}\label{m_1}
 \varpi_i  = 1 +  \max_{X_i} \big|\partial_s \varphi^{(i)}(s,y^{(i)},t)\big| + \max_{\mathfrak{X}_i}\big|\overline{V}^{(i)}\big(y^{(i)}_1, \overline{\xi}_1, t\big)\big| \ge 1.
\end{equation}
Here $X_i$ is the domain of the function $\varphi^{(i)}$ (see {\bf A2}).
It is easy to verify that
\begin{equation}\label{z_1}
  z_\varepsilon\big|_{\Gamma^{(i)}_\varepsilon} =1, \qquad - \partial_{\overline{\nu}_\varepsilon} z_\varepsilon\big|_{\Gamma^{(i)}_\varepsilon} = \varpi_i, \qquad
  1 < z_\varepsilon \le 1  +  \tfrac{1}{2 }\, \varepsilon \, \varpi_i \, h_i \quad \text{in} \ \ \Omega^{(i)}_\varepsilon,
  \end{equation}
  \begin{equation}\label{z_2}
   |\nabla_{\overline{y}^{(i)}_1 } z_\varepsilon|^2  \le \varpi_i^2
   \quad \text{and} \quad
  \Delta_{\overline{y}^{(i)}_1 } z_\varepsilon  =  -\varepsilon^{-1} \frac{2 \varpi_i }{h_i} \quad \text{in} \ \ \Omega^{(i)}_\varepsilon .
  \end{equation}

After the substitution, the function $\mathcal{K}_\varepsilon$ satisfies the differential equation
\begin{equation}\label{eta_1}
  \partial_t \mathcal{K}_\varepsilon - \varepsilon\, \Delta \mathcal{K}_\varepsilon + \Big(\frac{2 \varepsilon\, \nabla_{\overline{y}^{(i)}_1} z_\varepsilon}{z_\varepsilon} + \overrightarrow{V_\varepsilon}^{(i)}\Big)\cdot \nabla \mathcal{K}_\varepsilon
\end{equation}
$$
  +
  \Bigg(\frac{2 \varpi_i}{h_i z_\varepsilon} - \frac{2 \varepsilon\, |\nabla_{\overline{y}^{(i)}_1} z_\varepsilon|^2}{z^2_\varepsilon}
    - \frac{ \varepsilon \, \overline{V}^{(i)}_\varepsilon \boldsymbol{\cdot} \nabla_{\overline{y}^{(i)}_1} z_\varepsilon}{z_\varepsilon}
   +  \frac{\partial v_1^{(i)}}{\partial y^{(i)}_1}+ \mathrm{div}_{\bar{\xi}_1}  \overline{V}^{(i)}(y^{(i)}_1, \bar{\xi}_1, t)\big|_{\bar{\xi}_1=\frac{\bar{y}^{(i)}_1}{\varepsilon}} +  \lambda   \Bigg) \mathcal{K}_\varepsilon
$$
$$
    =
    \left\{
      \begin{array}{ll}
        z_\varepsilon \, \varepsilon^2 \, e^{-\lambda t}\,  \mathcal{R}^{(i)}_\varepsilon, & \hbox{in}  \ \ \Omega^{(i)}_{\varepsilon,3\ell_0,\gamma}\times (0, T);
\\[4pt]
        z_\varepsilon\, \varepsilon^{{\gamma}}\,  e^{-\lambda t}\, \mathcal{R}^{(i)}_{\varepsilon,\ell_0}   & \hbox{in} \ \
 \big(\Omega_\varepsilon^{(i)}\setminus \Omega^{(i)}_{\varepsilon,3\ell_0,\gamma}\big) \times (0, T),
      \end{array}
    \right.
$$
and, by virtue of the first two equalities in \eqref{z_1}, the boundary conditions
\begin{equation}\label{eq_6}
  \partial_{\boldsymbol{\nu}_\varepsilon}\mathcal{K}_\varepsilon + \big( \varpi_i + \partial_s \varphi^{(i)}_\varepsilon(\theta_i,y^{(i)},t)  - \overline{V}^{(i)} \boldsymbol{\cdot}\overline{\nu}_\varepsilon \big)\, \mathcal{K}_\varepsilon =
   - \varepsilon^{2} e^{-\lambda t}\,  \Phi_\varepsilon^{(i)}
\end{equation}
on \ $\Gamma_\varepsilon^{(i)}, \ y^{(i)}_1\in [3\ell_0 \varepsilon^\gamma, \ell_i],$ and
\begin{equation}\label{eq_6+}
  \partial_{\boldsymbol{\nu}_\varepsilon}\mathcal{K}_\varepsilon +  \varpi_i \, \mathcal{K}_\varepsilon =
   0  \quad  \text{on} \  \Gamma_\varepsilon^{(i)}, \ y^{(i)}_1\in [\ell_0 \varepsilon,  3\ell_0 \varepsilon^\gamma].
\end{equation}
Due to the inequalities in \eqref{z_1} and \eqref{z_2}  we can choose the constant $\lambda$ independently of $\varepsilon$ in such a way that the coefficient at $\mathcal{K}_\varepsilon$ in \eqref{eta_1} is bounded from  below by a positive constant, and thanks to our  choice  of the constant $\varpi_i$ (see \eqref{m_1}),  in \eqref{eq_6} we have
\begin{equation}\label{eq_6++}
 \varpi_i + \partial_s \varphi^{(i)}_\varepsilon(\theta_i,y^{(i)},t)  - \overline{V}^{(i)} \boldsymbol{\cdot}\overline{\nu}_\varepsilon  \ge 1.
\end{equation}

By virtue of the first and third relation in \eqref{z_1}, the function $\mathcal{K}_\varepsilon$ takes in $\overline{\Omega}^{(i)}_\varepsilon$  the largest positive value also at the  point $P_1$. Therefore, $\partial_{\boldsymbol{\nu}_\varepsilon}\mathcal{K}_\varepsilon\big|_{P_1}  \ge 0$
 and  from \eqref{eq_6+} it follows that  $P_1$ cannot lie on $\Gamma_\varepsilon^{(i)}, \ y^{(i)}_1\in [\ell_0 \varepsilon,  3\ell_0 \varepsilon^\gamma],$ and from \eqref{eq_6} and \eqref{eq_6++} we get
$$
 \mathcal{K}_\varepsilon(P_1) \le -  \varepsilon^{2} \big( e^{-\lambda t}\,  \Phi_\varepsilon^{(i)}\big)\big|_{P_1} \ \ \Longrightarrow \ \
 \Psi_\varepsilon(P_1) \le -  \varepsilon^{2} \big( e^{-\lambda t}\,  \Phi_\varepsilon^{(i)}\big)\big|_{P_1}.
$$

It should be noted that  the point $P_1$ cannot  lie  on circular strips of  the lateral surfaces of the thin cylinders near their bases
$\{\Upsilon_{\varepsilon}^{(i)} (\ell_i)\}_{i=1}^M$ and $\{\Upsilon_{\varepsilon}^{(i)} (\varepsilon\,\ell_0)\}_{i=1}^M$ for any $t\in \times [0, T],$ since  $\{\overline{V}^{(i)}_\varepsilon\},$ $\{\varphi^{(i)}_\varepsilon\}$ and  $\{\Phi^{(i)}_\varepsilon\}$ vanish there (see the assumptions {\bf A1} and {\bf A2}) and the condition \eqref{eq_6+} holds.

In the case when  $\Psi_\varepsilon$ reaches its smallest negative value  at a point
$P_2 \in  \Gamma_\varepsilon^{(i)}\times (0, T)$ for some fixed index $i\in \{1,\ldots,M\}$ we should consider
the function  $- \Psi_\varepsilon$ and repeat the previous argumentations.
As the result, we get
\begin{equation}\label{max_psi}
  \max_{\overline{\Omega_\varepsilon}\times [0, t_1]} |\Psi_\varepsilon|  \le \max_{\cup_{i=1}^M\overline{\Gamma^{(i)}_\varepsilon}\times (0, t_1]} |\Psi_\varepsilon| \le \varepsilon^2 \, \sum_{i=1}^{M}\max_{\overline{\Gamma^{(i)}_\varepsilon}\times (0, T]} |\Phi^{(i)}_\varepsilon(y^{(i)}, t)|.
\end{equation}

Now it remains to consider the case when the function $\Psi_\varepsilon$ takes the largest positive value at a  point $P_2$ lying  on $\Gamma^{(0)}_\varepsilon \times (0, t_1],$ more precisely, at the support of  $\varphi_\varepsilon^{(0)}$ since  its normal derivative vanishes on the other part of  $\Gamma^{(0)}_\varepsilon$ (see \eqref{Res_7}).
 We again introduce a new function $\mathcal{K}^{(0)}_\varepsilon = \dfrac{\Psi_\varepsilon}{z^{(0)}_\varepsilon},$ where
\begin{equation*}
  z^{(0)}_\varepsilon(x) = 1 + \varepsilon\, \frac{\varpi_0 \, h_0}{2 } \Bigg( 1 - \sum_{k =1}^{3}\frac{(x_k - \varepsilon a_k)^2}{\varepsilon^2 h_0^2}\Bigg), \quad x \in \overline{B}_{\varepsilon h_0}(A_\varepsilon) \subset \overline{\Omega}_\varepsilon^{(0)}.
\end{equation*}
Here,  $B_{\varepsilon h_0}(A_\varepsilon)$ is a ball of radius $\varepsilon h_0$ centered at a point  $A_\varepsilon= \varepsilon (a_1, a_2, a_2)$ such that the $x$-coordinates of the point $P_2$ belong to $\partial B_{\varepsilon h_0}(A_\varepsilon)$. The constant is given by
\begin{equation}\label{m_0}
 \varpi_0  = 1 +  \max_{X_0} \big|\partial_s \varphi^{(0)}(s,x,t)\big| \ge 1,
\end{equation}
where $X_0$ is the domain of the function $\varphi^{(0)}$ (see {\bf A1}). It is easy to check that the function $z^{(0)}_\varepsilon$ satisfies  similar relations as in \eqref{z_1} and \eqref{z_2}, but already on the boundary of the ball $B_{\varepsilon h_0}(A_\varepsilon)$ and inside it, respectively.

The function $\mathcal{K}^{(0)}_\varepsilon$ satisfies in the ball a differential equation of the same type as $\mathcal{K}_\varepsilon$ (see \eqref{eta_1}) and  the constant $\lambda$ can be chosen independently of $\varepsilon$ in such a way that the coefficient at $\mathcal{K}^{(0)}_\varepsilon$ in this differential equation would be bounded below by a positive constant.
Obviously,  $\mathcal{K}^{(0)}_\varepsilon$  takes on the largest positive value at the point $P_2.$
Then, on the one hand, $\partial_{\boldsymbol{\nu}_\varepsilon} \mathcal{K}^{(0)}_\varepsilon\big|_{P_2} \ge 0,$ on the other hand, thanks to~\eqref{eq_4} we have
\begin{equation}\label{eq_7}
  \partial_{\boldsymbol{\nu}_\varepsilon}\mathcal{K}^{(0)}_\varepsilon\big|_{P_2} + \big( \varpi_0 + \partial_s \varphi^{(0)}_\varepsilon(\theta_0,x,t) \big)\, \mathcal{K}^{(0)}_\varepsilon\big|_{P_2} =
   - \varepsilon\, e^{-\lambda t}\,  \Phi_\varepsilon^{(0)}\big|_{P_2},
\end{equation}
whence
$$
 \mathcal{K}^{(0)}_\varepsilon(P_2) \le -  \varepsilon \, \big( e^{-\lambda t}\,  \Phi_\varepsilon^{(0)}\big)\big|_{P_2} \ \ \Longrightarrow \ \
 \Psi_\varepsilon(P_2) \le -  \varepsilon\,  \big(e^{-\lambda t}\,  \Phi_\varepsilon^{(0)}\big)\big|_{P_2}.
$$
Similarly, we consider the case when  $\Psi_\varepsilon$ reaches its smallest negative value  at a point
$P_3 \in  \Gamma_\varepsilon^{(0)}\times (0, T)$ and obtain
$$
\max_{\overline{\Omega_\varepsilon}\times [0, t_1]} |\Psi_\varepsilon|  \le \max_{\overline{\Gamma^{(0)}_\varepsilon}\times (0, t_1]} |\Psi_\varepsilon|
\le \varepsilon \, \max_{\overline{\Gamma^{(0)}_\varepsilon}\times (0, T]} |\Phi_\varepsilon^{(0)}(x, t)|.
$$

Thus, we have proved the estimate \eqref{e_8} for all cases, however, the constant $\lambda$ depends additionally on $\{\varpi_i\}_{i=0}^M$ (see \eqref{m_1} and  \eqref{m_0}) and the constant $C$ depends on constants in  \eqref{Res_10} and \eqref{Res_11}.
\end{proof}

To explain the result of Theorem \ref{Th_1} we note that it  follows from \eqref{max_1} that terms of the order $\mathcal{O}(\varepsilon)$ are redundant in the
approximation $\mathfrak{A}_\varepsilon$.  The asymptotic estimate  can therefore be properly  re-formulated
in terms of the zeroth-order  approximation
\begin{equation}\label{zero_app}
\mathfrak{A}_{0,\varepsilon} =
\left\{
  \begin{array}{ll}
   w_0^{(i)}(y^{(i)}_1,t) & \text{in} \   \Omega^{(i)}_{\varepsilon,3\ell_0,\gamma},
   \\ & i\in\{1,\ldots,m\},
\\
   w_0^{(i)}(y^{(i)}_1,t) + \chi_\delta^{(i)}(y^{(i)}_1) \, \Pi_0^{(i)}
    \left(\frac{\ell_i - y^{(i)}_1}{\varepsilon}, \frac{y^{(i)}_2}{\varepsilon}, \frac{y^{(i)}_3}{\varepsilon}, t\right) & \text{in}  \
    \Omega^{(i)}_{\varepsilon,3\ell_0,\gamma},
    \\
    & i\in\{m+1,...,M\},
    \\[2pt]
    N_0\left(\frac{x}{\varepsilon},t\right) & \text{in} \   \Omega^{(0)}_{\varepsilon, \gamma},
\\[2pt]
\chi_{\ell_0}\big(\frac{y^{(i)}_1}{\varepsilon^\gamma}\big)\, w_0^{(i)}(y^{(i)}_1,t) +
\big(1- \chi_{\ell_0}\big(\frac{y^{(i)}_1}{\varepsilon^\gamma}\big)\big)N_0\left(\frac{\mathbb{A}_i^{-1}y^{(i)}}{\varepsilon},t\right) & \text{in} \
\Omega^{(i)}_{\varepsilon,2\ell_0,3\ell_0,\gamma },
\\ & \ i\in\{1,\ldots,M\},
  \end{array}
\right.
\end{equation}
directly from \eqref{max_1}. Here  $t\in [0,T],$ $\gamma$ is a fixed number from $(\frac23, 1),$
 $\delta$ is a small fixed positive number such that $\chi_\delta^{(i)}$ vanishes in the support of $\varphi_\varepsilon^{(i)}$
$(i\in\{m+1,\ldots,M\}).$ As a result, we have the following statement.

\begin{corollary}\label{corol_1}
  There exist constants $C_0>0$ and $\varepsilon_0>0$ such that for all $\varepsilon\in(0, \varepsilon_0)$ the  estimate
 \begin{equation}\label{main_1}
   \max_{\overline{\Omega_\varepsilon}\times [0, T]} |\mathfrak{A}_{0,\varepsilon} - u_\varepsilon |  \le C_0 \, \varepsilon^{{\gamma}}
 \end{equation}
holds, where $\gamma$ is a fixed number from the interval $(\frac23, 1).$
\end{corollary}

Depending on  the number $m$ that fixes the  number of inlet/outlet cylinders,
the functions   $\{w^{i}_0\}_{i=1}^M$   in   $\mathfrak{A}_{0,\varepsilon}$
satisfy the corresponding one-dimensional limit problems from \S~\ref{limitproblemr}; the boundary-layer part $\Pi^{i}_0$ solves the problem \eqref{prim+probl+0} in the semi-infinite cylinder $\mathfrak{C}_+^{(i)}$  (independent on $\varepsilon$), which is coupled to  $w^{i}_0$ by the boundary condition in \eqref{prim+probl+0}; and the node-layer part $N_0$ solves the problem \eqref{N_0_prob} in the  (unit)  three-dimensional junction (independent on $\varepsilon$), which is related to the functions $\{w^{i}_0\}_{i=1}^M$ by the corresponding conditions at infinity.  In this way, we can provide a closed problem for the entire limit behaviour.

As the next step  we prove  an   estimate on the gradient
$\nabla U_\varepsilon.$  For this, we additionally assume that
\begin{equation}\label{zero_2}
\partial_s \varphi^{(i)}(s,y^{(i)},t) \ge 0 \quad \text{in}  \ \ X_i, \quad \text{for} \ \  i\in \{1,\ldots,M\}.
\end{equation}

\begin{theorem}\label{apriory_estimate} Let the assumptions of Theorem \ref{Th_1} and the inequalities from  \eqref{zero_2} be satisfied.\
Then, for the difference $U_\varepsilon = \mathfrak{A}_\varepsilon - u_\varepsilon$ the estimate
\begin{equation}\label{app_estimate}
\tfrac{1}{\sqrt{|\Omega_\varepsilon|}}\,\| {\nabla_x U_\varepsilon}\|_{L^2(\Omega_\varepsilon\times (0, T))} \le
 C_2\, \varepsilon^{\gamma - \frac{1}{2}}
  \end{equation}
holds,  where  $|\Omega_\varepsilon|$ is  the Lebesgue measure of $\Omega_\varepsilon$.
%$\gamma$ is a fixed number from  $(\frac23, 1).$
\end{theorem}

Error estimates in the $L^2$-norm for thin domains must be written in the rescaled form, namely,  divided by the square of the volume of the corresponding thin domain. Obviously, $|\Omega_\varepsilon|$ has the order $\varepsilon^2.$

\begin{proof} We multiply the differential equations \eqref{dif_1}--\eqref{dif_3}  with  $U_\varepsilon$ and integrate them over the corresponding domain and over $ (0, \tau),$ where $\tau$ is an arbitrary number from $(0, T).$ Integrating by parts and taking \eqref{lap_1}--\eqref{grad_1} and the boundary conditions and initial condition into account, we get
 \begin{multline*}
  \frac{1}{2} \int_{\Omega^{(0)}_\varepsilon}U_\varepsilon^2\big|_{t=\tau}\, dx +
  \frac{1}{2}\sum_{i=1}^{M}\int_{\Omega^{(i)}_\varepsilon} U_\varepsilon^2\big|_{t=\tau}\, dy^{(i)}
  + \varepsilon \, \| {\nabla U_\varepsilon}\|^2_{L^2(\Omega_\varepsilon\times (0, \tau))}
\\
 + \sum_{i=1}^{M} \int\limits_{0}^{\tau}\bigg(\int_{\Gamma_\varepsilon^{(i)}}\partial_s \varphi^{(i)}_\varepsilon(\theta_i,y^{(i)},t) \, U^2_\varepsilon \, dS_{y^{(i)}}\bigg) dt
 \\
  =  \int\limits_{0}^{\tau} \bigg(\int_{\Omega^{(0)}_\varepsilon} U_\varepsilon\,
    \overrightarrow{V_\varepsilon} \cdot
    {\nabla U_\varepsilon} \, dx + \sum_{i=1}^{M}\int_{\Omega^{(i)}_\varepsilon} U_\varepsilon\,
    \overrightarrow{V_\varepsilon} \cdot
    {\nabla U_\varepsilon} \, dy^{(i)}\bigg) dt
    \\
    +
       \int\limits_{0}^{\tau}\bigg(
\varepsilon \int\limits_{\Omega_\varepsilon^{(0)}}  \mathcal{R}^{(0)}_\varepsilon \, U_\varepsilon  dx
+ \varepsilon^{{\gamma}} \sum_{i=1}^{M} \int\limits_{\Omega_\varepsilon^{(i)}\setminus \Omega^{(i)}_{\varepsilon,3\ell_0,\gamma}}  \mathcal{R}^{(i)}_{\varepsilon,\ell_0} \, U_\varepsilon dy^{(i)} + \varepsilon^2 \sum_{i=1}^{M} \int\limits_{ \Omega^{(i)}_{\varepsilon,3\ell_0,\gamma}}  \mathcal{R}^{(i)}_{\varepsilon} \, U_\varepsilon dy^{(i)}\bigg)dt
\\
 -  \int\limits_{0}^{\tau}\bigg(\int_{\Gamma_\varepsilon^{(0)}}\Big(\partial_s \varphi^{(0)}_\varepsilon(\theta_0,x,t) \, U_\varepsilon + \varepsilon\, \Phi^{(0)}_\varepsilon\Big) \, U_\varepsilon \, dS_x
+ \varepsilon^2 \sum_{i=1}^{M} \int_{\Gamma_\varepsilon^{(i)}}  \Phi^{(i)}_\varepsilon \, U_\varepsilon \, dS_{y^{(i)}}\bigg) dt.
  \end{multline*}
Taking  \eqref{zero_2}, \eqref{Res_10}, \eqref{Res_11} and \eqref{max_1}  into account,  we deduce from the previous equality that
\begin{multline}\label{s_1}
  \varepsilon \, \| {\nabla U_\varepsilon}\|^2_{L^2(\Omega_\varepsilon\times (0, \tau))}
\le  \varepsilon \sum_{i=1}^{M}\int\limits_{0}^{\tau} \bigg(  \int_{\Omega^{(i)}_\varepsilon} U_\varepsilon\,
    \overline{V}_\varepsilon^{(i)} \cdot
    {\nabla_{\overline{y}^{(i)}_1} U_\varepsilon} \, dy^{(i)}\bigg) dt
 + C_4 \, \varepsilon^{2 + 2 \gamma}
\\
 + \frac{1}{2}  \int\limits_{0}^{\tau} \bigg( \int_{\Omega^{(0)}_\varepsilon}     \overrightarrow{V_\varepsilon} \cdot
    \nabla\big( U^2_\varepsilon\big) \, dx + \sum_{i=1}^{M}\int_{\Omega^{(i)}_\varepsilon}
    {v}^{(i)}_1\big(y_1^{(i)},t\big) \, \partial_{y^{(i)}_1}\big(U^2_\varepsilon\big) \, dy^{(i)}\bigg) dt.
   \end{multline}
Owing to \eqref{V_1},  the incompressibleness  of $\overrightarrow{V_\varepsilon}^{(0)}$  in $\Omega^{(0)}_\varepsilon $
and the Dirichlet conditions for $U_\varepsilon$ on $\{\Upsilon_{\varepsilon}^{(i)} (\ell_i)\}_{i=1}^M,$
\begin{equation}\label{incom}
  \int\limits_{\Omega^{(0)}_\varepsilon}     \overrightarrow{V_\varepsilon} \cdot
    \nabla\big( U^2_\varepsilon\big) \, dx + \sum_{i=1}^{M}\int\limits_{\Omega^{(i)}_\varepsilon}
    {v}^{(i)}_1\big(y_1^{(i)},t\big) \, \partial_{y^{(i)}_1}\big(U^2_\varepsilon\big) \, dy^{(i)} =
- \sum_{i=1}^{M}\int\limits_{\Omega^{(i)}_\varepsilon}
    U^2_\varepsilon  \, \partial_{y^{(i)}_1}\big( {v}^{(i)}_1\big) \, dy^{(i)} .
\end{equation}
Thanks to the boundedness of $\partial_{y^{(i)}_1}\big( {v}^{(i)}_1\big)$ (see \S~\ref{subsec_V}) and again using \eqref{max_1}, we derive from \eqref{s_1} and \eqref{incom} the inequality
\begin{align*}%\label{ineq_1}
  \| {\nabla U_\varepsilon}\|^2_{L^2(\Omega_\varepsilon\times (0, \tau))} \le & \ C_5 \, \varepsilon^{1+ 2\gamma} + C_6 \, \varepsilon^{1+\gamma} \, \| {\nabla U_\varepsilon}\|_{L^2(\Omega_\varepsilon\times (0, \tau))} + C_4 \, \varepsilon^3
  \\
  \le & \ C_7 \, \varepsilon^{1+ 2\gamma} + C_8  \varepsilon^{2+ 2\gamma} + \frac{\| {\nabla U_\varepsilon}\|^2_{L^2(\Omega_\varepsilon\times (0, \tau))}}{2},
\end{align*}
whence we get \eqref{app_estimate}.
\end{proof}

It follows from \eqref{app_estimate} that terms of the order $\mathcal{O}(\varepsilon^2)$ are redundant in the approximation $\mathfrak{A}_\varepsilon$. Denote by $\mathfrak{A}_{1,\varepsilon}$ the approximation function $\mathfrak{A}_\varepsilon$ without the terms $ \{\varepsilon^2 u_2^{(i)}\}_{i=1}^M$ and  $\{\varepsilon^2 \Pi_2^{(i)}\}_{i=1}^M.$
\begin{corollary}\label{corol_2}
  There exist constants $C_3>0$ and $\varepsilon_0>0$ such that for all $\varepsilon\in(0, \varepsilon_0)$ the  estimate
\begin{equation}\label{app_estimate_1}
\tfrac{1}{\sqrt{|\Omega_\varepsilon|}}\,\| \nabla_x \mathfrak{A}_{1,\varepsilon} - \nabla_x u_\varepsilon\|_{L^2(\Omega_\varepsilon\times (0, T))} \le
 C_3\, \varepsilon^{\gamma - \frac{1}{2}}
  \end{equation}
 holds, where $\gamma$ is a fixed number from $(\frac23, 1).$
\end{corollary}
%%%%%%%%%%%%%%%%%%%%%

%%%%%%%%%%%%%%%%%%%%%

\subsection{General networks}\label{explanation}
To explain how to apply the approximation results  to a general network, consider a network having two nodes
$\Omega^{(0)}_\varepsilon$ and $B^{(0)}_\varepsilon$ (see Fig.~\ref{f4}); the dynamic of the
convective vector field $\overrightarrow{V_\varepsilon}$ is shown by arrows in this figure.
\begin{figure}[htbp]
\centering
\includegraphics[width=9cm]{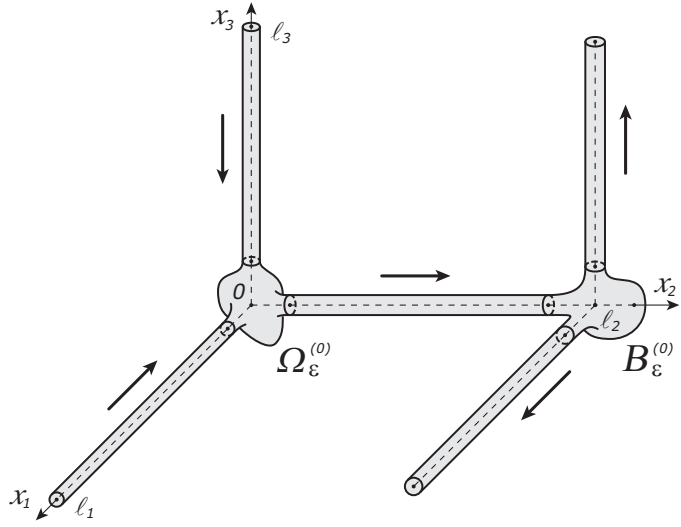}
\caption{A network with two nodes}\label{f4}
\end{figure}
Thus, we have two input thin cylinders $\Omega^{(1)}_\varepsilon$ and $\Omega^{(3)}_\varepsilon$ (concerning the vector field),
one output cylinder $\Omega^{(2)}_\varepsilon$ from the  node $\Omega^{(0)}_\varepsilon,$  which is simultaneously the input one into
 the second node $B^{(0)}_\varepsilon;$ also we have two outgoing cylinders $B^{(4)}_\varepsilon$ and $B^{(5)}_\varepsilon$ from
  $B^{(0)}_\varepsilon.$
This thin network shrinks  into the graph
$\mathcal{I}_5 := \bigcup_{i=1}^5 I_i,$ where $I_i =  \{x\colon x_i \in  [0, \ell_i], \ \overline{x}_i = (0, 0)\},$ $i\in \{1, 2, 3\},$ and
$I_4 =  \{x\colon x_1 \in  [0, \ell_4], \ x_2 = \ell_2, \ x_3 =0\},$
$I_5 =  \{x\colon x_1 =0, \ x_2 = \ell_2, \ x_3 \in  [0, \ell_5] \},$ with vertices at the origin and the point $(0, \ell_2, 0).$

The zero approximation  in the case of $\alpha=1$ is equal to
\begin{equation}\label{zero_app+2}
\mathfrak{A}_{0,\varepsilon}=
\left\{
  \begin{array}{l}
   w_0^{(i)}(x_i,t) \hfill  \text{in} \quad   \Omega^{(i)}_{\varepsilon,3\ell_0,\gamma}, \ \ i\in\{1,3\},
\\[4pt]
N_0\left(\frac{x}{\varepsilon},t\right) \hfill  \text{in} \quad   \Omega^{(0)}_{\varepsilon, \gamma},
\\[4pt]
    \chi_{\ell_0}\big(\frac{x_i}{\varepsilon^\gamma}\big)\, w_0^{(i)}(x_i,t) +
\Big(1- \chi_{\ell_0}\big(\frac{x_i}{\varepsilon^\gamma}\big)\Big) N_0\left(\frac{x}{\varepsilon},t\right) \quad  \text{in} \quad
\Omega^{(i)}_{\varepsilon,2\ell_0,3\ell_0,\gamma },  \ \  i\in\{1,2,3\},
\\[4pt]
 w_0^{(2)}(x_2,t) \hfill  \text{in} \quad   \Omega^{(2)}_{\varepsilon} \cap \big\{x\colon x_2 \in [3\ell_0 \varepsilon^\gamma, \ell_2  - 3\ell_0 \varepsilon^\gamma]\big\},
\\[4pt]
 \chi_{\ell_0}\big(\frac{\ell_2 - x_2}{\varepsilon^\gamma}\big)\, w_0^{(2)} +
\Big(1- \chi_{\ell_0}\big(\frac{\ell_2 - x_2}{\varepsilon^\gamma}\big)\Big) \mathfrak{N}_0\left(\frac{x_1}{\varepsilon},\frac{\ell_2 -x_2}{\varepsilon},\frac{x_3}{\varepsilon},t\right)
   \\
    \hfill \text{in} \quad   \Omega^{(2)}_{\varepsilon} \cap \big\{x\colon x_2 \in [\ell_2  - 3\ell_0 \varepsilon^\gamma, \ell_2  - 2\ell_0 \varepsilon^\gamma]\big\} ,
\\[4pt]
\mathfrak{N}_0\left(\frac{x_1}{\varepsilon},\frac{\ell_2 -x_2}{\varepsilon},\frac{x_3}{\varepsilon},t\right) \hfill  \text{in} \quad  B^{(0)}_{\varepsilon, \gamma},
\\[4pt]
    \chi_{\ell_0}\big(\frac{x_1}{\varepsilon^\gamma}\big)\, z_0^{(4)}(x_1,t) +
\Big(1- \chi_{\ell_0}\big(\frac{x_1}{\varepsilon^\gamma}\big)\Big) \mathfrak{N}_0\left(\frac{x_1}{\varepsilon},\frac{\ell_2 -x_2}{\varepsilon},\frac{x_3}{\varepsilon},t\right) \hfill   \text{in} \quad
B^{(4)}_{\varepsilon,2\ell_0,3\ell_0,\gamma },
\\[4pt]
    \chi_{\ell_0}\big(\frac{x_3}{\varepsilon^\gamma}\big)\, z_0^{(5)}(x_3,t) +
\Big(1- \chi_{\ell_0}\big(\frac{x_3}{\varepsilon^\gamma}\big)\Big) \mathfrak{N}_0\left(\frac{x_1}{\varepsilon},\frac{\ell_2 -x_2}{\varepsilon},\frac{x_3}{\varepsilon},t\right) \hfill  \text{in} \quad
B^{(5)}_{\varepsilon,2\ell_0,3\ell_0,\gamma },
\\[4pt]
   z_0^{(4)}(x_1,t) + \chi_\delta^{(4)}(x_1) \, \Pi_0^{(4)}
    \left(\frac{\ell_4 - x_1}{\varepsilon}, \frac{x_2 - \ell_2}{\varepsilon}, \frac{x_3}{\varepsilon}, t\right) \hfill  \text{in}  \quad
    B^{(4)}_{\varepsilon,3\ell_0,\gamma},
\\[4pt]
   z_0^{(5)}(x_3,t) + \chi_\delta^{(5)}(x_3) \, \Pi_0^{(5)}
    \left(\frac{x_1}{\varepsilon}, \frac{x_2 - \ell_2}{\varepsilon}, \frac{\ell_5 - x_3}{\varepsilon}, t\right) \hfill  \text{in}  \quad
    B^{(5)}_{\varepsilon,3\ell_0,\gamma},
      \end{array}
\right.
\end{equation}
where  the notation for the domains are the same as in \eqref{note_doms}; the cut-off functions $\chi_{\ell_0},$ $\chi_\delta^{(4)}$ and $\chi_\delta^{(5)}$ are defined in \S~\ref{Sec:justification};

$\bullet$  $\{w_0^{(i)}\}_{i=1}^3$ and $\{z_0^{(i)}\}_{i=4}^5$ form the solution to the corresponding limit problem on the graph  $\mathcal{I}_5,$
namely, $w_0^{(1)}$ and $w_0^{(3)}$ are solutions to the problem \eqref{limit_prob_1_2}, respectively; $w_0^{(2)}$ is a solution to the problem \eqref{limit_prob_1_4} (now $m=2$ and $M=3)$; and $\{z_0^{(i)}\}_{i=4}^5$ are solutions to the problem \eqref{limit_prob_1_3}, respectively, with boundary condition $z_0^{(i)}(0,t) = w_0^{(2)}(\ell_2, t),$ $i\in \{4, 5\};$

$\bullet$  $N_0$ and $\mathfrak{N}_0$ are nodal solutions to the problem \eqref{N_0_prob} in the corresponding unbounded domain  that related to the corresponding node; $\mathfrak{N}_0$ satisfies the following conditions:
$\mathfrak{N}_0(\xi,t)  \sim z^{(4)}_{0}(0,t)$   as $\xi_1 \to +\infty,$ \
$\mathfrak{N}_0(\xi,t)  \sim w^{(2)}_{0}(\ell_2,t)$   as $\xi_2 \to +\infty,$ \
$\mathfrak{N}_0(\xi,t)  \sim z^{(5)}_{0}(0,t)$   as $\xi_3 \to +\infty;$

$\bullet$  $\Pi_0^{(4)}$ and $\Pi_0^{(5)}$ are boundary-layer solutions to the corresponding problem \eqref{prim+probl+0}.

The approximation function \eqref{zero_app+2} satisfies the same asymptotic estimate as in Corollary~\ref{corol_1}. Similarly we can construct a first-order approximation and prove the corresponding  estimate, as  in Corollary~\ref{corol_2}.

%%%%%%%%%%%%%%%%%%%%%
%%%%%%%%%%%%%%%%%%%%%

\section{Asymptotic approximation in the  case $\alpha > 1$}\label{Sec:expansions+}
 If $\alpha \ge 2,$ then the order of sources of convective and diffusion flows at the boundary is small compared to the orders of residuals, which the  approximation $\mathfrak{A}_\varepsilon$ leaves in the problem (see Lemma~\ref{Prop-3-1}). Therefore, the same approximating function can be used to approximate the solution $u_\varepsilon,$ however, functions $\{\varphi_\varepsilon^{(i)}\}_{i=0}^M$ and their derivatives must be absent in the problems that determine the coefficients of $\mathfrak{A}_\varepsilon.$
In addition, the same asymptotic estimates as in \S~\ref{A priori estimates}  hold. They can be proved without the additional assumptions for
$\{\varphi_\varepsilon^{(i)}\}_{i=0}^M$ made in Remark~\ref{add_assumptios} and the assumption \eqref{zero_2}.

When a problem has two parameters, it is necessary to change the asymptotic scale by adjusting it precisely for these parameters (see, e.g., \cite{Mel_Kle_2019,M-AA-2021}). For $\alpha \in (\frac32 , 2)$, we propose the following ansatzes:\\
$\bullet$ \   the regular ansatz
  \begin{multline}\label{regul_1}
\mathcal{U}^{(i)}_\varepsilon =
w_0^{(i)}(y^{(i)}_1,t) + \varepsilon^{\alpha-1} w_{\alpha-1}^{(i)}(y^{(i)}_1,t) +  \varepsilon
\Big(w_1^{(i)}(y^{(i)}_1,t)  + u_1^{(i)} \big(y^{(i)}_1, \tfrac{\overline{y}^{(i)}_1}{\varepsilon}, t \big)
     \Big)
     \\
     + \varepsilon^{2\alpha-2} w_{2\alpha-2}^{(i)}(y^{(i)}_1,t)
 + \varepsilon^\alpha \Big(w_\alpha^{(i)}(y^{(i)}_1,t)  + u_\alpha^{(i)} \big(y^{(i)}_1, \tfrac{\overline{y}^{(i)}_1}{\varepsilon}, t \big)\Big)
+ \varepsilon^2 u_2^{(i)} \big(y^{(i)}_1, \tfrac{\overline{y}^{(i)}_1}{\varepsilon}, t \big)
\\
+ \varepsilon^{2\alpha -1} u_{2\alpha -1}^{(i)} \big(y^{(i)}_1, \tfrac{\overline{y}^{(i)}_1}{\varepsilon}, t \big)
+ \varepsilon^{\alpha+1} u_{\alpha +1}^{(i)} \big(y^{(i)}_1, \tfrac{\overline{y}^{(i)}_1}{\varepsilon}, t \big)
\quad \text{in} \ \ \Omega^{(i)}_\varepsilon \ \ (i \in \{1,\ldots,M\});
\end{multline}
$\bullet$ \  the boundary-layer ansatz
\begin{equation}\label{prim_1}
\mathcal{B}^{(i)}_\varepsilon(y^{(i)},t)
 := \sum_{\flat \in  \beth^\setminus \{ 2 \alpha -1, \alpha +1\}}\varepsilon^{\flat} \, \Pi_\flat^{(i)}
    \left(\tfrac{\ell_i - y^{(i)}_1}{\varepsilon}, \tfrac{y^{(i)}_2}{\varepsilon}, \tfrac{y^{(i)}_3}{\varepsilon}, t\right)
 \end{equation}
in a neighborhood of the base $\Upsilon_{\varepsilon}^{(i)} (\ell_i)$ of the cylinder $\Omega^{(i)}_\varepsilon$  $(i \in \{m+1,\ldots,M\}),$
 where the index set $\beth := \{0, \alpha -1, 1, 2\alpha -2, \alpha, 2, 2 \alpha -1,  \alpha+1\};$\\
$\bullet$ \ and the node-layer one
  \begin{equation}\label{junc+1}
  \mathcal{N}_\varepsilon(x,t) :=\sum_{\flat \in  \beth^\setminus \{2, 2 \alpha -1, \alpha +1 \}}\varepsilon^\flat N_\flat\left(\tfrac{x}{\varepsilon},t\right)
\end{equation}
in a neighborhood of the node $\Omega^{(0)}_\varepsilon$.

\begin{remark}\label{Rem_4_1}
 The closer the parameter $\alpha$ is  to $1,$  the more terms will be between $\varepsilon^0$ and $\varepsilon^1$ (similar as in
\cite[\S 4]{M-AA-2021}).
So, for example, if $\alpha \in (\frac43, \frac32),$ the asymptotic scale is as follows
$$
\varepsilon^0, \ \ \varepsilon^{\alpha -1}, \ \ \varepsilon^{2\alpha -2},  \ \ \varepsilon^1, \ \ \varepsilon^{3\alpha -3},\ \ \varepsilon^{\alpha}, \ \ \varepsilon^{2\alpha -1}, \ \ \varepsilon^2, \ \ \varepsilon^{3\alpha -2}, \ \ \varepsilon^{\alpha +1}.
$$
To obtain an appropriate estimate in the case $\alpha = \frac32,$ we should consider the
regular ansatz in the form
$$
\mathcal{U}^{(i)}_\varepsilon =
w_0^{(i)}(y^{(i)}_1,t) + \varepsilon^{\frac12} w_{\frac12}^{(i)}(y^{(i)}_1,t) +  \sum_{\flat \in  \beth^\setminus \{0,\frac12\}}\varepsilon^\flat\Big(w_\flat^{(i)}(y^{(i)}_1,t)  + u_\flat^{(i)} \big(y^{(i)}_1, \tfrac{\overline{y}^{(i)}_1}{\varepsilon}, t \big)
     \Big),
$$
where the index set $\beth = \{0, \frac12, 1, \frac32, 2, \frac52\},$
$ w_2^{(i)} \equiv w_{\frac52 }^{(i)} \equiv 0.$

Therefore, in this section, we restrict ourselves to the  case $\alpha \in (\frac32 , 2)$, the results of which will show peculiarities for all
$\alpha \in (1, 2).$
\end{remark}

Substituting \eqref{regul_1} and \eqref{junc+1}  into the differential equation and boundary conditions of the problem \eqref{probl_rewrite},
collecting coefficients at the same power of $\varepsilon,$ by the same way as in \S \ref{regular_asymptotic} and \S \ref{subsec_Inner_part} we get the following problems.
\ 
The limit problem looks now as follows
\begin{equation}\label{limit_prob_>1}
 \left\{\begin{array}{rcll}
 \partial_t{w}^{(i)}_0 + \Big( v_i^{(i)}(y^{(i)}_1,t)\, w^{(i)}_0 \Big)^\prime &=& 0,&  (y^{(i)}_1, t) \in  I_i \times (0, T),
 \ \ i \in \{1,\ldots,M\},
 \\[1mm]
 \sum_{i=1}^{M}  h_i^2 \,\mathrm{v}_i(t)  \,  w_0^{(i)}(0, t) &=& 0, & t \in (0, T),
\\[1mm]
w_0^{(i)}(\ell_i,  t) & = & q_i(t), & t \in [0, T],  \ \  i \in \{1,\ldots,m\},
\\[1mm]
    w_0^{(i)}(y^{(i)}_1,  0) & = & 0, & y^{(i)}_1  \in [0, \ell_i],  \ \ i \in \{1,\ldots,M\}.
 \end{array}\right.
\end{equation}
The terms  $\{{w}^{(i)}_{\alpha-1}\}_{i=1}^M$ are a solution to the problem
\begin{equation}\label{limit_prob_alpha-1}
 \left\{\begin{array}{rcll}
 \partial_t{w}^{(i)}_{\alpha-1} + \Big( v_i^{(i)}(y^{(i)}_1,t)\, w^{(i)}_{\alpha-1} \Big)^\prime &=&  - \widehat{\varphi}^{(i)}\big(w^{(i)}_0, y^{(i)}_1, t\big),&  (y^{(i)}_1, t) \in  I_i \times (0, T),
 \\
 &&& i \in \{1,\ldots,M\},
 \\
 \sum_{i=1}^{M}  h_i^2 \,\mathrm{v}_i(t)  \,  w_{\alpha-1}^{(i)}(0, t) &=& 0, & t \in (0, T),
\\[1mm]
w_{\alpha-1}^{(i)}(\ell_i,  t) & = & 0, \quad t \in [0, T], &  i \in \{1,\ldots,m\},
\\[1mm]
    w_{\alpha-1}^{(i)}(y^{(i)}_1,  0) & = & 0, \quad y^{(i)}_1  \in [0, \ell_i], & i \in \{1,\ldots,M\},
 \end{array}\right.
\end{equation}
where $\widehat{\varphi}^{(i)}$ is determined in \eqref{hat_phi}. These problems are linear and due to the assumption made in Section~\ref{Sec:Statement} they have classical solutions; in addition the explicit  representations are possible  for them  (see \cite[\S 3.2.1]{Mel-Roh_preprint-2022}).

The differential equations of these problems are solvability conditions for the problems to determine  $\{u_1^{(i)}\}$ and $\{u_\alpha^{(i)}\},$ respectively. The problem for $u_1^{(i)}$ consists of the differential equation \eqref{eq_1} and the boundary condition \eqref{bc_1}, where $\varphi^{(i)}$ is absent. Thus,  $u_1^{(i)}$ is a unique solution to the Neumann problem
\begin{equation}\label{u_1}
\left\{
  \begin{array}{c}
\Delta_{\bar{\xi}_1}u^{(i)}_1\big(y^{(i)}_1, \bar{\xi}_1, t\big)
    =   w^{(i)}_0(y^{(i)}_1, t) \,   \mathrm{div}_{\bar{\xi}_1}\overline{V}^{(i)}(y^{(i)}_1, \bar{\xi}_1, t), \qquad \overline{\xi}_1\in
\Upsilon_i (y^{(i)}_1)
\\[2pt]
\partial_{\bar{\nu}_{\bar{\xi}_1}} u^{(i)}_1\big(y^{(i)}_1, \bar{\xi}_1, t\big) = w^{(i)}_{0}(y^{(i)}_1,t) \, \overline{V}^{(i)}(y^{(i)}_1, \bar{\xi}_1,t) \boldsymbol{\cdot} \bar{\nu}_{\xi_1}, \quad \bar{\xi}_1 \in \partial\Upsilon_i\big(y^{(i)}_1\big),
\\
\langle u_1^{(i)} \rangle_{\Upsilon_i\big(y^{(i)}_1\big)} =0;
  \end{array}
\right.
\end{equation}
$u_\alpha^{(i)}$ is a unique solution to the problem
\begin{equation}\label{u_alpha}
\left\{
  \begin{array}{c}
\Delta_{\bar{\xi}_1}u^{(i)}_\alpha   =   w^{(i)}_{\alpha-1}(y^{(i)}_1, t) \,   \mathrm{div}_{\bar{\xi}_1}\overline{V}^{(i)}(y^{(i)}_1, \bar{\xi}_1, t)
 - \widehat{\varphi}^{(i)}\big(w^{(i)}_0, y^{(i)}_1, t\big), \quad \overline{\xi}_1\in \Upsilon_i (y^{(i)}_1)
\\[2pt]
-\partial_{\bar{\nu}_{\bar{\xi}_1}} u^{(i)}_\alpha +  w^{(i)}_{\alpha-1} \, \overline{V}^{(i)}(y^{(i)}_1, \bar{\xi}_1,t) \boldsymbol{\cdot} \bar{\nu}_{\xi_1} = \varphi^{(i)}\big(w^{(i)}_{0}, y^{(i)}_1, \bar{\xi}_1, t\big), \quad \bar{\xi}_1 \in \partial\Upsilon_i\big(y^{(i)}_1\big),
\\[2pt]
\langle u_\alpha^{(i)} \rangle_{\Upsilon_i\big(y^{(i)}_1\big)} =0.
  \end{array}
\right.
\end{equation}

The conditions at the vertex in \eqref{limit_prob_>1} and \eqref{limit_prob_alpha-1}  are solvability conditions
for problems to determine  $\widetilde{N}_0$ and $\widetilde{N}_{\alpha-1},$ respectively. The problems for $\widetilde{N}_0$ and $\widetilde{N}_{\alpha-1}$ are coincide with \eqref{tilda_N_0_prob}, but the right-hand side
$$
g_{\alpha-1}^{(i)}(\xi^{(i)}_1,t)  =   w^{(i)}_{\alpha-1}(0,t) \, \chi''_{\ell_0}(\xi^{(i)}_1) - \mathrm{v}_i(t)  \, w^{(i)}_{\alpha-1}(0,t) \, \chi'_{\ell_0}(\xi^{(i)}_1).
$$
Thus, the coefficients ${N}_\flat $, $\flat \in \{0, \alpha -1\},$
have  the following asymptotics uniform with respect to $t\in [0, T]$:
\begin{equation}\label{rem_exp-decrease+alpha-1}
N_\flat(\xi,t) = w^{(i)}_{\flat}(0,t)  +  \mathcal{ O}(\exp(-\beta_0\xi_i))
\quad \mbox{as} \ \ \xi_i\to+\infty,  \ \  \xi  \in \Xi^{(i)},    \ \ i=\{1,\ldots,M\}.
\end{equation}

The problem for  the terms $\{{w}^{(i)}_{1}\}_{i=1}^M$ is  similar to \eqref{prob_w_1}, but  the right-hand sides $\{f_i\}_{i=1}^M,$ the coefficient~ ${\bf d}_1$ (see \eqref{d_1}) and $\{a_i\}_{i=1}^M$ (see \eqref{a_i}) don't contain summands with the functions $\{\varphi^{(i)}\}_{i=0}^M$ and their derivatives. The Kirchhoff  condition for $\{{w}^{(i)}_{1}\}_{i=1}^M$  is the solvability condition for the problem for determining the term $\widetilde{N}_1.$ This problem coincides with  \eqref{tilda_N_1_prob}, but
$\partial_{\boldsymbol{\nu}_\xi}  \widetilde{N}_1 = 0$ on the node boundary $\Gamma_0.$
Thus, ${N}_1$ has the asymptotics \eqref{rem_exp-decrease+1}.
 The differential equations for $\{{w}^{(i)}_{1}\}_{i=1}^M$ are the solvability conditions for the problems to determine $\{u^{(i)}_{2}\}_{i=1}^M,$ respectively.
Now, the problem for $u^{(i)}_{2}$ looks like this
 \begin{equation}\label{u_2}
\left\{
  \begin{array}{c}
\Delta_{\bar{\xi}_1}u^{(i)}_2
    =     \mathrm{div}_{\bar{\xi}_i} \Big(\big[ w^{(i)}_{1} + u^{(i)}_{1}\big]\,  \overline{V}^{(i)} \Big) + \Big(  v^{(i)}_1 \, u^{(i)}_{1} \Big)^\prime
   +   \partial_t {u}_1^{(i)}, \ \ \overline{\xi}_1\in
\Upsilon_i (y^{(i)}_1)
\\[2pt]
\partial_{\bar{\nu}_{\bar{\xi}_1}} u^{(i)}_2 = \big[ w^{(i)}_{1} + u^{(i)}_{1}\big]\,  \overline{V}^{(i)} \boldsymbol{\cdot} \bar{\nu}_{\xi_1}, \quad \bar{\xi}_1 \in \partial\Upsilon_i\big(y^{(i)}_1\big),
\quad \langle u_2^{(i)} \rangle_{\Upsilon_i\big(y^{(i)}_1\big)} =0.
  \end{array}
\right.
\end{equation}

For $\{w^{(i)}_{2\alpha-2}\}_{i=1}^M$ we get the problem
\begin{equation}\label{limit_prob_2alpha-2}
 \left\{\begin{array}{rcl}
 \partial_t{w}^{(i)}_{2\alpha-2} + \Big( v_i^{(i)}(y^{(i)}_1,t)\, w^{(i)}_{2\alpha-2} \Big)^\prime &=& - \partial_s\widehat{\varphi}^{(i)}\big(w^{(i)}_0, y^{(i)}_1, t\big)\, w^{(i)}_{\alpha-1} , \quad (y^{(i)}_1, t) \in  I_i \times (0, T),
 \\ &&   \hfill    i \in \{1,\ldots,M\},
 \\
 \sum_{i=1}^{M}  h_i^2 \,\mathrm{v}_i(t)  \,  w_{2\alpha-2}^{(i)}(0, t) &=& 0, \quad t \in (0, T),
\\[2mm]
w_{\alpha-1}^{(i)}(\ell_i,  t) & = & 0, \quad  t \in [0, T],  \ \  i \in \{1,\ldots,m\},
\\[2mm]
    w_{\alpha-1}^{(i)}(y^{(i)}_1,  0) & = & 0, \quad  y^{(i)}_1  \in [0, \ell_i],  \ \ i \in \{1,\ldots,M\},
 \end{array}\right.
\end{equation}
where $\widehat{\varphi}^{(i)}$ is defined in \eqref{hat_phi}. The gluing condition at the vertex  is the solvability condition for the problem
to determine $\widetilde{N}_{2\alpha -2}.$ It is coincide with the problem \eqref{tilda_N_0_prob}, but the right-hand side
$$
g_{2\alpha-2}^{(i)}(\xi^{(i)}_1,t)  =   w^{(i)}_{2\alpha-2}(0,t) \, \chi''_{\ell_0}(\xi^{(i)}_1) - \mathrm{v}_i(t)  \, w^{(i)}_{2\alpha-2}(0,t) \, \chi'_{\ell_0}(\xi^{(i)}_1).
$$
Thus, the coefficient ${N}_{2\alpha-2}$ has the asymptotics \eqref{rem_exp-decrease+alpha-1} for $\flat = 2 \alpha -2.$
The differential equations of \eqref{limit_prob_2alpha-2} are solvability conditions for problems to determine $\{u_{2\alpha-1}^{(i)}\}$:
\begin{equation}\label{u_2alpha-1}
\left\{
  \begin{array}{c}
\Delta_{\bar{\xi}_1}u^{(i)}_{2\alpha-1}
    =   w^{(i)}_{2\alpha-2} \,   \mathrm{div}_{\bar{\xi}_1}\overline{V}^{(i)}(y^{(i)}_1, \bar{\xi}_1, t)
 - \partial_s\widehat{\varphi}^{(i)}\big(w^{(i)}_0, y^{(i)}_1, t\big)\, w^{(i)}_{\alpha-1}, \quad  \overline{\xi}_1\in
\Upsilon_i (y^{(i)}_1)
\\[2pt]
-\partial_{\bar{\nu}_{\bar{\xi}_1}} u^{(i)}_{2\alpha-1} +  w^{(i)}_{2\alpha-2} \, \overline{V}^{(i)} \boldsymbol{\cdot} \bar{\nu}_{\xi_1} = \partial_s\varphi^{(i)}\big(w^{(i)}_{0}, y^{(i)}_1, \bar{\xi}_1, t\big)\, w^{(i)}_{\alpha-1}, \quad \bar{\xi}_1 \in \partial\Upsilon_i\big(y^{(i)}_1\big)
\\[2pt]
\langle u_{2\alpha-1}^{(i)} \rangle_{\Upsilon_i\big(y^{(i)}_1\big)} =0.
  \end{array}
\right.
\end{equation}

The coefficients  $\{w^{(i)}_{\alpha}\}_{i=1}^M$ are solution to the problem
\begin{equation}\label{prob_w_alpha}
 \left\{\begin{array}{rcll}
 \partial_t{w}^{(i)}_{\alpha}  + \Big( v_i^{(i)}(y^{(i)}_1,t)\, w^{(i)}_{\alpha}\Big)^\prime  & =&  f^{(i)}_\alpha  & \text{in} \  I_i \times (0, T),  \ \ i \in \{1,\ldots,M\},
 \\[3mm]
\sum_{i=1}^{M} h_i^2 \,  \mathrm{v}_i \,   w_\alpha^{(i)}(0, t) & = & {\bf d}_\alpha(t), & t \in (0, T),
 \\[2mm]
    w_\alpha^{(i)}(\ell_i,  t) & = & 0, & t \in [0, T],   \ \ i \in \{1,\ldots,m\},
\\[2mm]
    w_\alpha^{(i)}(x_i,  0) & = & 0, & x_i  \in [0, \ell_i],  \ \ i \in \{1,\ldots,M\},
  \end{array}\right.
\end{equation}
where
\begin{align*}%\label{f_alpha}
  f^{(i)}_\alpha(y^{(i)}_1,t) =  &  \ \Big( w^{(i)}_{\alpha-1}(y^{(i)}_1,t)  \Big)^{\prime\prime}  -
w^{(i)}_{1}(y^{(i)}_1, t) \, \partial_s \widehat{\varphi}^{(i)}\big(w^{(i)}_0(y^{(i)}_1,t), y^{(i)}_1, t\big)
\\
- &  \ \frac{1}{\pi h_i^2} \int_{\partial \Upsilon_i\big(y^{(i)}_1\big)} \partial_s\varphi^{(i)}\big(w^{(i)}_{0}(y^{(i)}_1,t), y^{(i)}_1, \bar{\xi}_1, t\big)
\, u^{(i)}_{1}(y^{(i)}_1,\bar{\xi}_1, t) \,  d\sigma_{\bar{\xi}_1},  \notag
\end{align*}
\begin{align}\label{d_alpha}
 {\bf d}_\alpha(t) :=&  - \frac{1}{\pi}\int\limits_{\Gamma_0} \varphi^{(0)}\big(N_0, \xi,t\big)\, d\sigma_\xi - \frac{1}{\pi}\int\limits_{\Xi^{(0)}} \partial_t {N}_{\alpha-1}(\xi,t)\, d\xi - \frac{1}{\pi} \sum_{i=1}^{M}\int\limits_{\Xi^{(i)}}  \partial_t \widetilde{N}_{\alpha-1}(\xi,t) \, d\xi
\notag
\\
 & + \sum_{i=1}^{M}h^2_i \, \partial_{y_1^{(i)} }w_{\alpha-1}^{(i)}(0,t)  \Big(1 - \mathrm{v}_i(t)  \,  \int_{2\ell_0}^{3\ell_0} \xi^{(i)}_1\,  \chi^\prime_{\ell_0}(\xi^{(i)}_1) \, d\xi^{(i)}_1 \Big).
\end{align}

The Kirchhoff  condition in \eqref{prob_w_alpha}  is the solvability condition for the problem
\begin{equation}\label{tilda_N_alpha}
\left\{\begin{array}{rcll}
-   \Delta_\xi \widetilde{N}_\alpha +
  \overrightarrow{V}(\xi,t) \boldsymbol{\cdot} \nabla_\xi \widetilde{N}_\alpha(\xi,t) & = &  - \partial_t {{N}}_{\alpha - 1}(\xi,t), &  \xi \in\Xi^{(0)},
\\[2mm]
- \partial_{\boldsymbol{\nu}_\xi}  \widetilde{N}_\alpha(\xi,t) &=& \varphi^{(0)}\big(N_0(\xi,t), \xi,t\big) , &   \xi \in \Gamma_0,
\\[2mm]
 -   \Delta_{\xi^{(i)}} \widetilde{N}_\alpha  +
  \mathrm{v}_i(t) \, \partial_{\xi^{(i)}_1}\widetilde{N}_\alpha(\xi^{(i)},t) & = &  - \partial_t \widetilde{N}_{\alpha -1} + g_\alpha^{(i)}, &
     \xi^{(i)} \in\Xi^{(i)},
\\[2mm]
\partial_{\bar{\nu}^{(i)}} \widetilde{N}_\alpha(\xi^{(i)},t)  &=&  0, &
   \xi^{(i)} \in \Gamma_i,
\\[2mm]
\widetilde{N}_\alpha(\xi^{(i)},t) \ \rightarrow \ 0 &\text{as} & \xi^{(i)}_1 \to +\infty, &  \xi^{(i)}  \in \Xi^{(i)}, \quad i\in \{1,\ldots,M\},
 \end{array}\right.
\end{equation}
where
\begin{align*}%\label{g_alpha}
g_\alpha^{(i)}(\xi^{(i)}_1,t) = & \  w^{(i)}_\alpha(0,t) \, \chi''_{\ell_0}(\xi^{(i)}_1) + \dfrac{\partial w_{\alpha -1}^{(i)}}{\partial y^{(i)}_1}(0,t)\,   \Big( \big(\xi^{(i)}_1 \chi_{\ell_0}^{\prime}(\xi^{(i)}_1)\big)^\prime
 +  \chi_{\ell_0}^{\prime}(\xi^{(i)}_1) \Big)  \notag
\\
- & \ \mathrm{v}_i(t)  \Big( w^{(i)}_\alpha(0,t)  + \dfrac{\partial w_{\alpha-1}^{(i)}}{\partial y^{(i)}_1}(0,t) \,  \xi^{(i)}_1 \Big) \chi^\prime_{\ell_0}(\xi^{(i)}_1).
\end{align*}
Thus, ${N}_\alpha$ has the asymptotics
\begin{equation}\label{expdecrease_alpha}
N_\alpha(\xi,t) = w^{(i)}_{\alpha}(0,t)  +  \Psi^{(i)}_{\alpha}(\xi_i,t) + \mathcal{ O}(\exp(-\beta_0\xi_i))
\quad \mbox{as} \ \ \xi_i\to+\infty,  \ \  \xi  \in \Xi^{(i)} \quad (\beta_0 >0),
\end{equation}
uniform with respect to $t\in [0, T],$ where $\Psi_{\alpha}^{(i)}(\xi^{(i)}_1,t)
 =   \xi^{(i)}_1 \, \partial_{\partial y^{(i)}_1}  w_{\alpha -1}^{(i)}(0,t).$

The differential equations of \eqref{prob_w_alpha} are solvability conditions for problems to determine $\{u_{\alpha+1}^{(i)}\}_{i=1}^M$:
\begin{equation*}%\label{u_al_1}
\left\{
  \begin{array}{c}
\hskip-2.7cm  \Delta_{\bar{\xi}_1}u^{(i)}_{\alpha+1}
  =  \Big(  v^{(i)}_1(y^{(i)}_1, t) \, \big[ w^{(i)}_{\alpha} + u^{(i)}_{{\alpha}}\big] \Big)^\prime
  + \partial_t {w}_{\alpha}^{(i)}  +   \partial_t {u}_{\alpha}^{(i)}
  \\
 \hfill  + \  \mathrm{div}_{\bar{\xi}_i} \Big( \overline{V}^{(i)}(y^{(i)}_1, \bar{\xi}_1,t) \,
            \big[ w^{(i)}_{{\alpha}} + u^{(i)}_{{\alpha}} \big]\Big)  - \Big( w^{(i)}_{\alpha-1}(y^{(i)}_1,t)  \Big)^{\prime\prime},
\quad  \bar{\xi}_1 \in \Upsilon_i\big(y^{(i)}_1\big),
\\[3pt]
- \partial_{\bar{\nu}_{\bar{\xi}_1}} u^{(i)}_{\alpha+1}+ \big( w^{(i)}_{\alpha} + u^{(i)}_{\alpha}\big) \, \overline{V}^{(i)} \boldsymbol{\boldsymbol{\cdot}} \bar{\nu}_{\xi_1} =
       \partial_s\varphi^{(i)}\big(w^{(i)}_{0}, y^{(i)}_1, \bar{\xi}_1, t\big) \, \big( w^{(i)}_{1} + u^{(i)}_{1}\big), \quad \bar{\xi}_1 \in \partial \Upsilon_i\big(y^{(i)}_1\big),
\\[2pt]
\langle u_{\alpha +1}^{(i)} \rangle_{\Upsilon_i\big(y^{(i)}_1\big)} =0.
  \end{array}
\right.
\end{equation*}

Similar as in \S \ref{subsec_Bound_layer}, we find  the coefficients of the boundary-layer ansatz
\eqref{prim_1}. For each $i \in \{m+1,\ldots,M\}$ the coefficients $\Pi^{(i)}_0,$ $\Pi^{(i)}_1$ and $\Pi^{(i)}_2$  are solutions to the problems  \eqref{prim+probl+0} and \eqref{prim+probl+k}, respectively;
the coefficients $\Pi^{(i)}_{\alpha -1}$  and $\Pi^{(i)}_{2\alpha -2}$  are solutions to the problem  \eqref{prim+probl+0} with
the boundary condition $\Pi^{(i)}_{\alpha -1}\big|_{\Upsilon^{(i)}} = - w_{\alpha -1}^{(i)}(\ell_i,t) $  and
$\Pi^{(i)}_{2\alpha -2}\big|_{\Upsilon^{(i)}} = - w_{2\alpha -2}^{(i)}(\ell_i,t)$, respectively; and
$\Pi^{(i)}_{\alpha}$  is a solution to the problem  \eqref{prim+probl+k} with the right-hand side $\partial_{t}\Pi_{\alpha-1}^{(i)}$
and boundary condition $\Pi^{(i)}_{\alpha}\big|_{\Upsilon^{(i)}} = - w_{\alpha}^{(i)}(\ell_i,t).$

Thus, all coefficients in \eqref{regul_1}--\eqref{junc+1}  are determined.
\begin{remark}\label{Rem_4_2}
When calculating discrepancies, the terms
$\big(u_2^{(i)}\big)^{\prime\prime},$  $\big(u_{2\alpha -1}^{(i)}\big)^{\prime\prime}$
and $\big(u_{\alpha +1}^{(i)}\big)^{\prime\prime}$ appear. This means that $\{w_\flat^{(i)}\}$ must have additional smoothness, e.g.,
$\{w^{(i)}_{0}\}$ and $\{w^{(i)}_{\alpha-1}\}$ must have the $C^4$-smoothness.
Since  all problems for determining the coefficients are linear, the following additional assumption are needed:
\ the boundedness of the derivatives of $\{\varphi^{(i)}\}_{i=1}^M$ up to  the third order in the variables $s$
 and $y^{(i)}_1$ from $X_i$, ${v}^{(i)}_1 \in C^{4}([0, \ell_i]),$ $\partial_{\xi_2}\overline{V}^{(i)}\in C^{2}([0, \ell_i]),$
$\partial_{\xi_3}\overline{V}^{(i)}\in C^{2}([0, \ell_i])$ for $i\in \{1,\ldots,M\},$
 \eqref{match_conditions+} and \eqref{add_phi_0}.
\end{remark}

Now, using the ansatzes \eqref{regul_1}--\eqref{junc+1}, we construct  the approximation function
$\mathfrak{A}^\alpha_\varepsilon$ according to the formula \eqref{first_app} and calculate residuals that $\mathfrak{A}^\alpha_\varepsilon$ leaves in the problem \eqref{probl_rewrite}. Similar as \S~\ref{Sec:justification}, we get the following statement
on a difference equation.

\begin{lemma}\label{Prop-4-1} There is a number $\varepsilon_0>0$ such that for each $\varepsilon\in (0, \varepsilon_0)$  the difference between $\mathfrak{A}^\alpha_\varepsilon$ and the solution to the problem \eqref{probl_rewrite} satisfies the following relations for all $t\in (0,T):$
\begin{gather*}%\label{dif_1}
  \partial_t(\mathfrak{A}^\alpha_\varepsilon - u_\varepsilon) -  \varepsilon\, \Delta_x \big(\mathfrak{A}^\alpha_\varepsilon - u_\varepsilon\big) +
   \overrightarrow{V_\varepsilon}^{(0)} \cdot \nabla_x(\mathfrak{A}^\alpha_\varepsilon - u_\varepsilon)  =  {\varepsilon}\,  \mathcal{R}^{(0)}_\varepsilon \quad  \text{in} \ \ \Omega_\varepsilon^{(0)},
\\
 \partial_t(\mathfrak{A}^\alpha_\varepsilon - u_\varepsilon) -  \varepsilon\, \Delta_{y^{(i)}} \big(\mathfrak{A}^\alpha_\varepsilon - u_\varepsilon\big) +
   \mathrm{v}_i \,\partial_{y^{(i)}}(\mathfrak{A}^\alpha_\varepsilon - u_\varepsilon)  =  {\varepsilon^\gamma}\,  \mathcal{R}^{(i)}_{\varepsilon,\ell_0}  \quad
    \text{in} \ \ \Omega_\varepsilon^{(i)}\setminus \Omega^{(i)}_{\varepsilon,3\ell_0,\gamma},
\\
  \partial_t(\mathfrak{A}^\alpha_\varepsilon - u_\varepsilon)   -  \varepsilon\, \Delta_{y^{(i)}} (\mathfrak{A}^\alpha_\varepsilon - u_\varepsilon)\big) +
  \mathrm{div}_{y^{(i)}} \big( \overrightarrow{V_\varepsilon}^{(i)} \, (\mathfrak{A}^\alpha_\varepsilon - u_\varepsilon)\big)
   =  {\varepsilon^{2\alpha-2}} \, \mathcal{R}^{(i)}_\varepsilon \ \  \text{in} \ \ \Omega^{(i)}_{\varepsilon,3\ell_0,\gamma},
\\
 -   \partial_{\boldsymbol{\nu}_\varepsilon}\big(\mathfrak{A}^\alpha_\varepsilon - u_\varepsilon\big)  -   \varepsilon^{\alpha -1} \varphi^{(0)}_\varepsilon\big(\mathfrak{A}^\alpha_\varepsilon,x,t)  + \varepsilon^{\alpha -1} \varphi^{(0)}_\varepsilon\big(u_\varepsilon,x,t) =  {\varepsilon^{2\alpha-2}}\, \Phi^{(0)}_\varepsilon  \quad
    \text{on} \ \ \Gamma_\varepsilon^{(0)},
\\
  -   \partial_{\boldsymbol{\nu}_\varepsilon}\big(\mathfrak{A}^\alpha_\varepsilon - u_\varepsilon\big)  =  0  \quad
    \text{on} \ \ \Gamma_\varepsilon^{(i)}, \ \ y^{(i)}_1\in [\ell_0 \varepsilon, 3\ell_0 \varepsilon^\gamma]
\\
 -    \partial_{\overline{\nu}_\varepsilon}(\mathfrak{A}^\alpha_\varepsilon - u_\varepsilon) +  (\mathfrak{A}^\alpha_\varepsilon - u_\varepsilon) \, \overline{V}^{(i)}_\varepsilon\boldsymbol{\cdot}\overline{\nu}_\varepsilon\,
 -  \,  \varepsilon^{\alpha -1} \varphi^{(i)}_\varepsilon\big(\mathfrak{A}^\alpha_\varepsilon, y^{(i)},t)  \, + \,\varepsilon^{\alpha -1} \varphi^{(i)}_\varepsilon\big(u_\varepsilon, y^{(i)},t)
\\
  =  {\varepsilon^{\alpha}} \Phi_\varepsilon^{(i)} \ \     \text{on} \   \Gamma_\varepsilon^{(i)}, \ \ y^{(i)}_1\in [3\ell_0 \varepsilon^\gamma, \ell_i]
\\
 (\mathfrak{A}^\alpha_\varepsilon - u_\varepsilon)\big|_{x_i= \ell_i}
 = 0 \ \  \text{on} \  \ \Upsilon_{\varepsilon}^{(i)} (\ell_i),
\  i\in\{1,\ldots,M\}, \qquad  (\mathfrak{A}^\alpha_\varepsilon - u_\varepsilon)\big|_{t=0}
  = 0 \ \ \text{on} \ \ \Omega_{\varepsilon},
\end{gather*}
where  the vector-function $\overline{V}^{(i)}_\varepsilon$ is defined in \eqref{V_i}, $\gamma$ is a fixed number from $(\frac23, 1),$
\begin{gather*}%\label{Res_10}
  \sup_{(\Omega^{(i)} _{\varepsilon} \setminus \Omega^{(i)}_{\varepsilon,3\ell_0,\gamma}) \times (0,T)} |\mathcal{R}^{(i)}_{\varepsilon,\ell_0}| +
  \sup_{\Omega^{(i)} _{\varepsilon}\times (0,T)} |\mathcal{R}^{(i)}_\varepsilon| \le C_i, \quad i\in\{1,\ldots,M\},
\\
\sup_{\Omega^{(0)} _{\varepsilon}\times (0,T)} |\mathcal{R}^{(0)}_\varepsilon| \le C_0, \quad   \sup_{\Gamma^{(i)}_{\varepsilon} \times (0,T)} |\Phi_\varepsilon^{(i)}| \le \tilde{C}_i, \quad i\in\{0,\ldots,M\},
\end{gather*}
and the support of $\{\Phi_\varepsilon^{(i)}\}_{i=1}^M$  with respect to $y^{(i)}_1$  lies in $(\varepsilon \ell_0, \ell_i)$ uniformly in $t\in [0, T]$.
\end{lemma}

Then, repeating the proof of Theorem ~\ref{Th_1} and taking into account that the smallest order of residuals in the relations of Lemma~\ref{Prop-4-1} is equal to $\varepsilon^{{\gamma}} $ $\big(\alpha \in (\frac32 , 2)\big),$ we obtain the assertion.

\begin{theorem}\label{Th_2} Let assumptions made in Section~\ref{Sec:Statement}, in Remark~\ref{Rem_4_2} and in \eqref{add_phi_0} are satisfied. Then
  \begin{equation}\label{max_2}
  \max_{\overline{\Omega_\varepsilon}\times [0, T]} |\mathfrak{A}^\alpha_\varepsilon - u_\varepsilon|  \le C_0(T) \, \varepsilon^{{\gamma}}
\end{equation}
holds for $\varepsilon$ small enough, where  $\gamma$ is a fixed number from the interval $(\frac23, 1).$
\end{theorem}

We can regard that $\gamma > \alpha -1.$ Therefore,  from \eqref{max_2} it follows  that terms of the order $\mathcal{O}(\varepsilon)$ are redundant in the approximation $\mathfrak{A}^\alpha_\varepsilon$, and the asymptotic estimate of this accuracy can be  obtained for the approximation
\begin{equation*}%\label{zero_alpha_ app}
\mathfrak{A}_{0,\alpha, \varepsilon} =
\left\{
  \begin{array}{l}
   w_0^{(i)}(y^{(i)}_1,t) + \varepsilon^{\alpha -1} w_{\alpha -1}^{(i)}(y^{(i)}_1,t) \hfill  \text{in} \  \ \Omega^{(i)}_{\varepsilon,3\ell_0,\gamma}, \ \ i\in\{1,\ldots,m\},
\\[4pt]
   w_0^{(i)}(y^{(i)}_1,t) + \varepsilon^{\alpha -1} w_{\alpha -1}^{(i)}(y^{(i)}_1,t)
   + \chi_\delta^{(i)}(y^{(i)}_1) \Big(\Pi_0^{(i)}
    \big(\frac{\ell_i - y^{(i)}_1}{\varepsilon}, \frac{y^{(i)}_2}{\varepsilon}, \frac{y^{(i)}_3}{\varepsilon}, t\big)
    + \varepsilon^{\alpha -1} \Pi_{\alpha -1}^{(i)}\Big)
     \\
      \hfill  \text{in}  \ \ 
    \Omega^{(i)}_{\varepsilon,3\ell_0,\gamma},
    \ \   i\in\{m+1,\ldots,M\},
    \\[6pt]
    N_0\left(\frac{x}{\varepsilon},t\right) + \varepsilon^{\alpha -1} N_{\alpha -1}\left(\frac{x}{\varepsilon},t\right) \hfill  \text{in} \   \ \Omega^{(0)}_{\varepsilon, \gamma},
\\[6pt]
\chi_{\ell_0}\big(\frac{y^{(i)}_1}{\varepsilon^\gamma}\big) \Big( w_0^{(i)}(y^{(i)}_1,t) + \varepsilon^{\alpha -1} w_{\alpha -1}^{(i)}(y^{(i)}_1,t)\Big)
\\[2pt]
+\Big(1- \chi_{\ell_0}\big(\frac{y^{(i)}_1}{\varepsilon^\gamma}\big)\Big)
\Big(N_0\big(\frac{\mathbb{A}_i^{-1}y^{(i)}}{\varepsilon},t\big) + \varepsilon^{\alpha -1} N_{\alpha -1}\big(\frac{\mathbb{A}_i^{-1}y^{(i)}}{\varepsilon},t\big) \Big) \quad \hfill \text{in} \ \

\Omega^{(i)}_{\varepsilon,2\ell_0,3\ell_0,\gamma },
\\
\hfill  i\in\{1,\ldots,M\},
  \end{array}
\right.
\end{equation*}
directly from \eqref{max_2}.  As a result, we have the statement.

\begin{corollary}
  There exist constants $C_0>0$ and $\varepsilon_0>0$ such that for all $\varepsilon\in(0, \varepsilon_0)$ the  estimate
 \begin{equation}\label{main_2}
   \max_{\overline{\Omega_\varepsilon}\times [0, T]} |\mathfrak{A}_{0,\alpha, \varepsilon} - u_\varepsilon |  \le C_0 \, \varepsilon^{{\gamma}}
 \end{equation}
holds, where $\gamma$ is a fixed number from the interval $(\frac23, 1)$ such that $\gamma > \alpha -1.$
\end{corollary}

Taking into account the order of the residuals in the relations of Lemma~\ref{Prop-4-1} and repeating the proof of Lemma~\ref{apriory_estimate}, we deduce the lemma.

\begin{lemma}\label{apriory_estimate+alpha} Let assumptions of Theorem \ref{Th_2} and the inequalities \eqref{zero_2} are satisfied. Then  the estimate
\begin{equation}\label{app_estimate+}
\tfrac{1}{\sqrt{|\Omega_\varepsilon|}}\,\| {\nabla_x \mathfrak{A}^\alpha_\varepsilon - \nabla_x u_\varepsilon}\|_{L^2(\Omega_\varepsilon\times (0, T))} \le
 C_2\, \varepsilon^{\frac{\alpha+ \gamma }{2} - 1}
  \end{equation}
holds,  where  $|\Omega_\varepsilon|$ is  the Lebesque measure of $\Omega_\varepsilon,$
$\gamma$ is a fixed number from  $(\frac23, 1).$
\end{lemma}

\begin{remark}\label{Rem_4_3} Since $\alpha \in (\frac32 , 2)$ and $\gamma \in (\frac23, 1),$  the number
$\frac{\alpha+ \gamma }{2} - 1$ is positive.
For other values of the parameter $\alpha \in (1, \frac32)$, the ansatz approximation will contain more terms, but  the first two terms will always remain the same (see Remark~\ref{Rem_4_1}), and in each case the residuals must be recalculated.
\end{remark}

It follows from \eqref{app_estimate+} that terms of the order $\mathcal{O}(\varepsilon^2)$  are redundant in the approximation $\mathfrak{A}^\alpha_\varepsilon$. Denote by $\mathfrak{A}^\alpha_{\sharp,\varepsilon}$ the approximation $\mathfrak{A}^\alpha_\varepsilon$  without the terms $\varepsilon^\flat \{u_\flat^{(i)}\}_{i=1}^M,$ $\flat \in  \{ 2, 2 \alpha -1, \alpha +1\}$ and  $\varepsilon^2 \{\Pi_2^{(i)}\}_{i=1}^M.$
\begin{corollary}
  There exist constants $C_3>0$ and $\varepsilon_0>0$ such that for all $\varepsilon\in(0, \varepsilon_0)$ we have \begin{equation}\label{app_estimate_2}
\tfrac{1}{\sqrt{|\Omega_\varepsilon|}}\,\| \nabla_x \mathfrak{A}^\alpha_{\sharp,\varepsilon} - \nabla_x u_\varepsilon\|_{L^2(\Omega_\varepsilon\times (0, T))} \le
 C_3\, \varepsilon^{\frac{\alpha+ \gamma }{2} - 1}.
  \end{equation}
\end{corollary}

\section{Conclusions}\label{conclusions}

 As we noted in the introduction, there is a big problem for conservation laws on networks
how to choose gluing conditions at the vertex of a graph when there are $n$ incoming flows and $m$ outgoing ones.
Other conditions may be added to the gluing conditions, depending on the problem being modeled (they are determined on the basis of physical, biological, engineering  or other considerations).
Since our work offers a general mathematical approach, in our opinion, the choice of the appropriate component of the weighted incoming concentration average in the boundary condition of the problem \eqref{limit_prob_1_5} is optimal from both a mathematical and a physical point of view. In addition, such a choice leads to the fulfillment of the generally recognized condition of the mass conservation at the vertex (in our case the condition \eqref{cong_cond}) .

\smallskip

$\bullet$ Many physical processes, especially in chemistry and medicine, have monotonous nature.
For instance, the function
$$
\phi(s) = \frac{\lambda \, s^p}{1+ \mu \, s^p} \quad (\lambda, \ \mu > 0),
$$
where $p$ is an odd number, satisfies the monotonic condition \eqref{zero_2}. If $p=1,$ then this function corresponds to the Michaelis-Menten hypothesis in biochemical reactions and to the Langmuir kinetics adsorption models \cite{Conca,Mel_IJB-2019,Pao}. Therefore, the monotonic conditions \eqref{zero_2} for the non-linear terms $\{\varphi^{(i)}\}_{i=1}^M$, which are necessary to prove the energy estimates \eqref{app_estimate}, \eqref{app_estimate_1}  and \eqref{app_estimate_2}, are not a strong assumption.

\smallskip

$\bullet$ The results obtained  show  that the asymptotic behavior of the solution essentially depends on the parameter $\alpha$
characterizing the intensity of processes at the boundary of a network. If $\alpha \ge 2,$ then we can ignore the right-hand sides
$\{\varepsilon^{\alpha-1} \varphi^{(i)}\}_{i=0}^M$ in the boundary conditions of the problem~\eqref{probl_rewrite}.

When $\alpha \in (1, 2),$ we see its influence in the second terms of the asymptotics (see \eqref{regul_1} - \eqref{junc+1}),
and  $\{\varepsilon^{\alpha-1} \varphi^{(i)}\}_{i=1}^M$ are transformed into the the right-hand sides of the problem \eqref{limit_prob_alpha-1}. The influence of the node in the asymptotics is observed through the nodal solutions, and the impact of the function $\varepsilon^{\alpha-1} \varphi^{(0)}$ is sighted only in the gluing condition of the problem \eqref{prob_w_alpha} for
$\{w^{(i)}_{\alpha}\}_{i=1}^M$ (see \eqref{d_alpha}) and in the boundary condition on $\Gamma_0$ of the problem \eqref{tilda_N_alpha}
for $N_\alpha.$ The limit problem \eqref{limit_prob_>1} on the graph, the solution of which is the main term of the asymptotics, does not depend on $\{\varphi^{(i)}\}_{i=1}^M.$
What is remarkable is the linearity of all problems, with the help of which the terms of asymptotics are determined.

If $\alpha =1,$ then limit problem \eqref{limit_prob_1}-\eqref{limit_prob_2} on the graph is semilinear,
 and in it we already see the effect of the functions $\{\varphi^{(i)}\}_{i=1}^M.$ The impact of $\varphi^{(0)}$ is observed  in the gluing condition of the problem \eqref{prob_w_1} for $\{w^{(i)}_{1}\}_{i=1}^M$  and in the boundary condition on $\Gamma_0$ of the problem \eqref{tilda_N_1_prob}.

\smallskip

$\bullet$ The impact of physical processes on the lateral boundary of the network becomes dominant in the case of $\alpha \in (0, 1),$ and it is natural to expect that they can provoke cardinal changes in the entire process in the network.
Indeed, since the functions $\{\varphi^{(i)}\}_{i=1}^M$ are bounded, the asymptotics for the solution $u_\varepsilon$  must contains terms $\varepsilon^{\alpha-1} w_{\alpha-1}^{(i)}(y^{(i)}_1,t)$ and $\varepsilon^\alpha  u_\alpha^{(i)} $ so that have the same order as the right-hand side in the boundary condition on $\Gamma^{(i)}_\varepsilon$ of the problem~\eqref{probl_rewrite}. But $\varepsilon^{\alpha-1} w_{\alpha-1}^{(i)}$
is unbounded  as $\varepsilon \to 0.$ Thus, we must know the asymptotics of  $\varphi^{(i)}$ at infinity and made additional assumptions for
$\varphi^{(i)}.$ The study in this case cannot be treated by any simple modifications of the approaches used in the previous cases, and it is postponed to a planned forthcoming paper by the authors.

\appendix %\section{Appendix}
\section{Elliptic boundary value problems\label{appendixellip}}
To state existence results for problems like \eqref{tilda_N_0_prob} and \eqref{tilda_N_1_prob}, we summarize results which have been  adapted for our case  in \cite{Mel-Klev_AA-2022} from the
general approach proposed in \cite[\S 5]{Koz-Maz-Ros_97}, \cite[\S 2.2]{Ole_book_1996} and \cite[\S 3]{Naz99}.
Let us define for $\beta >0 $ the weighted Sobolev space $\mathcal{H}_\beta$ as the set of all functions from $H^1(\Xi)$ with  finite norm
$$
\|u\|_{\beta} := \bigg(\int_{\Xi^{(0)}} \Big(|\nabla_\xi u|^2 + |u|^2 \Big) d\xi + \sum_{i=1}^{M} \int_{\Xi^{(i)}} \varrho(\xi^{(i)})\Big(|\nabla_{\xi^{(i)}} u|^2 + |u|^2 \Big) d\xi^{(i)} \bigg)^{1/2}.
$$
Here $\varrho$ is a smooth positive function such that
$$
\varrho = \left\{
                 \begin{array}{ll}
                   1, & \ \ \xi \in \Xi^{(0)},
\\[2pt]
                   e^{\beta \xi^{(i)}_1} , &\ \  \xi^{(i)}_1 \ge 2 \ell_0, \ \ \xi^{(i)} \in \Xi^{(i)}, \ i\in\{1,\ldots,M\}.
                 \end{array}
               \right.
$$

Consider the problem
\begin{equation}\label{tilda_N}
\left\{\begin{array}{rclll}
-   \Delta_\xi \widetilde{N} +
  \overrightarrow{V} \boldsymbol{\cdot} \nabla_\xi \widetilde{N} & = & F^{(0)}, &  \xi \in\Xi^{(0)},&
\\[1mm]
- \partial_{\boldsymbol{\nu}_\xi}  \widetilde{N} &=& \Psi^{(0)}, &   \xi \in \Gamma_0,&
\\[1mm]
 -   \Delta_{\xi^{(i)}} \widetilde{N}  +
  \mathrm{v}_i \, \partial_{\xi^{(i)}_1}\widetilde{N} & = & F^{(i)}, &
     \xi^{(i)} \in\Xi^{(i)},&  i\in \{1,\ldots,M\},
\\[1mm]
\partial_{\bar{\nu}^{(i)}} \widetilde{N}  &=&  0, &
   \xi^{(i)} \in \Gamma_i, & i\in \{1,\ldots,M\}.
 \end{array}\right.
\end{equation}
Due to \eqref{lap_1} -- \eqref{grad_1} and the assumptions made in Section~\ref{subsec_V}, we can define a weak solution to \eqref{tilda_N} as follows. A function $\widetilde{N}\in \mathcal{H}_\beta$ is a weak solution to the problem~\eqref{tilda_N} if the identity
$$
\int_{\Xi^{(0)}} \Big(\nabla_\xi \widetilde{N}  - \widetilde{N} \,  \overrightarrow{V}\Big) \boldsymbol{\cdot} \nabla_\xi \phi \, d\xi
+ \sum_{i=1}^{M}\int_{\Xi^{(i)}}\Big(\nabla_{\xi^{(i)}} \widetilde{N} \boldsymbol{\cdot} \nabla_{\xi^{(i)}} \phi -  \mathrm{v}_i  \, \widetilde{N} \, \partial_{\xi^{(i)}_1} \phi \Big)  d\xi^{(i)}
$$
$$
= \int_{\Xi^{(0)}} F^{(0)}  \phi\, d\xi + \sum_{i=1}^{M}\int_{\Xi^{(i)}} F^{(i)}  \phi\, d\xi^{(i)} - \int_{\Gamma_0} \Psi^{(0)}  \phi\, d\sigma_\xi
$$
holds for any function $\phi \in \mathcal{H}_{-\beta}.$
Based on results of \cite[Lemma 3.1]{Mel-Klev_AA-2022} we have the following  statement.
\begin{proposition}\label{Prop-2-1} Let $F^{(0)} \in L^2(\Xi^{(0)}),$ $\Psi^{(0)} \in L^2(\Gamma_0),$ and for all $i\in \{1,\ldots,M\}$
$$
  \int_{\Xi^{(i)}} e^{\beta \xi^{(i)}_1} \, \big(F^{(i)}(\xi^{(i)})\big)^2 \,  d\xi^{(i)} < +\infty \qquad (\beta > 0).
$$

Then the problem \eqref{tilda_N} has a unique  weak solution in the space $\mathcal{H}_\beta$ if and only if
the following equality is satisfied
\begin{equation}\label{cong_cond_g}
 \int_{\Xi^{(0)}} F^{(0)} \, d\xi + \sum_{i=1}^{M}\int_{\Xi^{(i)}} F^{(i)}(\xi^{(i)}) \,  d\xi^{(i)} = \int_{\Gamma_0} \Psi^{(0)}(\xi) \,  d\sigma_\xi .
\end{equation}
\end{proposition}

\section{Initial-boundary value problems for advection equations\label{appendixhyp}}
We summarize  basic wellposedness  results for the solution of
initial-boundary value problems for advection equations as needed to solve the hyperbolic limit problems.
Let us  consider as model mixed problem
\begin{equation}\label{prob_1}
 \left\{\begin{array}{rcll}
 \partial_t w + v(x,t)\, \partial_{x} w  &=& - \partial_{x}v(x,t)\, w(x,t) + \Psi(w,x,t),& (x, t) \in (0, \ell) \times (0, T) ,
    \\[2mm]
    w(0,  t) & = & q(t), & t \in [0, T],
    \\[2mm]
    w(x, 0) & = & 0, & x \in [0, \ell],
    \end{array}\right.
\end{equation}
where the given functions $v(x,t), \ (x,t)\in [0, \ell]\times [0, T],$  $\Psi(s,x,t), \ (s,x, t) \in \Bbb R \times [0, \ell] \times [0, T],$ and $q(t), \ t\in[0,T],$  are smooth  on their domains of definition,   and in addition, $q(0)=0,$ $\frac{d q}{dt}(0)=0,$ and $v >0$ for all $(x,t)\in [0, \ell]\times [0, T].$

The corresponding characteristic system for the semi-linear differential equation in \eqref{prob_1} is as follows
\begin{equation}\label{char_system}
  \frac{dx}{dt} = v(x,t), \quad \frac{dw}{dt} = - \partial_{x}v(x,t)\, w(x,t) + \Psi(w,x,t).
\end{equation}
Let $F(x,t) = c$ be a  first integral of the differential equation $x'= v(x,t)$ such that $F(0,0) = 0.$ Since $v>0,$ all solutions of this equation
starting from the points $(x,0), \ x\in [0, \ell],$ are strictly increasing.
The graph of the solution $x = \mathcal{V}(t)$ outgoing from the origin
 divides the rectangle $(0, \ell) \times (0, T)$ into two domains, namely
$$
\mathfrak{D}_1 := \{(x,t)\colon x\in (0,\ell), \ \ t \in (0, \mathcal{V}^{-1}(x))\} \quad \text{and} \quad \mathfrak{D}_2 := \big\{(0, \ell) \times (0, T)\big\} \setminus \overline{\mathfrak{D}_1},
$$
where $\mathcal{V}^{-1}$ is the inverse function to $\mathcal{V}.$
From $F(x,t) = c$ we can get a general solution $x= X(t; c).$ Clearly, that $X(t;0) = \mathcal{V}(t)$ and $X(t;F(x,t)) \equiv x.$
In addition, the inverse function $t= X^{-1}(x; c)$ satisfies relations
$$
X^{-1}(x; 0) = \mathcal{V}^{-1}(x) \quad \text{and}\quad X^{-1}(x;F(x,t)) \equiv t.
$$

By integrating over the characteristics, the problem \eqref{prob_1} is reduced to the nonlinear Volterra integral equations
\begin{equation}\label{Volt_1}
  w(x,t) = \int\limits_{0}^{t} \exp\Big(- \int_{\tau}^{t}\partial_{x}v\big( X(\mu;F(x,t)),\mu\big) \, d\mu\Big) \, \Psi\big(w(X(\tau;F(x,t)), \tau),   X(\tau;F(x,t)), \tau\big)\, d\tau
\end{equation}
in the domain $\mathfrak{D}_1,$ and
\begin{multline}\label{Volt_2}
  w(x,t) = q\big(X^{-1}(0;F(x,t))\big) \, \exp\Big(- \int_{X^{-1}(0;F(x,t))}^{t}\partial_{x}v\big( X(\tau;F(x,t)),\tau\big) \, d\tau\Big)
\\
+\int\limits_{X^{-1}(0;F(x,t))}^{t} \exp\Big(- \int_{\tau}^{t}\partial_{x}v\big( X(\mu;F(x,t)),\mu\big) \, d\mu\Big) \,  \Psi\big(w(X(\tau;F(x,t)), \tau),   X(\tau;F(x,t)), \tau\big)\, d\tau
\end{multline}
in $\mathfrak{D}_2.$ It should be noted that from \eqref{Volt_1} it follows that $w(x,0)=0,$ from \eqref{Volt_2} we get $w(0,t)=q(t),$ and  the equations  \eqref{Volt_1} and \eqref{Volt_2} coincide  on the characteristic $x = \mathcal{V}(t).$  Therefore, we can consider \eqref{Volt_1} and \eqref{Volt_2} as one Volterra integral equation in the rectangle $[0, \ell]\times [0, T].$

Applying the results of  \cite{Myshkis_1960}, we can state that  the mixed problem \eqref{prob_1} has a unique classical solution if there is a positive constant $C$ such that
\begin{equation}\label{bound_1}
  \big|\partial_s \Psi(s,x,t)\big| \le C, \quad \big|\partial^2_{ss} \Psi(s,x,t)\big| \le C, \quad
  \big|\partial^2_{sx} \Psi(s,x,t)\big| \le C,
\end{equation}
for all $(s,x, t) \in \Bbb R \times [0, \ell] \times [0, T],$ and the matching conditions are satisfied:
\begin{equation}\label{math_1}
q(0)=0, \quad \frac{d q}{dt}(0)=0, \quad \Psi(0,0,0) =0
\end{equation}

\section*{Acknowledgments}
The first author is grateful to the Alexander von Humboldt Foundation for the possibility to carry
out this research at the University of Stuttgart. The second author thanks for  funding by the  Deutsche Forschungsgemeinschaft (DFG, German Research Foundation) – Project Number 327154368 – SFB 1313.

%\bibliographystyle{siamplain}
%\bibliography{references}

\end{document}